\documentclass[11pt]{amsart}
\usepackage{amsmath,amsthm,amssymb,amsfonts,mathrsfs,latexsym}
\usepackage{bbm}
\usepackage{graphicx}

\usepackage{enumitem} % control spacing  

\usepackage[colorlinks=true, linkcolor=blue, citecolor=green, urlcolor=cyan]{hyperref}

\usepackage{xcolor}
\definecolor{defaultColor}{HTML}{5B68FF}

\newcommand{\hh}[1]{\textcolor{defaultColor}{#1}}

\newcommand{\R}{\mathbb{R}}
\newcommand{\N}{\mathbb{N}}
\newcommand{\Prob}{\mathbb{P}}
\newcommand{\cal}{\mathcal}
\newcommand{\opn}[1]{\operatorname{#1}}

\newtheorem{theorem}{Theorem}[section]
\newtheorem{corollary}[theorem]{Corollary}
\newtheorem{lemma}[theorem]{Lemma}
\newtheorem{proposition}[theorem]{Proposition}
\theoremstyle{definition}
\newtheorem{definition}[theorem]{Definition}

\newtheorem{remark}[theorem]{Remark}
\newtheorem{conj}[theorem]{Conjecture}

\title{ Rank of sparse Bernoulli matrices } 
\author{Han Huang}
\email{hhuang@missouri.edu }

\begin{document}
\maketitle

\begin{abstract}
	Let $ A $ be an $n \times n$ random matrix 
	with i.i.d Bernoulli($p$) entries. 
	For a fixed positive integer $\beta$, suppose 
	$p$ satisfies  
	$$ 
	\frac{ \log n  }{ n } \le p \le c_\beta
	$$
	where $c_\beta \in ( 0, 1/2 )$ is a $\beta$-dependent 
	value. 
For $t \ge 0$,
\begin{align*}
	&\mathbb{P} \left\{ 
		s_{ n - \beta + 1}(A)
	\le
	t n^{-2\beta + o_n(1) }(pn)^{-7}
	\right\} \\
 = &
	t + 
	( 1 + o_n(1) )  \mathbb{P} \bigg\{ 
		\mbox{either $\beta$ rows or $\beta$ columns of $A$ equal $\vec{0}$} \bigg\}. 
\end{align*}	

\end{abstract}

\tableofcontents
\section{Introduction}

% Polished version
In the study of random discrete matrices, a central problem is to understand when such matrices are singular.  
For the canonical model of an \(n\times n\) \emph{random sign matrix} $B_n$—one whose entries are independent Rademacher variables—the recently resolved conjecture has guided much of the field:
\[
\mathbb{P}\!\bigl\{B_n\text{ is singular}\bigr\}\le \bigl(\tfrac12+o(1)\bigr)^{n}\qquad(n\to\infty).
\]

Intuitively, it says that singularity is mainly caused by a \emph{local} obstruction: two rows or two columns that coincide up to a sign.  In other words, a global failure of invertibility is governed by configurations involving only a few rows or columns.

Komlós first addressed this question in 1967, proving that \[\mathbb{P}\{B_n\text{ singular}\}\to 0\] as \(n\to\infty\)~\cite{Kom67}.  In 1995, Kahn–Komlós–Szemerédi showed exponential decay, \[\mathbb{P}\{B_n\text{ singular}\}\le e^{-c n}\] for some \(c>0\)~\cite{KKS95}.  This bound was sharpened by Tao–Vu~\cite{TV06,TV07} and Bourgain–Vu–Wood~\cite{BVW10}, and was finally proved by Tikhomirov in 2018~\cite{Tik18}.

Similar behaviour is expected for \emph{Bernoulli} matrices $A_n$, whose entries are independent \(\operatorname{Bernoulli}\) variables with $p=p_n$.  
Across the dense-to–sparse range, the following strong version of the singularity conjecture is widely believed and partially verified: 

\begin{conj}[Strong singularity conjecture for Bernoulli matrices]\label{conj:Ber}
Let \(A_n\) be an \(n\times n\) matrix whose entries are i.i.d.\ \(\operatorname{Bernoulli}(p_n)\) variables with \(p_n\in(0,\tfrac12]\).  Then
\begin{align*}
	&\mathbb{P}\{\text{$A_n$ is singular}\} \\
	&\phantom{AA}= (1+o_n(1)) \mathbb{P}\Big\{\text{a row or a column of $A_n$ equals zero} \\
  &	\phantom{AA= (1+o_n(1)) \mathbb{P}\{AAA\,}
	\text{ or two rows or columns are equal}\Big\}.
	\end{align*}
If, in addition, \(\limsup_{n\to\infty} p_n<\tfrac12\), then
\[
\mathbb{P}\{A_n\text{ is singular}\} =(1+o_n(1))\,
\mathbb{P}\!\bigl\{\text{some row or column of \(A_n\) is zero}\bigr\}.
\]
\end{conj}

Thus, as in the sign–matrix setting, singularity is predicted to be dictated by a small collection of local coincidences rather than by complicated global dependencies.

We remark that for \(p<\tfrac{\log n}{n}\) the matrix \(A_n\) contains a zero row or column with probability tending to~\(1\), so
	Conjecture~\ref{conj:Ber} is trivial in this range.

	The main body of Conjecture~\ref{conj:Ber} was verified in 2019 by
	Litvak and Tikhomirov, who obtained quantitative estimates that cover
	all but the two extreme regimes mentioned above~\cite{LT20}.
	For an \(n\times n\) matrix \(M\) write
	\[
	s_1(M)\ge s_2(M)\ge\cdots\ge s_n(M)
	\]
	for its singular values.
	
\begin{theorem} [Litvak-Tikhomirov \cite{LT20}]
	\label{thm: LTmain}
	There is a universal constant $C>1 $ with 
	the following property. 
	Let $A_n$ be 
	an $n \times n$ random matrix
	whose entries are i.i.d Bernoulli($ p $),
	with $ p = p_n$  
	satisfying 
	\begin{align*}
		C \frac{ \log n  }{n  }
	\le 
		p
	\le 
		C^{-1}.
	\end{align*}	
	Then, when $n$ is sufficiently large, for $t \ge 0$, 
	\begin{align*}
		\mathbb{P} \left\{ 
			s_{n}( A )
		\le
		t\exp( -3 \log^2(2n) )
		\right\} 
	& = 
		t + 
		( 2 + o_n(1) )n(1-p)^n.
	\end{align*}	
	Further, this estimate can be improved when $p$ is also bounded below 
	by a constant. Let $q \in (0, C^{-1})$ be a parameter. 
	Then, there exists  $C_q >0$ such that if $p \ge q$, 
	\begin{align*}
		\mathbb{P} \left\{ 
			s_{n}( A )
		\le
			C_q n^{-2.5} t		
		\right\} 
	& = 
		t + 
		( 1 + o_n(1) )2(1-p)^nn.
	\end{align*}	
	for sufficiently large $n$ which may depend on $q$. 
\end{theorem}

	\medskip
	What remains open is the intermediate dense range
	\[
	c\;\le\; p \;\le\;\tfrac12,
	\qquad\text{for some fixed }c>0,
	\]
	and the sparse window
	\begin{equation}\label{eq:sparse-range}
	1\;\le\;\liminf_{n\to\infty}\frac{pn}{\log n}
	\;\le\;\limsup_{n\to\infty}\frac{pn}{\log n}
	<\infty.
	\end{equation}
	Jain, Sah, and Sawhney recently resolved the dense case
	\(p\in(0,\tfrac12)\) with quantitative bounds~\cite{JSS20}.
	On the sparse side, Basak and Rudelson~\cite{BR18} proved
	Conjecture~\ref{conj:Ber} when
	\[
	\log n\;\le\;pn\;\le\;\log n + o(\log\log n),
	\]
	so the gap described in~\eqref{eq:sparse-range} remains the main
	outstanding case in the sparse regime.

Beyond singularity, similar questions can be asked about 
the rank of random discrete matrices. Indeed,  recently
Jain, Sah, and Sawhney \cite{JSS21} settles the rank conjecture for random sign matrices: 
\begin{align*} 
	\mathbb{P} \left\{ 
		\opn{corank}(B_n) \ge k 
	\right\}
=
	\mathbb{P} \left\{ 
		s_{n - k +1 } (B_n)  =  0
	\right\}
&
= 
	( \frac{1}{2} + o(1))^{kn}\,,
\end{align*}
and also resolved the same type result for Bernoulli matrices 
with constant $p$.

The study of rank of discrete random matrices is beyond such models. 
In the symmetric case, the rank of a symmetric random 
Bernoulli matrix (the adjacency matrix of a $G(n,p)$ graph) 
has been studied by Costello and Vu \cite{CV08, CV09}
from the range $ \frac{ \log n  }{ 2n } << p
\le \frac{ 1 }{ 2 }$.  
Recently, a result of Coja-Oghhan, Erg{\"u}r, Gao, Hetterich, and Rolvien \cite{CEGHR} 
studies the rank of sparse random matrices with prescribed support size for each column and row. 
For further discussion about the rank of discrete random 
matrices, we refer the readers to the survey 
\cite{Vu20}.

In this paper, we establish a result for rank of
Bernoulli matrices in the sparse regime. 

Let $ A $ be a $n\times n$ matrix with i.i.d Bernoulli entries with probability $ p  $ where $ p  =  p _n$ is a $ n $-dependent value. 

\begin{theorem} \label{thm: main}
	For a positive integer $\beta$, there exists $c_\beta > 0 $
	so that the following holds. 
	Suppose $ p $ satisfies 
	\begin{align*}
		\frac{ \log( n )  }{ n   }	\le p 
		\le c_\beta.
	\end{align*}	
	Then, when $ n $ is sufficiently large, for any $t \ge 0$,
	\begin{align*}
		 &
		\mathbb{P} \left\{ 
			s_{ n - \beta + 1}( A  )
		\le
		tn^{-2\beta +1+ o_n(1)  }( pn )^{-7}
		\right\} \\
	=&  
		t + 
		( 1 + o_n(1) )  \mathbb{P} \bigg\{ 
\max\bigl\{\#\{\text{zero rows of }A\},
					   \#\{\text{zero columns of }A\}\bigr\}
%		\max\{ \mbox{number of zero columns , number of zero rows}\} 
	\ge \beta  \bigg\}.
	\end{align*}	
	
	\end{theorem}
	We remark that while $c_\beta$ is a constant, it is strictly less than $ \frac{ 1 }{ 2 }$. 
	Our result does not cover the regime for $p$ close to $ \frac{ 1 }{ 2 }$. 
	In the regime when $p$ is close to  $\log n /n$. By setting $ \beta =1$, we resolve Conjecture 
	\ref{conj:Ber} in the regime described in \eqref{eq:sparse-range}.

	\begin{corollary}
		Suppose $ p $ satisfies 
	\begin{align*}
		1 \le \liminf \frac{  pn }{  \log( n)  }
		\le \limsup \frac{ pn  }{ \log( n)  }
		 <  +\infty. 
	\end{align*}	
	Then, when $ n $ is sufficiently large, for $t \ge 0$,
	\begin{align*}
		&\mathbb{P} \left\{ 
			s_{ n - \beta + 1}( A  )
		\le
			tn^{-2\beta + o_n(1)  }
		\right\} \\
	=  &
		t + 
		( 1 + o_n(1) )  \mathbb{P} \bigg\{ 
		 { A }\mbox{ has a zero column or a zero row.} \bigg\}  \\
	= & t + (1+o_n(1)) 
	\left(  1 - ( 1- (1 -  {p} )^n ) 
		)^{2 n} \right) .
	\end{align*}
	\end{corollary}
		A simple computation shows that the probability of a \emph{zero row}  or a \emph{zero column} in \(A_n\) is
		\[
		\mathbb{P}\!\bigl\{A_n\text{ has a zero row or a zero column}\bigr\}
		  \;=\;(1+o_n(1))
				\Bigl[1-\bigl(1-(1-p)^n\bigr)^{n}\Bigr].
		\]
				\medskip
		When \(p\gg \tfrac{\log n}{n}\) the event is rare, and a first–order expansion yields
		\[
		1-\bigl(1-(1-p)^n\bigr)^{n}\;=\;n(1-p)^n\bigl(1+o_n(1)\bigr),
		\]
		so that
		\[
		\mathbb{P}\!\bigl\{A_n\text{ has a zero row or a zero column}\bigr\}
		  \;=\;(1+o_n(1))\,2n(1-p)^n,
		\]
		precisely the leading‐order term that appears in
		Theorem~\ref{thm: LTmain}.  A detailed proof of this estimate is given in Lemma~\ref{lem: nonemptyEstimate}.

\subsection{Acknowlegement}
The author would like to thank Konstantin Tikhomirov for the suggestion of
this question and the fruitful discussion on this project. 
\subsection{Outline of proof}

	Fix an integer \(1\le\beta\le n\) and define the event  
	\[
	\Omega_{\mathrm{RC}}=\Omega_{\mathrm{RC}}(A)
	  :=\Bigl\{
		   \max\bigl\{\#\{\text{zero rows of }A\},
					   \#\{\text{zero columns of }A\}\bigr\}<\beta
		 \Bigr\}.
	\]
	Our goal is to construct an event \(\Omega\) such that  
	\begin{equation}\label{eq:properProb}
	\mathbb{P}\bigl(\Omega^{\mathrm c}\bigr)
	   =(1+o_n(1))\,
		 \mathbb{P}\bigl(\Omega_{\mathrm{RC}}^{\mathrm c}\bigr),
	\end{equation}
	and, moreover, for every \(A\in\Omega\cap\Omega_{\mathrm{RC}}\) the
	singular value \(s_{\,n-\beta+1}(A)\) is bounded below by a prescribed
	threshold.
	By the  min–max theorem for singular values, 
	\[
	s_{\,n-\beta+1}(A)
	  \;\ge \;
	  \max_{\substack{I,J\subset[n]\\|I|=|J|=n-\beta+1}}
		   s_{\min}\bigl(A_{I,J}\bigr),
	\]
	where \(A_{I,J}\) denotes the submatrix determined by rows~\(I\) and
	columns~\(J\).  Consequently  
	\[
	\mathbb{P}\bigl\{s_{\,n-\beta+1}(A)\le t\bigr\}
	   \;\le\;
	   \mathbb{P}\Bigl\{\max_{I,J}
			  s_{\min}\bigl(A_{I,J}\bigr)\le t\Bigr\}.
	\]
	A naïve union bound over all
	\(\binom{n}{n-\beta+1}^2\) submatrices would be far too coarse if we
	wish to attain the precision demanded by~\eqref{eq:properProb}.  The
	key instead is to show that, with probability
	\(1-(1+o_n(1))\,\mathbb{P}(\Omega_{\mathrm{RC}}^{\mathrm c})\), there
	exists an \emph{individual} pair of index sets \((I_0,J_0)\) such that
	\begin{enumerate}
	\item[(i)] \(A_{I_0,J_0}\) contains no zero rows or columns, and
	\item[(ii)] no linear depencence among a few rows or columns of \(A_{I_0,J_0}\). 
	\end{enumerate}
	
	On this favourable event we may restrict attention to the single
	\((n-\beta+1)\times(n-\beta+1)\) submatrix \(A_{I_0,J_0}\), obtaining  
	\begin{align*}
	\mathbb{P}\bigl\{s_{\,n-\beta+1}(A)\le t\bigr\}
	  \;\le\; &
	  (1+o_n(1))\,\mathbb{P}\bigl(\Omega_{\mathrm{RC}}^{\mathrm c}\bigr)\\
	  & \,+\mathbb{P}\Bigl\{s_{\min}(A_{I_0,J_0})\le t\;
		  \Big|\;A_{I_0,J_0}\text{ satisfies \rm{(i)}–\rm{(ii)}}\Bigr\}.
	\end{align*}

	Now, we turn our attention to the estimation of the least singular value of $A_{I_0, J_0}$. We need to demonstrate with high probability that for every non-zero vector $x \in \mathbb{R}^{J_0}$ (the coordinate subspace of $\mathbb{R}^{n}$ corresponding to the basis $(e_j)_{j \in J_0}$), the following inequality holds:

	\begin{align*}
		\frac{|A_{I_0, J_0}x|}{|x|} \ge t.
		\end{align*}

		\noindent
		\textbf{Vector decomposition.}
		A standard approach to estimating the least singular value of a square
		random matrix is to decompose the ambient Euclidean space into several
		classes of vectors and analyse each class separately.  Depending on
		the class, one invokes appropriate tools—Littlewood--Offord–type
		anti‐concentration estimates, geometric arguments combined with
		Rogozin’s theorem, or expansion properties in the sparse setting.
		
		This strategy originates in the seminal papers of Tao–Vu~\cite{TV09}
		and Rudelson~\cite{Rud08}, and has since proved remarkably effective.
		A landmark result of Rudelson and Vershynin~\cite{RV08} establishes the
		optimal tail bound
		\[
		\mathbb{P}\bigl\{s_{\min}(M_n)\le t\bigr\}
		   \;\le\; C t + e^{-c n},
		\qquad t\ge 0,
		\]
		for any \(n\times n\) matrix \(M_n\) with i.i.d.\ sub-Gaussian entries
		of mean~\(0\) and unit variance.  Subsequent work of
		Rebrova–Tikhomirov~\cite{RT15}, Livshyts~\cite{Liv18}, and
		Livshyts–Tikhomirov–Vershynin~\cite{LTV19} removed the sub-Gaussian
		and unit‐variance assumptions, extending the bound to much broader
		entry distributions.
		
		When the entries are no longer independent, analogous techniques have
		been developed for structured random matrices.  A series of papers by
		Litvak, Lytova, Tikhomirov, Tomczak‐Jaegermann, and
		Youssef~\cite{LLTTY16,LLTTY17,LLTTY-CircularLaw,LLTTY18,
		LLTTY19-ProbTheo,LLTTY19-TranAmer} applies the vector-decomposition
		framework to adjacency matrices of random regular graphs, yielding
		sharp bounds on their least singular values and related spectral
		questions.

		Our decomposition of \(\mathbb{R}^{n}\) is a slight refinement of the
		scheme introduced by Litvak–Tikhomirov~\cite{LT20}, where the ambient
		space is partitioned into three classes of vectors—\(\mathcal{V}\!-\),
		\(\mathcal{R}\!-\), and \(\mathcal{T}\!-\)vectors.  To describe these
		classes clearly, we begin with some notation.
		
		For any vector \(x=(x_1,\dots,x_n)\in\mathbb{R}^{n}\), let
		\(\sigma_x:[n]\to[n]\) be a permutation satisfying
		\[
		\bigl|x_{\sigma_x(1)}\bigr|
		  \;\ge\;
		  \bigl|x_{\sigma_x(2)}\bigr|
		  \;\ge\;\cdots\;\ge\;
		  \bigl|x_{\sigma_x(n)}\bigr|.
		\]
		Define the \emph{non-increasing rearrangement} of \(x\) by
		\[
		x^{*}_{i}\;:=\;\bigl|x_{\sigma_x(i)}\bigr|,
		\qquad i\in[n].
		\]
		Thus \(x^{*}=(x^{*}_{1},\dots,x^{*}_{n})\) lists the absolute values of
		the coordinates of \(x\) in descending order, and serves as a convenient
		tool for formulating size and sparsity conditions on vectors.

		\subsubsection*{Gradual vectors (\(\mathcal{V}\)-vectors)}
The first class in our decomposition consists of \emph{gradual vectors},
denoted \(\mathcal{V}\).  A vector \(x\in\mathbb{R}^{n}\) belongs to
\(\mathcal{V}\) if
\begin{itemize}
	\item[i] its support size satisfies \(\#\operatorname{supp}(x)\ge cn\) for a fixed \(c>0\); and
	\item[ii]  the \emph{profile} \(i\mapsto x^{*}_{i}/x^{*}_{\lceil cn\rceil}\) varies slowly as \(i\searrow1\).
\end{itemize}
 
\noindent
Litvak and Tikhomirov established the following partial estimate for
\(\mathcal{V}\)-vectors:

\textit{Let \(B\) be an \(n\times n\) matrix with i.i.d.\
\(\operatorname{Bernoulli}(p)\) entries.  Then}
\begin{equation}\label{eq:introLTpartial}
\mathbb{P}\!\Bigl\{
  \forall\,x\in\mathcal{V}:\ \|Bx\|\ge t\|x\|
  \,\Bigm|\, 
  \forall\,y\notin\mathcal{V}:\ \|B^{\!\top}y\|>0
\Bigr\}
\;\ge\; 1-Ct-c_{p,n}.
\end{equation}

The bound holds whenever \(p\ge\tfrac{\log n}{n}\) and
\(c_{p,n}\le e^{-Cpn}\).  Adapting the “invertibility via distance’’
argument (see Theorems 2.1–2.2 of~\cite{LT20}), one can choose
\(c_{p,n}=o\!\bigl(\mathbb{P}\{\Omega_{\mathrm{RC}}^{\mathrm
c}(\beta)\}\bigr)\) for fixed \(\beta\) and sufficiently large~\(n\).
Hence~\eqref{eq:introLTpartial} is applicable to the submatrix
\(A_{I_{0},J_{0}}\) by tacitly identify each \emph{gradual vector}
\(x\in\mathbb{R}^{J_{0}}\) with its extension in \(\mathbb{R}^{n}\) by
zero outside \(J_{0}\).
This leads to one condition for the row and column index set \(I_{0}\) and \(J_{0}\): 
\begin{enumerate}
\item[\textup{(i)}] \(\displaystyle
      \forall\,x\notin\mathcal{V}\ \text{with}\ 
      \operatorname{supp}(x)\subseteq J_{0}:\ 
      \|A_{I_{0},J_{0}}x\|\;>\;c'_{p,n}\|x\|;\)
\item[\textup{(ii)}] \(\displaystyle
      \forall\,x\notin\mathcal{V}\ \text{with}\ 
      \operatorname{supp}(x)\subseteq I_{0}:\ 
      \|A_{I_{0},J_{0}}^{\!\top}x\|\;>\;c'_{p,n}\|x\|.\)
\end{enumerate}
The second condition is needed to ensure it meets the conditioned event described in \eqref{eq:introLTpartial}.

\medskip
Note that \(\mathbb{R}^{n}\setminus\mathcal{V}\) contains all
\emph{sparse} vectors.  Hence if \(A_{I_{0},J_{0}}\) exhibited a local
linear dependence among a few columns, some non-zero
\(x\in\mathbb{R}^{n}\setminus\mathcal{V}\) would satisfy
\(A_{I_{0},J_{0}}x=0\), contradicting~(i).  Similarly, the family
\(\mathbb{R}^{n}\setminus\mathcal{V}\) detects zero columns: if \(A\)
had \(\beta\) zero columns we could take a non-zero \(x\) supported on
those columns, forcing \(Ax=0\).

\subsubsection*{Steep vectors (\(\mathcal{T}\)-vectors)}

Our main task is to control the class of \emph{steep vectors},
denoted \(\mathcal{T}\).  A vector \(x\in\mathbb{R}^{n}\) is placed in
\(\mathcal{T}\) if there exist indices
\(1\le n_{1}<n_{2}\ll n\) such that
\begin{equation}\label{eq:IntroSteepCond}
  x^{*}_{n_{1}}
    >C_{n_{1},n_{2}}\;x^{*}_{n_{2}},
\end{equation}
where \(C_{n_{1},n_{2}}>1\) is (much) larger than~1.  In words,
\(\mathcal{T}\)-vectors have a sharp drop in magnitude between the
\(n_{1}\)-th and \(n_{2}\)-th largest coordinates.

For a positive integer \(k\) we write \([k]=\{1,\dots,k\}\).
Given an \(m_{1}\times m_{2}\) matrix \(M\), let
\(\mathbf{R}_{i}(M)\) and \(\mathbf{C}_{j}(M)\) denote its \(i\)-th row
and \(j\)-th column, respectively.

For \(x\) satisfies~\eqref{eq:IntroSteepCond},  if there is a row
\(\mathbf{R}_{i}(A_{n})\) and an index
\(j_{0}\in\sigma_{x}\!\bigl([n_{1}]\bigr)\) such that
\[
a_{ij_{0}}=1
\quad\text{and}\quad
a_{ij}=0\quad
\text{for all }j\in\sigma_{x}\!\bigl([n_{2}]\bigr)\setminus\{j_{0}\},
\]
then
\begin{equation}\label{eq:introSteepx}
  (Ax)_{i}\;=\;
  \underbrace{x_{j_{0}}}_{\text{dominant term}}
  \;+\;
  \sum_{j\notin\sigma_{x}([n_{2}])}a_{ij}x_{j}.
\end{equation}
For sufficiently large \(C_{n_{1},n_{2}}\) the first term dominates
and \(|(Ax)_{i}|\gtrsim x^{*}_{n_{1}}\), guaranteeing that
\(\|Ax\|\) is bounded away from~\(0\).

The chance that such a row exists becomes delicate when
\(p\) approaches the critical threshold
\(\tfrac{\log n}{n}\), especially for small \(n_{1},n_{2}\).  This is
the main obstacle that prevented Litvak–Tikhomirov~\cite{LT20} from
pushing their singularity bound further into the sparse regime.
Basak–Rudelson~\cite{BR18} faced an analogous difficulty.  Their key
observation is that the relevant probability depends strongly on the
\emph{support size} of the columns indexed by
\(\sigma_{x}\!\bigl([n_{1}]\bigr)\).  

Consider the toy case
\(n_{1}=2\).  If we condition on
\[|\operatorname{supp}\mathbf{C}_{1}(A)|=
|\operatorname{supp}\mathbf{C}_{2}(A)|=pn\] (the typical size), the
probability that the two columns coincide is
\(\binom{n}{pn}^{-1}\simeq\exp(-Cpn\log(1/p))\).  However, if both
columns are conditioned to have support size~1, which has probability not negligible comparing to $\mathbb{P}\bigl(\Omega_{\mathrm{RC}}^{\mathrm c}\bigr)$, the same event occurs
with probability \(1/n=\exp(-\log n)\), sacrificing a factor
\(\log(1/p)\) in the exponent when \(p\asymp\frac{\log n}{n}\).

Our treatement for such vectors are developed based on the insight from Basak-Rudelson
\cite{BR18} while aiming at both the precise probability bound and the goal to 
showing the existence of $(I_0, J_0)$ which the argument could 
work for $  A _{ I_0,J_0 }$ and $  A_{ I_0, J_0 }^\top$ 
at the same time. 

For \(k\in[n]\) define
\begin{align*}
\mathcal{J}_{A}(k)
  :=&\Bigl\{j\in[n] : |\operatorname{supp}\mathbf{C}_{j}(A)|\le k\Bigr\}, \quad \mbox{and}\\
\,
\mathcal{I}
  :=& \Bigl\{i\in[n] :
          \exists\,j\in\mathcal{J}_{A}(s_{0}),\;a_{ij}=1\Bigr\},
\end{align*}
where \(s_{0}=s_{0}(p,\beta)\) will be specified later and
\(\mathcal{J}:=\mathcal{J}_{A}(s_{0})\).

Reordering indices so that \(\mathcal{I}\) and \(\mathcal{J}\) come
first, \(A\) assumes the block form
\[
A \;=\;
\begin{pmatrix}
  H & G \\[2pt]
  0 & W
\end{pmatrix},
\qquad
H=A_{\mathcal{I},\mathcal{J}},\;
G=A_{\mathcal{I},\,[n]\setminus\mathcal{J}},\;
W=A_{[n]\setminus\mathcal{I},\,[n]\setminus\mathcal{J}}.
\]
(The missing block
\(A_{[n]\setminus\mathcal{I},\,\mathcal{J}}\) is identically zero by
construction.)

\newpage

\paragraph{Choice of the parameter \(\boldsymbol{s_0}\).}
At first glance one might wish to pick \(I_{0},J_{0}\) so that the
submatrix \(A_{I_{0},J_{0}}\) retains the “good’’ block \(W\) and avoids
the “bad’’ block \(H\) introduced earlier.
From a technical perspective, however, the optimal choice of the
threshold
\[
  s_0=s_0(p,\beta)
\]
depends in a subtle way on both \(p\) and \(\beta\); in fact it
undergoes a phase transition at
\(p\asymp\tfrac{\log n}{n}\).  We distinguish two regimes.

\noindent
\underline{Critical window}
Assume
\[
  \frac{\log n}{n}\;\le\;p\;\le\;
  \Bigl(1+\frac{1}{2\beta}\Bigr)\frac{\log n}{n}.
\]
Here we take \(s_0\) to be a \(\beta\)-dependent \emph{constant}.  Even
though \(A\) may then possess up to \(\operatorname{polylog}(n)\)
columns whose supports have size at most \(s_0\), the probability that
these columns become linearly dependent is still
\(
  o\!\bigl(\mathbb{P}\{\Omega_{\mathrm{RC}}^{\mathrm c}\}\bigr)
\).
To see this, let \(J\) be the set of column indices with support size
exactly \(1\).  For any threshold \(u\ge1\),
\begin{align*}
  & \mathbb{P}\!\bigl\{\exists\,j_1,j_2\in J:
         \mathbf{C}_{j_1}(A)=\mathbf{C}_{j_2}(A)\bigr\}\\
  \le&
    \underbrace{\mathbb{P}\{|J|\ge u\}}_{=:p_1}
    +
    \underbrace{\mathbb{P}\{|J|<u\}\,
      \mathbb{P}\!\bigl\{\exists\,j_1,j_2\in J:
        \mathbf{C}_{j_1}(A)=\mathbf{C}_{j_2}(A)
        \,\bigm|\,|J|<u\bigr\}}_{=:p_2}.
\end{align*}
When \(p\) is near the critical value, \(p_1\) ceases to be negligible
for fixed \(u\); choosing \(u=\operatorname{polylog}(n)\) restores the
bound \(p_1=o\!\bigl(\mathbb{P}\{\Omega_{\mathrm{RC}}^{\mathrm
c}\}\bigr)\).  The second term satisfies \(p_2\le
n^{-1+o(1)}\) if one applied the union bound, provided \(u\le\operatorname{polylog}(n)\), hence it is
also negligible.

Consequently \(J_0\) must \emph{contain} the set \(J\) (omitting only
zero columns), and \(I_0\) must contain
\[
  \bigl\{\,i\in[n] : \exists\,j\in J \text{ with } a_{ij}=1\bigr\}.
\]
Thus \(A_{I_0,J_0}\) necessarily contains the block \(H\) (minus its
zero columns).

\noindent
\underline{Above the critical window}
Now suppose
\[
  p\;>\;\Bigl(1+\frac{1}{2\beta}\Bigr)\frac{\log n}{n}.
\]
In this regime the previous bound on \(p_2\) fails, because
\[
  \mathbb{P}\!\bigl\{\mathbf{C}_1(A)=\mathbf{C}_2(A),\;
    |\mathbf{C}_1(A)|=|\mathbf{C}_2(A)|=1\bigr\}
    \not=o\!\bigl(\mathbb{P}\{\Omega_{\mathrm{RC}}^{\mathrm c}\}\bigr).
\]
The remedy is to enlarge \(s_0\) so that
\(
  |\mathcal{J}_A(s_0)|\le\beta
\)
with probability
\(1-o\!\bigl(\mathbb{P}\{\Omega_{\mathrm{RC}}^{\mathrm c}\}\bigr)\).
We then choose \(J_{0}\) to intersect \(\mathcal{J}_A(s_0)\) in at most
one index.  Any vector supported on \(J_0\) is therefore a linear
combination of at most one narrow column together with columns of large
support, preventing the cancellations that could otherwise force
\(Ax=0\).

With this construction we have ensured that for every
\(x\) supported on \(J_{0}\)
\[
  Ax=\sum_{j\in J_0}\mathbf{C}_{j}(A)\,x_j
\]
contains no non–trivial linear combination of columns with small
support.  This rules out the most dangerous obstructions and completes
the analysis of steep (\(\mathcal{T}\)) vectors, thereby finalising the
treatment of vectors outside the gradual class \(\mathcal{V}\).

\noindent
{\bf Overview and Structure of paper}

To summarize, to prove the main theorem, we break $\mathbb{R}^n$ into $\cal T$, $\cal R$, and $\cal V$ vectors, and treating bounding $\|A_{I_0,J_0}x\|$ (or $\|Ax\|$) away from $0$ quantitatively for each type of vectors. 

Roughly speaking, {\bf only the lower bound of $\|Ax\|$ for $ x \in \cal T$ is sensitive to a small number of zero columns of $A$.} For the other types of vectors, due to the mass of the vectors are not concentrated on a small number of components, the estimate for bounding $\|Ax\|$ away from $0$ does not change much whether $A$ has a constant number of zero columns or not. Therefore, the treatment for $\cal R$ and $\cal V$ vectors are essentially the same as shown in the work of Litvak-Tikhomirov \cite{LT20}. 

For the above reason, the treatment for $\cal T$ vectors is the main contribution of the paper. Following from this, after we introduce some basic notations and tools (Section 2), we break the paper into two parts:
\begin{enumerate}
	\item \textbf{Part I}: The first part (Section 3-5) is about how we handle steep vector $\cal T$. Formally speaking, we aim to prove a theorem of the following type:  

	For $x \in \cal T$, there exists $n_1=n_1(x),n_2=n_2(x) \in [n]$ such that $x^*_{n_1} \ge C pn x^*_{n_2}$. We will prove the theorem of the following form: With probability at least $1- (1+o_n(1))\mathbb{P}\{ \Omega_{\rm RC}^c(\beta)\}$, there exists $I,J\subseteq[n]$ with $|I|=|J|= n-\beta+1$ such that for each $x \in \cal T$, 
	\begin{itemize}
		\item if supp$(x) \subseteq J$, $\|A_{I,J}x\| \ge cx^*_{n_1}$. 
		\item 
	 if supp$(x) \subseteq I$, $\|A_{I,J}^\top x\| \ge cx^*_{n_1}$. 
	\end{itemize}
	The reason we also work with $A_{I,J}^\top x$ is that we will apply \eqref{eq: introLTpartial} with $B=A_{I,J}$, and we want to show the event that we are conditioned on has non-trivial probability. The concrete statement theorem and definition are Theorem \ref{thm: T_1s} and Theorem \ref{thm: T_1+}. 

	\item \textbf{Part II}: The second part \hh{(Section 6-11)} follows from the framework of Litvak-Tikhomirov \cite{LT20}, and we prove the main theorem. There is some modification we needed due to the estimate go from least singular value estimate to rank estimate, and the difference of the technical details. We also move some of the parts into appendix, whenever the proof does not essentially differ from the work of Litvak-Tikhomirov \cite{LT20}.   
\end{enumerate}
 
% 
% 
% 
% 
% {\bf Main Contribution}
% The main contribution of the paper is on handling the treatment on the steep vector ${\cal T}$ described before. The framework of the proof, in particular about the decomposition of the vectors, follows the structure of the work of Litvak-Tikhomirov. The structure of the paper is arranged in the following:
% 
% Section 2: Notations and tools 
% \hh{
% Section 3-5: We describe the main contribution of the work here about addressing estimating $\|A_{I_0,J_0}x\|$  for steep type vector $x$, and the choice of $I_0$ and $J_0$.  }
% 
% \hh{
% Section 6-8: Here we decompose the vectors of $\R^n$ from the framework of Litvak-Tikhomirov. Due to some technical difference, we cannot immediate recycle the work, including the treatment needed to be adjusted for estimating the rank, instead of the least singular value of $A$.  }
% 
% 
% Section 4: Corresponding $\| A x\|$ estimate for $x \in \mathcal{T}_1$.
% 
% Section 5: Corresponding $\| A x\|$ estimate for $ x\in \mathcal{T}_2$, 
% $\mathcal{T}_3$, and $\mathcal{R}$.
% 
% Section 6: Proof of main theorem.
% 
\section{Notations and  standard Probability Estimates}
Let \([m]\) denote the set \(\{1, 2, \dots, m\}\) for a positive integer \(m\). For \(a < b\), let \([a, b]\) denote the discrete interval:
\begin{align*}
    [a, b] := \{ x \in \mathbb{Z} : a \le x \le b \}.
\end{align*}

By $Y\sim \opn{Bin}(n,p)$, 
we mean the random variable $Y$ is distributed 
as a Binomial random variable with parameter $n$ and $p$.

For an \(m_1 \times m_2\) matrix \(A = (a_{ij})_{i \in [m_1], j \in [m_2]}\), let \(\mathbf{R}_i(A)\) and \(\mathbf{C}_j(A)\) denote its \(i\)th row and \(j\)th column, respectively. For subsets \(I \subset [m_1]\) and \(J \subset [m_2]\), define the submatrix
\[
A_{IJ} := (a_{ij})_{i \in I, j \in J}
\]
as the matrix of \(A\) corresponding to rows \(I\) and columns \(J\). Additionally, we adopt the notation
\[
A_J := A_{[m_1], J} \quad \text{and} \quad A_{i_0J} := (a_{i_0j})_{j \in J}
\]
for \(i_0 \in [m_1]\). 
For a non-negative integer \(k \ge 0\), define the column statistic of \(A\) as
\begin{align} \label{eq:columnStat}
	 \mathcal{J}_A (k) : = \left| 
		\left\{  i \in [m_2]\,:\, 
			|\mbox{supp}\left( {\bf C}_i(A) \right)  | \le k \right\} 
	\right|.
\end{align}
Further, for a $n \times n$ random matrix $A$ and a positive integer $\beta$, recall the event we introduced in the first section: 
\begin{align}
\label{eq: omegaRC}
 	\Omega_{\rm RC}(\beta)
=
	\Omega_{\rm RC}(\beta,A) 
:=
	\big\{ \max\{ |\cal J_A(0)|, |\cal J_{A^\top}(0)|\} < \beta\big\}.
\end{align}

Now, we will introduce necessary results to prove the main theorem.  
First, we cite a norm estimate for sparse Bernoulli matrices 
from 
Litvak-Tikhomirov \cite{LT20}:
\begin{lemma} \label{lem: OperatorNorm}
	For every $s>0$ and $R\ge 1$, there exists a 
	constant $C\ge 1$ depending on $s,R$ 
	with the following property. 
	Let $n \ge \frac{16}{s}$ be large enough and 
	$p \in (0,1)$ satisfies $s\log n  \le pn$. Let 
	$A$ be a $n\times n$ Bernoulli($p$) matrix. 
	Then, 
	\begin{align*}
		\mathbb{P}\left\{ \left\| A-\mathbb{E}A 
			\right\| \ge C  \sqrt{ pn } 
		\mbox{ or } \left\| A \right\| \ge 
		C \sqrt{ pn } + pn
		\right\}  \le \exp(-Rpn).
	\end{align*}
\end{lemma}

Next, we also need Theorem of Rogozin on anticoncentration. 
Let us introduce the following notion of anticoncentration on a random variable: For any random variable $X$ and a positive constant $c>0$, we define 
\begin{align}
	\cal L(X, c) := \liminf_{x \in \R} \Prob\{ |X - x| \le c \}.
\end{align}

The theorem of Rogozin states quantitatively that sum of independent random variables can have stronger anticoncentration, comparing to each individual random variable: 

\begin{theorem} \label{thm:Rogozin}
	There exists a universal constant $C_{\rm rgz}\ge 1$ so that the following holds. 
	Consider independent random variables $X_1,\dots, X_n$ and  
	positive coefficients $\lambda_1,\dots,\lambda_n>0$. For any $\lambda
	>\max_{i\in [n]} \lambda_i$, we have 
	\begin{align*}
		 \mathcal{L}\big(\sum_{i\in[n]}X_i, \lambda \big) 
		\le \frac{C_{\rm rgz} \lambda}{\sqrt{\sum_{i\in [n]}
			\lambda_i^2 ( 1 - \mathcal{L}(X_i,\,\lambda_i))
		}}.
	\end{align*}
\end{theorem}
We rely the theorem to bound $(Ax)_j$ away from $0$ for a fixed vector $x$ and a Benroulli random matrix $A$: 
For $x\in \mathbb{R}^n$ and 
 $\xi_1,\dots,\xi_n$ be i.i.d. Bernoulli random variables with parameter $p$.
For $I\subset [n]$ and $ \lambda>\left\| x_I \right\|_\infty $,
we have 
\begin{align} \label{eq:RogozinIprod}
	\mathcal{L} \big(\sum_{i\in [n]}x_i\xi_i, \lambda \big)
	\le \mathcal{L} \big( \sum_{i\in I}x_i\xi_i, \lambda \big)
	\le \frac{C_{\rm{rgz}}\lambda}{
		\sqrt{\sum_{i\in I} x_i^2 p}
	} \le \frac{C_{\rm{rgz}}\lambda}{
		\sqrt{p}\left\| x_I \right\| 
	} . 
\end{align}

Here we include some standard tail estimates for Binomial and Hypergeometric Distribution. 
\begin{proposition} \label{prop: Binomial}
	Let $Y$ be a Binomial random variable with parameter  
	$n$ and $p\in (0,\frac{1}{2})$. 
	The following tail estimates hold:
	\begin{align}
	\label{eq:BnmUp}
		 k &\ge 2pn &
		 \mathbb{P} \left\{ Y\ge k \right\} 
		&\le 2\left(\frac{enp}{k(1-p)}\right)^k
		(1-{p})^{n}, \\
	\label{eq:BnmLow} 
		 k &\le \frac{1}{2}pn &
		 \mathbb{P} \left\{ Y\le k \right\}
		&\le 2\left(\frac{enp}{k(1-p)}\right)^k
		(1-{p})^{n} .
	\end{align}
\end{proposition}

\begin{proposition} \label{prop: HyperGeom}
	Assume $n$ is sufficiently large. Let $0<k\le m$ be a pair of positive integers such that $m\le \frac{n}{2}$ and 
	$4k^2 \le n $. 
	
	Consider a random subset $U\subset [n]$ distributed uniformly among all subsets of size $k$. Then,
	\begin{align} \label{eq:HyperGeomTail}
		\forall l \ge 0, \quad  
	\mathbb{P}\left\{ |U \cap [m] | \ge l \right\}
	\le C_{\rm{hg}} \left(\frac{3mk}{ln}\right)^l.
\end{align}
for a universal constant $C_{\rm{hg}}\ge 1$.
\end{proposition}
{\bf Remark}: While the inequlaity holds for all $l$, but the bound 
is trivial when $ l < \frac{ 3mk }{ n }$. 

\bigskip

\part{Treatment for Steep Vectors}
\section{Statistics of Sparse Rows and Columns and Expansion Property of $A$}

\subsection{Estimate of $|\cal J(s)|$ and $\Prob\{ \Omega_{\rm RC}^c(\beta)\}$}

For \( s \in [0, n] \), let
\begin{align*}
    p_s = p_{A, s} := \mathbb{P}\left\{ |\operatorname{supp}(\mathbf{C}_1(A))| \le s \right\} = \mathbb{P}_{Y \sim \opn{Bin}(n,p)}\{ Y \le s \},
\end{align*}

Following from Proposition \ref{prop: Binomial},  
 \begin{align} \label{eq: p_kEstimate}
	\forall s \in [0, \frac{1}{2}pn],\,\,\, 
     \left( \frac{epn}{s(1-p)} \right)^s (1-p)^n
 \le 
     p_s
 \le 
     2 \left( \frac{epn}{s(1-p)} \right)^s (1-p)^n.
 \end{align}

% Applying tail estimate of Binomial random variable from Proposition \ref{prop: Binomial}, 
% 
% \HH{Ideally, the following estimate should have been done in Prop \ref{prop: Binomial}. Rather than computing similar estimates again.}
% 
With this description, the size \( |\mathcal{J}(s)| \sim  \opn{Bin}(n,p)\). 

\begin{proposition} \label{prop: zeroColumns}
    For any positive integer \( \beta \),
    \begin{align*}
		\Prob\{ \Omega_{\rm RC}^c(\beta)\} \ge 
        \mathbb{P}\left\{ \left| \mathcal{J}_A(0) \right| \ge \beta \right\} \ge (1 + o_n(1)) \frac{1}{e\beta!} \left(n(1 - p)^n \right)^{\beta}
    \end{align*}
    when \( n \) is large enough.
	For a rough bound, there exists a universal constant $C>1$ so that 
	$$
		\Prob\{ \Omega_{\rm RC}^c(\beta)\} \ge \exp( - C\beta pn).
	$$
Further, for \( p = o(n^{-1/2}) \), it can be simplified to
\begin{align} \label{eq: J0smallp}
	\Prob\{ \Omega_{\rm RC}^c(\beta)\} \ge 
    \mathbb{P}\left\{ \left| \mathcal{J}_A(0) \right| \ge \beta \right\} \ge (1 + o_n(1)) \frac{1}{e\beta!} \exp\left( -\left( pn - \log n  \right) \beta \right).
\end{align}
\end{proposition}
\begin{proof}

	Since $|\cal J(0)|\sim \opn{Bin}(n,p_0)$, we have 
	\begin{align}
	\label{eq: JA0}
		\mathbb{P} \left\{ \left|  \mathcal{J}_A(0) \right| 
		= \beta  \right\} 
	= 
		{n \choose \beta} 
		\left( p_0 \right)^\beta (1-p_0)^{n-\beta} 
	=
		(1+o_n(1)) \frac{(p_0n)^\beta}{\beta!}(1-p_0)^{n-\beta}. 
	\end{align}
	where we applied $  \frac{ \log n  }{ n } \le p \le \frac{ 1 }{ 2 }$. \\

	\medskip

	%Now we will estimate $(p_0n)^\beta$ and $(1-p_0)^{n-\beta}$, respectively. 
	Recall that $p_0 = (1-p)^n$. With the assumption that $p \ge \log n /n$ and the fact that $1+x \le e^x$ for $x\in \R$, 
	we have 
	\begin{align}
		\label{eq: p_0}
		p_0 n \le \exp\big(-(pn - \log n ) \big) \le 1
	\end{align}
	and  
	$$
		p_0 = o_n(1). 
	$$
	
	Next, by $ ( 1+x ) = \exp( x + O(x^2))$ for $x \in \R$ and \eqref{eq: p_0}, 
	\begin{align*}
		( 1- p_0 )^{n-\beta} 
	= &
		\exp\big( -p_0 (n-\beta) + O(p_0^2 (n-\beta) ) \big)  \\
	= &
		( 1+o_n(1) ) \exp(-p_0n)
	\ge 
		( 1+o_n(1) ) \frac{1}{e}. 
	\end{align*}
	Substituting the above bound into \eqref{eq: JA0} we have  
	\begin{align*}
		\mathbb{P} \left\{ \left|  \mathcal{J}_A(0) \right| \le \beta  \right\} 
	\ge
		\mathbb{P} \left\{ \left|  \mathcal{J}_A(0) \right| = \beta  \right\} 
	\ge
		( 1 + o_n(1) ) \frac{ 1 }{ e \beta! } 
		(n(1-p)^n)^\beta 
		%\exp\big( - (pn - \log n )\beta\big)	
		,
	\end{align*}
	which is the first statement of the Proposition. 

	Further, from some standard estimate, there exists $C \ge 1$ such that  
	$(1-p) \ge \exp(-Cp)$ for $p \in [0,1/2]$. Thus, 
	\begin{align*}
		( 1 + o_n(1) ) \frac{ 1 }{ e \beta! } 
		(n(1-p)^n)^\beta 
	\ge 
		\underbrace{\frac{1}{e} \Big(\frac{n}{\beta}\Big)^\beta}_{\ge 1} (1-p)^{n\beta}
	\ge 
		\exp(-Cpn), 
	\end{align*}
	and the second statement follows. 
	Finally, the last statement follows since 
	$$ 1-p = \exp \big( -p + O(p^2) \big) \quad \mbox{and} 
	\quad p^2n = o_n(1).$$ 
\end{proof}

\begin{proposition} \label{prop: L_sboundSmallp}
	Suppose $\frac{\log n }{n} < p \le \big(1+\frac{1}{2\beta}\big) \frac{\log n }{n}$. 
	Then,
	\begin{align} 
	\label{eq: L_sboundSmallp}
		&\mathbb{P}\left\{ 
			\forall s \in \Big[1, \frac{1}{2}pn \Big], \,\,		
		\left| \mathcal{J}_A(s) \right| \ge \log^2(n) \left( \frac{e np}{s(1-p)} \right)^s \exp(-pn) n \right\} \\
	\nonumber
	&\phantom{AAA AAA AAA AAA AAA AAA AAA}\le 
	pn\exp(-\log n ^2) 
		= 
o(\Prob\{ \Omega_{\rm RC}^c(\beta)\}).
	\end{align}
\end{proposition}

\textbf{Remark:} 
The term  $\left( \frac{e np}{s(1-p)} \right)^s \exp(-pn)$ is a rough estimate of the expected size of $\mathcal{J}_A(s) $. 
The factor $\log^2(n)$ was introduced so that the probability estimate in \eqref{eq: L_sboundSmallp} is negligible comparing to $\Prob\{ {\cal J}_A(0) \ge \beta\}$ (See Proposition \ref{prop: zeroColumns}). 

\begin{proof}
	Since $|  \mathcal{J}_A(s) | \sim \opn{Bin}(n,p_s)$, we can apply the Binomial tail estimate  \eqref{eq:BnmUp} from Proposition \ref{prop: Binomial} 
	\begin{align} \label{eq: L_A(s)}
	\forall k \ge 2p_s n ,\,\,\,
		\mathbb{P} \left\{ 
			 | \mathcal{J}_A(s)| \ge k 
		 \right\} 
	\le
		2 \left( \frac{ep_sn}{k} \right)^k.
	\end{align}

	Given that $p_s = \Prob_{Y \sim \opn{Bin}(n,p)} \{ Y \le s\}$
	and the assumption $s \le \frac{1}{2}pn$, we apply Proposition \ref{prop: Binomial} again to get 
	\begin{align} \label{eq: psn}
		p_s n 
	\le  2 \left( \frac{enp}{s(1-p)} \right)^s (1-p)^n n
	\le  2 \left( \frac{enp}{s(1-p)} \right)^s \exp(-pn)n.
	\end{align}
	where we relied on  
	$ 1+x \le e^x $ for $x\in \mathbb{R}$.

	Evaluating \eqref{eq: L_A(s)} with 
	$$k = 
		\log^2(n) \left( \frac{e np}{s(1-p)} \right)^s\exp(-pn)n 
		\ge e^2 p_s n,
	$$
	we obtain
	\begin{align*}
		\mathbb{P} \left\{ 
			 | \mathcal{J}_A(s)| \ge k 
		 \right\} 
		<
		e^{-k} \le 2\exp\big(- \log^2(n)
			\big).
	\end{align*}
	Then, the inequality in \eqref{eq: L_sboundSmallp} follows from the union bound for $s \in \big[ 1, \frac{1}{2}pn\big]$. Finally, the equality in \eqref{eq: L_sboundSmallp} follows from Proposition \ref{prop: zeroColumns}.

\end{proof}

\begin{proposition} \label{prop: L_kBeta}
	Let $\beta$ be a positive integer and $c >0$ be a fixed constant. With the additional assumption on $p$ that
		$ (  1+c) \frac{ \log n  }{ n } \le p \le 1/2$.
	Then, 
	there exists $\lambda = \lambda(\beta, c) \in (0, \frac{1}{2})$ 
	such that 
\begin{align*}
	\mathbb{P} \left\{ \left|  \mathcal{J}_A( \lambda pn ) \right| 
	\ge \beta + 1 \right\} 
\le 
	{\big( ( 1-p )^n n \big)^{\beta + 3/4} }
= 
	o(\Prob\{ \Omega_{\rm RC}^c(\beta)\}),
\end{align*}
when $n$ is large enough. 
\end{proposition}

\begin{proof}
	By \eqref{eq:BnmUp} from Proposition \ref{prop: Binomial} and $p_0 = (1-p)^n$,
	\begin{align*}
		p_s 
	\le 
		2  \Big(\frac{enp}{s(1-p)}\Big)^s  (1-p)^n 
	\le 
		2  \Big(\frac{2enp} {t_2}\Big)^s p_0,
	\end{align*}
	where the last inequality follows from $p \le \frac{1}{2}$. 
	
	By setting $s = \lambda pn$ for $\lambda \in (0,\frac{1}{2})$, 
	the inequality can be rewritten in the form 
	\begin{align*}
		p_{\lambda pn} n
	\le 
		2  \exp\left(  
			\lambda \log \Big(\frac{2e}{\lambda} \Big) pn
		 \right) p_0 n. 
	\end{align*}

	From the additional assumption that $ p \ge ( 1+c ) \frac{ \log n  }{ n }$,
	\begin{align*}
		p_0n 
	= 
		(1-p)^nn \le \exp(-pn + \log n ) 
	\le 
		\exp \Big(-  \frac{c}{1+c}pn \Big).
	\end{align*}

	Since $\lambda \log( \frac{ 2e }{ \lambda })$ converges to $0$ as $\lambda$ approaches $0$, there exists $\lambda \in ( 0,1/2 )$ satisfying 
	$$
		\lambda \log\Big( \frac{2e}{\lambda} \Big) 
	\le \frac{c}{1+c} \frac{1}{5(\beta+1)}.
	$$
	With such choice of $\lambda$, we have 
	\begin{align}
		\label{eq: p_lambda_pn}
		p_{\lambda pn } 
	\le  
		(p_0n)^{1 - \frac{1}{5(\beta+1)}}.
	\end{align}

	Next, recall that  $\left|  \mathcal{J}_A( \lambda pn ) \right| \sim \opn{Bin}(n,p_{\lambda pn})$. Together with \eqref{eq: p_lambda_pn} that $ 2 p_{\lambda pn} n  \le 1 \le \beta +1$ when $n$ is sufficiently large, we apply Binomial tail bound \eqref{eq:BnmUp} from Proposition \ref{prop: Binomial}:
	\begin{align} \label{eq: LlambdaBeta}
		\mathbb{P} \left\{ \left|  \mathcal{J}_A( \lambda pn ) \right| 
		\ge \beta + 1 \right\} 
	\le &
		2\left(\frac{enp_{\lambda pn}}{(\beta+1)(1-p_{\lambda pn})}\right)^{\beta +1}
			(1-{p_{\lambda pn}})^{n} \\
	\nonumber
	\le &
		2 \Big(\frac{e}{\beta+1}\Big)^{\beta +1} \big( p_{ \lambda pn }n \big)^{\beta+1}\\
	\le &
		2 \Big(\frac{e}{\beta+1}\Big)^{\beta +1}
		(p_0n)^{\beta+4/5}
	\le (p_0n)^{\beta+3/4},
	\nonumber
	\end{align}
	where the last inequality holds since $p_0n = o_n(1)$. 

	Finally, the statement of the Proposition follows from Proposition \ref{prop: zeroColumns}.

% 	Due do $ p \le \frac{ 1 }{ 2 }$, 
% 		$(1-p)^nn \ge ( 1-p )^n \ge \exp(-C_1 pn)$ 
% 	for some $C_1>0$ which is independent from $p$ and $n$. 
% 	Hence, there exists $ 0<\lambda_1<\lambda_0$ which 
% 	depends on $C_1$ and $\beta$ such that for $0< \lambda <\lambda_1$, 
% 	\begin{align*}
% 		\lambda \log( \frac{ 2e }{ \lambda } ) ( \beta +1 ) 
% 	\le 
% 		\frac{ 1 }{ 5 } C_1. 		
% 	\end{align*}
% 	And thus, for $\lambda \in ( 0, \lambda_1 )$, 
% 	\begin{align*}
% 		p_0 n
% 		\exp\big( \lambda \log( \frac{ 2e }{ \lambda } ) pn 
% 		( \beta +1 ) \big) 
% 	\le 
% 		p_0 n 
% 		\exp\big( \frac{ 1 }{ 5 }C_1 pn \big)
% 	\le 
% 		p_0 n  ( p_0 n)^{ - \frac{ 1 }{ 5 } }
% 	\le 
% 		( p_0 n)^{4/5}. 
% 	\end{align*}
% 	\begin{align*}
% 		\mathbb{P} \left\{ \left|  \mathcal{J}_A( \lambda pn ) \right| 
% 		\ge \beta + 1 \right\} 
% 	\le
% 		C'_\beta (  p_0 n )^{ \beta + \frac{ 4 }{ 5 } }. 
% 	\end{align*}
% 	
% 	We remark that $ (  p_0 n )^{c}  = o_1(n)$ for any $c>0$ since 
% 	$ p_0 n \le  \exp(-c_1pn)$. 
% 	Thus, for $n$ sufficiently large, 
% 	\begin{align*}
% 		\eqref{eq: LlambdaBeta} 
% 		\le (  p_0 n )^{ \beta + \frac{ 3 }{ 4 } }. 
% 	\end{align*}

\end{proof}

\subsection{Expansion Property}
Consider an $m_1 \times m_2$ Bernoulli($p$) random matrix
$B$ with i.i.d. Bernoulli entries with parameter $p$. We will use a vector $$ b=(b_1,\dots,\, b_{m_2}) \in \{0,\, 1,\, \dots m_1\}^{m_2}$$ 
to denote the columns' support sizes of $B$. 
Let 
\begin{align}
\label{eq: supportProfile}	
	\Omega_{b}(B)
:=
	\Big\{ B\,:\, \forall i \in [m_2],\, |\opn{supp}({\bf C}_i(B))| = b_i  \Big\}.
\end{align}

Notice that, conditioning on $\Omega_{b}(B)$, 
$ \left\{ {\bf C}_i(B) \right\}_{i\in [m_2]}  $ 
are independent. For each $i \in [m_2]$, 
the support of ${\bf C}_i(B)$ is chosen uniformly 
among subsets of $[m_1]$ with size $b_i$. 

For $J_1\subset J_2 \subset [m_2]$, 
we want to collect each row of $B$ which has exactly non-zero entry 
in $J_1$ and zero entries in $J_2$. We define
\begin{align} \label{eq: expansionDefine}
	I_B(J_1,J_2) := \big\{ i \in [m_1]\,:\,
		\exists j_1 \in J_1 \mbox{ s.t. }
		a_{ij_1}=1 \mbox{ and } a_{ij}=0 
		\mbox{ for } j\in J_2 \backslash \{j_0\}
	\big\}.
\end{align}

The goal of this subsection to show that conditioning on the event $\Omega_b(B)$, the set $I_B(J_1,J_2)$ is not small with high probability, under suitable parameter assumptions on the triple $(J_1,J_2,b)$:
\begin{lemma} \label{lem:I(J_1,J_2)}
	Let $ B $ be an $ m_1 \times m_2 $ Bernoulli$ (p) $ matrix,  and $J_1'\subset J_2' \subset [ m_2 ]$ be two subsets. Further, 
	let $b \in \{0,1,\dots, m_1\}^{m_2}$ and $1 \le r \le  \min_{\alpha \in J_1'} b_\alpha$.
	If 
	\begin{enumerate}
		\item $2\|b\|_\infty^2 \le m_1$, and
		\item $ r \ge |J_2'|\frac{24  \|b\|_\infty^2}{m_1}$,
	\end{enumerate}	
	then 
	\begin{align*}
	    \mathbb{P} \Big\{ 
			|I_B(J_1', J_2')|  <  \frac{ |J_1'|r}{4} \Big\vert\, 
 			\Omega_{b}(B) 
		\Big\}  \le  
		C_{\rm hg}^{|J_1'|} \exp \left( -  
			\log \left(  \frac{ rm_1 }{ 24 |J_2'| \|b\|_\infty^2} 
			\right) \frac{r |J_1'|}{8} 
		\right).	
	\end{align*}		
	where $C_{\rm{hg}} > 1$ is a universal constant.
\end{lemma}

% The proof of this lemma relies on the following Proposition:
% 
% \begin{proposition} \label{prop:Expansion}
% 	Given a positive integer $k \in \N$ and a sequence of $(b_0, b_1,\dots, b_k) \in \N^{k}$ of positive integers. Let $U_0, U_1,\dots, U_k$ be independent random subsets of $[ m_1 ]$, where for each $\alpha \in [0,k]$, $U_\alpha$ is uniformly chosen among all subsets of $[m_1]$ of size $b_\alpha$. Further, let $U_{\le \alpha} = \bigcup_{\beta \in [0,\alpha]} U_\beta$. 
% 
% 	\begin{align*}
% 		\forall t > &\max\Big\{ \frac{ \|b\|_\infty ( \|b\|_1 + s)}{m_1}\,,\, 1 \Big\},\\
% 			& \mathbb{P} \bigg\{ \sum_{\alpha \in [k]} |U_\alpha \setminus U_{\le \alpha-1}|		
% 		\ge tk \bigg\} \le 
% 		C_{\rm{hg}}^k\exp \left( -  
% 			\log \left(  \frac{tm_1}{6\|b\|_\infty ( \|b\|_1 + s)} 
% 			\right)tk 
% 		\right).
% 	\end{align*}
% 	Also, for $\alpha \in [0,k]$, let 
% 	$U_{\le \alpha} = (\cup_{\beta \in [0,\alpha]} U_\beta)$. 
% 	Further, let 
% 	$$X_\alpha:= \left| U_{\le \alpha-1} \cap U_\alpha \right|$$
% 	for $\alpha \in [k]$. 
% 
% 	For $t>\max\{ \frac{ \max_{\alpha \in [k]}b_k \sum_{\alpha \in [0,k]}b_\alpha }{m_1}
% 	,1\}$,
% 	we have
% 	\begin{align*}
% 		\mathbb{P} \left\{ \sum_{i=1}^k X_i \ge tk \right\} \le 
% 		C_{\mbox{hg}}^k\exp \left( -  
% 			\log \left(  \frac{t}{6Ss/m_1} 
% 			\right)tk 
% 		\right).
% 	\end{align*}
% 	
% \end{proposition}

\begin{proof} 	
Without lose of generality, we assume $ J_1'=[k]$.
	
	\underline{Step 1. Reduction} 
	\smallbreak

	For $\alpha \in [k]$, let
	$$	
		I_\alpha = I_B\big([\alpha], J_2' \backslash [\alpha+1,k] \big).
	$$	
	Notice that $I_k = I_B([k],J_2')$, which is the set that we are interested. We break 
	$$
		|I_k| = |I_k| - |I_{k-1}| + |I_{k-1}| - |I_{k-2}| \ldots -|I_1|+|I_1|,
	$$
	and we will estimate $|I_\alpha| - |I_{\alpha -1}|$ instead.   

	Now, for $\alpha \in [k]$, let 
\begin{itemize}
	\item $ U_\alpha:= \opn{supp}({\bf C}_\alpha(A))$, 
	\item $U_{\le \alpha} = \bigcup_{\beta \in [\alpha]} U_{\le \alpha-1}$, and 
	\item $V:=  \bigcup_{ j \in J_2' \backslash [k] } 
	\opn{supp}({\bf C}_j(B))$. 
\end{itemize}	

		Observe that 
	\begin{align}
		\label{eq: Ialpha}
		I_\alpha 
	=& 
		I_B\big([\alpha-1], J_2' \setminus [\alpha+1,k] \big)
		\sqcup I_B\big(\{\alpha\}, J_2' \setminus [\alpha+1,k] \big)\\
	\nonumber
	=&
		\Big( I_{\alpha-1} \setminus U_\alpha\Big)
		\sqcup 
		\Big( U_\alpha \setminus U_{\le \alpha} \Big), 
	\end{align}
	and 
	$$
		I_k = I_B([k],J_2'),
	$$
	which is the desired set we want to bound. 

	First, we have 
	\begin{align*}
		|I_1| =  |U_1 \backslash V | = b_1 - | U_1 \cap V |.	
	\end{align*}

	Second, for $\alpha \in [2,k]$, from \eqref{eq: Ialpha} we have 
	\begin{align}
		\label{eq: Ialpha2}
		|I_\alpha| & 
	=   
		( |I_{\alpha-1}| -  |I_{\alpha-1} \cap U_\alpha| ) 
		+  ( |U_\alpha| - |U_\alpha \cap (V \cup U_{\le \alpha-1})| ). 
	\end{align}

	With 
	$$I_{\alpha-1} \subseteq U_{\le \alpha-1}
	\subseteq V \cup U_{ \le \alpha-1},$$
	we have 
	$$
		|I_{\alpha -1} \cap U_\alpha| 
	\ge 
		|U_\alpha \cap (V \cup U_{\le \alpha-1})|.
	$$
	Substituting the above bound into \eqref{eq: Ialpha2} we obtain 
	$$
		|I_\alpha| - |I_{\alpha-1}| 
	\ge 
		b_\alpha - 2 \underbrace{|U_\alpha \cap (V \cup U_{\le \alpha-1})|}_{:=X_\alpha}.
	$$
	Therefore, we have 
	$$
		|I_B(J_1',J_2') 
	\ge 
		|I_k| 
	\ge
		\sum_{\alpha \in [k]}b_\alpha - 2\sum_{\alpha \in [k]}X_\alpha
	\ge 
		rk - 2 \sum_{\alpha \in [k]} X_\alpha,\,
	$$
	and thus 
	\begin{align*}
    	\mathbb{P} \Big\{ 
			|I_B([k], J_2')|  <  \frac{ kr}{2} \Big\vert\, 
 			\Omega_{b}(B) 
		\Big\}
	\le & 
		\mathbb{P} \Big\{ 
		2\sum_{i=1}^k X_i \ge rk - \frac{rk}{2}	
		 \Big\vert\, 
 			\Omega_{b}(B) 
			\Big\} \\
	= &
		\mathbb{P} \Big\{ 
		\sum_{i=1}^k X_i \ge \frac{rk}{4}	
		 \Big\vert\, 
 			\Omega_{b}(B) 
			\Big\}.
	\end{align*}

	\underline{Step 2: Replacing $\{X_\alpha\}$ by independent random variables $\{Y_\alpha\}$}
	\smallbreak

	If $\{X_\alpha\}_{\alpha \in [k]}$ are independent after condition on $\Omega_b(B)$, then we could compute the exponential moments of their sum and derive the desired bound. Since they don't have the independence, the task is to replace 
	$ \left\{X_i\right\} _{i\in [k]}$ 
	by a suitable collection of independent random variables. 

	For $\alpha\in [k]$, let $\sigma_\alpha$ be a permutation on 
	$[m_1]$ determined by 
	$V \cup U_{\le \alpha-1}$ satisfying 
	$ \sigma_\alpha\big(\big[ |V \cup U_{\le \alpha-1}| \big]\big) 
	= V \cup U_{\le \alpha-1}$.
	Observe that, condition on $\Omega_b(B)$, 
	$(V,U_1,\dots,U_k)$ and $(V_0, \sigma_1(U_1),\dots, \sigma_k(U_k)$ have the same distribution. 
	(which can be verified inductively.)

	Now, we set 
	$$ \tilde X_\alpha 
	= 
		\Big| \big[ |V \cup U_{\le \alpha-1}| \big]
		\cap U_\alpha \Big|,
	$$
	and let 
	$$
		Y_\alpha = \big| \big[  |s_\alpha | \big] \cap U_\alpha \big|,
	$$ 
	where $s_\alpha = |V| + \sum_{\beta \in [k]}|U_\beta|$. 
	Clearly, we have $X_\alpha \le Y_\alpha$ for $\alpha \in [k]$, which in turn implies  
	\begin{align*}
		\mathbb{P} \Big\{ 
			\sum_{\alpha=1}^k X_\alpha \ge \frac{r k}{4} \Big\vert\, 
 			\Omega_{b}(B) 
		\Big\}
	= &
		\mathbb{P} \Big\{ 
			\sum_{\alpha=1}^k \tilde X_\alpha \ge \frac{r k}{4} \Big\vert\, 
 			\Omega_{b}(B) 
		\Big\} \\
	\le  &
	\mathbb{P} \Big\{ 
			\sum_{\alpha=1}^k Y_\alpha \ge \frac{r k}{4} \Big\vert\, 
 			\Omega_{b}(B) 
		\Big\}.
	\end{align*}

	\underline{Step 3: Exponential Moments of $\sum_{\alpha \in [k]}Y_\alpha$ and Tail Estimate}
	\smallbreak

	Now, for each $\alpha \in [k]$, $Y_\alpha$ is a hypergeometric distribution.	By the Hypergeometric tail bound \eqref{eq:HyperGeomTail} from Proposition \ref{prop: HyperGeom},
	for $\alpha \in [k]$, 
	% \HH{Need $s_k \le \frac{m_1}{2}$, and $\|b\|_\infty = o(\sqrt{m_1})$. Maybe check on how we apply it before adjust this.}
	\begin{align*}
		\mathbb{P} \left\{  Y_\alpha = t \right\} 
	\le
		\mathbb{P} \left\{  Y_\alpha \ge t \right\}  
	\le 
		C_{\rm{hg}} & \left( \frac{ 3s_\alpha b_\alpha }{tm_1} \right)^t\\
	&\le 
		C_{\rm{hg}} \left( \frac{ 3s_k \|b\|_\infty }{tm_1} \right)^t
	\le 
		C_{\rm{hg}} \left( \frac{ 3|J_2'|\|b\|^2_\infty }{tm_1} \right)^t
	\end{align*}
	for $t \ge 1$. 

	Now, let 
	$$\lambda(u):=
	\log \Big(\frac{um_1}{3|J_2'|\|b\|^2_\infty } \Big)>0.$$
	For $t\ge 1$, we have
	\begin{align*}
		\mathbb{P} \left\{ Y_j=t \right\} \exp(\lambda(u)t) 
		\le C_{\rm{hg}}\Big( \frac{u}{t} \Big)^t,
	\end{align*}
	and hence,
	\begin{align*}
		\mathbb{E} \exp(\lambda(u)Y_j) 
		\le 
		1+\sum_{t=1}^\infty
		C_{\rm{hg}}\Big( \frac{u}{t} \Big)^t,
	\end{align*}
	To estimate the sum, notice that 
	the function $t \rightarrow (\frac{u}{t})^t$ reaches its 
	maximum when $t = \frac{u}{e}$ with value $e^{\frac{u}{e}}$.
	And for $t\ge 2u$, it is bounded by $ \left( \frac{1}{2} \right)^{t}$.
	Therefore, 

	\begin{align*}
		\mathbb{E} \exp(\lambda(u)Y_j)	
		\le 1+ 2uC_{\rm{hg}}\exp(u/e) + 
		2C_{\rm{hg}}\left( \frac{1}{2} \right)^{2u}
		\le C \exp(2u/e),
	\end{align*}
	for some universal constant $C\ge 1$. 

	Since exponential moments are known, we may apply the Makrov's
	inequality to derive the tail bound of the summations:
	\begin{align*}
	\mathbb{P} \Big\{ 
			\sum_{\alpha=1}^k Y_\alpha \ge \frac{r k}{4} \Big\vert\, 
 			\Omega_{b}(B) 
		\Big\}
	\le &
		\mathbb{E} \exp \Big( \lambda(u)\sum_{\alpha=1}^kY_\alpha \Big) \exp\Big(-\lambda(u)\frac{rk}{4} \Big)\\
	\le &
		C^k
		\exp \Big( \lambda(u) k\Big(\frac{2}{e}u- \frac{r}{4}\Big) \Big).
	\end{align*}

	By setting $ u = \frac{r}{8}$, 
$$
	C^{k} \exp \Big( \lambda(u) k\Big(\frac{2}{e}u- \frac{r}{4}\Big) \Big)
\le 
	C^{k} 
	\exp\Big( 
			-\log\Big(\frac{1}{24} 
			\frac{m_1}{|J_2'|\|b\|_\infty} 
			\frac{r}{\|b\|_\infty}\Big)
	 \frac{r}{8}k \Big).
$$
\end{proof}

\section{Typical Realization of $A$ for 
$\frac{ \log( n )  }{ n  }
\le 	p 
\le 
	( 1 + \frac{ 1 }{ 2 \beta } ) 
	\frac{ \log( n )  }{ n  } $
}
\label{sec: T*}

In this section, we assume $ \frac{\log n }{n} \le p \le \big( 1+ \frac{1}{2\beta}\big) \frac{\log n }{n}$. For a given pair of constants $C \ge 1$ and $c \in (0,1)$, let  
\begin{align*}
	\cal T^*  =  \cal T^*(C,\gamma)  \subseteq \R^n
\end{align*}
be the collection of $x \in \R^n$ such that 
$$
	x_{n_1}^* \ge Cpn x^*_{n_2}  
$$
for a pair of $(n_1,n_2) = (n_1(x),n_2(x)) \in \N^2$ satisfying one of the following two conditions:
\begin{enumerate}
	\item [(i)]$1 \le n_1 
	\le  \Big\lceil  \exp \Big( \frac{  \log n   }{ 
	\log^2( \log n  ) } \Big) \Big\rceil  \mbox{ and } n_2 \le 1000 n_1$. 
	\item[(ii)] $	\frac{1}{10} \Big\lceil \exp\Big( \frac{ \log n  }{ 
				\log^2(  \log n   ) } \Big) \Big\rceil
				\le  n_1 \le  
				\frac{1}{ \gamma  p } \mbox{ and }				
				n_1 < n_2 \le 
				\min \Big\{
					\frac{  \log n  }{  \log^3 ( \log n  )}n_1, 
					\frac{1}{ \gamma   p } 
				\Big\}$.
\end{enumerate}
Consider the event 
	\begin{align*}
		\Omega_{\cal T^*}(\beta&, C,\gamma, c) := \\
		\Big\{ 
		\exists &I,J  \subseteq [n] \mbox{ with } |I| =|J|=n-\beta+1 \mbox{ such that }		\\
		& \forall x \in \cal T^*(C,\gamma) \mbox{ with } \opn{supp}(x) \subseteq [J], \quad
			 \|A_{I,J}x\| \ge c x^*_{n_1(x)} \mbox{ and }
		 \\
		& \forall x \in \cal T^*(C,\gamma) \mbox{ with } \opn{supp}(x) \subseteq [I], \quad
			 \|A_{I,J}^\top x\| \ge c x^*_{n_1(x)}
		\phantom{ AAA AA} \Big\} .
	\end{align*}

The main goal is to establish the following theorem. 
\begin{theorem} \label{thm: T_1s}
	For a positive integer $\beta$, 
	there exists $c \in ( 0, 1/2 ),  C>1$, and $\gamma >1$
	so that the following holds: Suppose $ p $ 
	satisfies 
	$$
		\frac{ \log( n ) }{ n } 
	\le
		p 
	\le
		\big( 1+ \frac{1}{2\beta}\big) \frac{\log n }{n}.
	$$
	Then, 
	\begin{align*}
		\mathbb{P}\left\{ \Omega^c_{\cal T^*}(\beta, C,\gamma, c) \right\}  
	=
		( 1 + o_{ n}(1))
		\mathbb{P} \left\{ \Omega_{\rm RC}^c(\beta)\right\} .
	\end{align*}
\end{theorem}

The technial part of the above theorem is the following: 

\begin{theorem} \label{thm: SteepExpansion} 
	For a fixed positive integer $\beta$, there exists $ C \ge 1, \gamma \ge 1$, $c \in (0,1)$ and 
	$s_0 \ge \beta $ so that the following holds. 
	Suppose $ p $ satisfies 
	$$
		\frac{ \log( n ) }{ n } 
	\le
		p 
	\le
		\Big(  1+ \frac{ 1 }{ 2\beta } \Big)
		\frac{ \log( n ) }{ n } . 
	$$

	There exists an event $\Omega_{\ref{thm: SteepExpansion}} $
	with $$\mathbb{P}( \Omega_{\ref{thm: SteepExpansion}}^c) 
	= 	o \big(
		\mathbb{P}\{ \Omega_{\rm RC }^c(\beta) \}  \big)
	$$ satisfying the following 
	description: 
	For any sample $A \in \Omega_{\ref{thm: SteepExpansion}} $, 
	let $\mathcal{J} = \mathcal{J}_A(  s_0)$, 
	and 
	\begin{align*}
		\mathcal{I} := \left\{ i \in [n] 
			\, :\, 
			\exists j \in \mathcal{J} \mbox{ s.t. }
			a_{ij} = 1
		\right\}.
	\end{align*}
	For $x \in \cal T^*(C)$ satisfying $\opn{supp}(x) \cap \cal J(0) = \emptyset$, one of the following holds:	
	\begin{itemize}
		\item Either $\exists i \in \cal I$ such that $ |(Ax)_i|  
				\ge c x^*_{n_1}$, or
		\item 	$ | \left\{ 
				i \in [ n] \,: \, |(Ax)_i| \ge
			c x^*_{n_1}	
		 \right\}  | > \frac{  s_0}{8}$.
	\end{itemize}
\end{theorem}

This section consist 3 parts: 
\begin{enumerate}
	\item We construct the high probability event $\Omega_{\ref{thm: SteepExpansion}}$ described in the theorem. 
	\item Then, we derive the event $\Omega_{\ref{thm: SteepExpansion}}$ fails with probability which is negligible comparing to $\Prob\{ {\cal J}_A(0) \ge \beta \}$.  
	\item The proof of Theorem \ref{thm: SteepExpansion} and Theorem \ref{thm: T_1s}.
\end{enumerate}

In this regime, following from Proposition \ref{prop: zeroColumns} and \eqref{eq: J0smallp} we have 
\begin{align} \label{eq: J0smallp01}
	 \Prob\{ \Omega_{\rm RC}^c(\beta) \} \ge \Prob\{ {\cal J}_A(0) \ge \beta\} \ge (1+o_n(1)) \frac{1}{e\beta!}n^{-1/2}. 
\end{align}
Let $$ s_0 \ge \max\{\beta, 8\} \mbox{ be a fixed constant to be determined later and set } \cal {J} :=  \mathcal{J}( s_0 ).$$ 
Next, let 
\begin{align} \label{eq: defnIsparse}
	{\cal I} := \bigcup_{j \in {\cal J}} \opn{supp}({\bf C}_j(A)).
\end{align}
We decompose $A$ into block form according to the indices ${\cal I}$ and ${\cal J}$:
$$
	A  = 
\begin{pmatrix}
	H &  & G &  &  \\
		     &  &   &  &  \\
	0 &  & W &  &  \\
	  &  &   &  & 
\end{pmatrix},
$$
where $H:= A_{{\cal I}, {\cal J}}$, $G:= A_{{\cal I}, [n]\setminus {\cal J}}$, and $W:= A_{[n] \setminus {\cal I}, [n] \setminus {\cal J}}$. 

\medskip

Now, we will introduce some high probability events. Let $\phi$ denote a constant which we will be chosen later, and set  
\begin{align*}
	\Omega_1
= 
	\Omega_1(\phi)
:=	\bigg\{
		A\,:\, &
		\forall s \in \big[1, \frac{1}{2}pn \big],\, 
 		\left|  \mathcal{J}(s) \right| < 
			\underbrace{\log^2( n ) 
			\left( \frac{e  pn  } {t_2} \right)^ {t_2}
			\exp( -  pn  ) n}_{\text{see Proposition } \ref{prop: L_sboundSmallp}} \\
		&	
		 \phantom{AAA AAA AAA AAA AAA AAA}
		\quad \text{and}
		\quad \mathcal{J}( \phi  pn ) = [n]
	\bigg\}.
\end{align*}
Further, let 
$$
	\Omega_{\rm row}
=
	\Omega_{\rm row}(\phi) 
:= 
	\big\{ 
		A\,:\, {\cal J}_{A^\top}(\phi pn) = [n]
	\big\}.
$$

\begin{lemma} \label{lem: Omega_1}
	There exists $\phi >1 $ 
	such that  
	\begin{align*}
		\max\{ \mathbb{P} \left\{ \Omega_1^c \right\} ,\, \mathbb{P} \left\{ \Omega_{\rm row}^c  \right\} \}
		=
	o(\Prob\{ \Omega_{\rm RC}^c(\beta)\}).
		\end{align*}
	
\end{lemma}

\begin{proof}
	First of all, by \eqref{eq:BnmUp}, for $j\in [n]$ 
	we have
	$ \mathbb{P} \left\{ j \notin  \mathcal{J}(\phi pn ) 
		\right\}  
	\le 
		\exp(- 3  pn )$
	if $\phi >1$ is sufficiently large. 	
	Together with the union bound 
	argument, we conclude 
	$ \mathbb{P} \left\{  
		 \mathcal{J}(\phi  pn ) \neq [n] 
	\right\} 
	\le \exp(-3  pn  ) $.
	By Proposition \ref{prop: L_sboundSmallp} 	
	and the union bound argument, 
	\begin{align*}
		\mathbb{P} \left\{ \Omega_1^c \right\}  
	\le 
		\exp(- 2  pn  ).
	\end{align*}	

	Due to $  A  $ and $  A ^\top $
	have the same distribution, 
	the estimate for $\Omega_{\rm row}^c$ follows from the 
	above computations. 

\end{proof}

Next, we will describe some typical subevents of $\Omega_1$ 
corresponding to $H, G$, and $W$.

Let 
\begin{align*}
	\Omega_H 
:=& 
	\Big\{ A\,:\,
	 \forall j_1,j_2 \in {\cal J}, \,
	 \opn{supp}({\bf C}_{j_1}(  A ))
		\cap \opn{supp} ({\bf C}_{j_2}(  A ) ) =\emptyset
	\Big\}, \quad \text{and}\\
	\Omega_G
:=&
	\Big\{ A\,:\,
	 \forall j \in [n] \setminus {\cal J}, \,
	 |\opn{supp}({\bf C}_{j}(  A ))| \le 2
	\Big\}. \\
\end{align*}
As for $W$, recall the definition from \eqref{eq: expansionDefine}: For $J_1 \subseteq J_2 \subseteq [n] \setminus {\cal J}$, 
\begin{align*}
	I_{ W }(J_1,J_2) := \Big\{ i \in [n] \backslash \mathcal{I}\,:\,
		\exists j_0 \mbox{ s.t. }
		a_{ij_0}=1 \mbox{ and } a_{ij}=0 
		\mbox{ for } j\in J_2 \backslash \{j_0\}
	\Big\}.
\end{align*}
Let $$ \Omega_{  W  }$$ be the event that 
\begin{itemize}
	\item For $1  \le n_1 \le  
		 \big\lceil \exp \big( \frac{ \log n  }{ 
		\log^2(  \log n   ) } \big) \big\rceil
		$ and 
		any subsets $J_1 \subset J_2 \subset 
		[n] \backslash \mathcal{J}$
		with $|J_1|=n_1$ and $|J_2| \le 1000 n_1$, 
		$ \left| I_{ W }(J_1,J_2) \right|  > s_0 /8 $.

	\item For $ \frac{1}{10}
		\big\lceil \exp \big( \frac{ \log n  }{ 
		\log^2(  \log n   ) } \big) \big\rceil
		\le n_1 \le \lfloor \frac{1}{ \gamma  p}
		\rfloor $, 
		and any subsets $J_1 \subset J_2 \subset 
		[n] \backslash \mathcal{J}$
		with $|J_1|=n_1$ and $|J_2| \le \max \{
			\lfloor \frac{1}{ \gamma  p} \rfloor ,\,
			 \frac{  pn  }{ \log^{ 3} ( pn )} n_1
		\}$, 
		$ \left| I_{ W }(J_1,J_2) \right| > s_0 /8 $.
\end{itemize}
Before we proceed, let's clarify our definition of $\Omega_W$. Notice that, the larger the ratio $\frac{|J_2|}{|J_1|}$, the larger the collection of 'steep-type' vectors for which we can bound $|Ax|$ away from zero. However, as $\frac{|J_2|}{|J_1|}$ increases, we also encounter a larger union bound to contend with, particularly when $|J_1|$ is relatively small. Due to these technical considerations, we introduce a truncation based on the size parameter $n_1 = |J_1|$.
In the end, let 
\begin{align} \label{eq: eventSteepExpansion}
\Omega_{\ref{thm: SteepExpansion}} = \Omega_1 
\cap ( \Omega_{ H }
	\cap \Omega_{  G  }
	\cap \Omega_{  W  } ) \cap \Omega_{\rm row} ,
\end{align}
which will be the event 
described in Theorem \ref{thm: SteepExpansion}. 
To proof this theorem, we shall first estimate 
$ \mathbb{P}( \Omega_{\ref{thm: SteepExpansion} }^c)$:

\begin{lemma} \label{lem: Omega_2}
	With the assumption that $ \frac{\log n }{n} \le p \le \big( 1 + \frac{1}{2\beta}\big) \frac{\log n }{n}$. 

	For any  
	$\phi >1$, we have
	\begin{align*}
		\mathbb{P} \left\{ \Omega^c_{\ref{thm: SteepExpansion} }
		\right\}  
	\le 	
		\exp\big(  4 s_0 \log(\log n ) - \log n 
		\big)
	\overset{\eqref{eq: J0smallp01}}{=}
	o(\Prob\{ \Omega_{\rm RC}^c(\beta)\}).
	\end{align*}	
\end{lemma}

\subsection{ Probability Estimate for $\Omega_{\ref{thm: SteepExpansion} }^c$. }

To prove Lemma \ref{lem: Omega_2}, 
{ 
	we need to estimate the probabilities of events related to 
}
$\Omega_{H}$, 
$\Omega_{G}$, and $\Omega_{W}$. 

\begin{proposition} \label{prop: Omega_H}
	The following probablity bound holds:
	\begin{align*}
		\mathbb{P} \left\{ \Omega_{H}^c \,|\,
			\Omega_1   
		\right\} 
	\le 	
		\exp\big( 
			3s_0 \log(\log n )
		-	\log n 
		\big)
	= 
		o(\Prob\{ \Omega_{\rm RC}^c(\beta)\}).
	\end{align*}	
\end{proposition}

\begin{proof}	
	{Recall from \eqref{eq: supportProfile}, for
	$b = (b_1,\dots, b_{n}) 
		\in \left\{ 0,1,\dots,n\right\}^{n}$,
	$\Omega_{b}( A )$ is the event that  
	\begin{align} \label{eq: defOmegabA}
	 \forall i \in [n],\, |\mbox{supp}({\bf C}_i( A ))| = b_i. 
	\end{align}
	} 

	Observe that $\Omega_1$ can be partitioned into 
	subevents of the form $\Omega_{ b }(  A )$. 

	{Now we condition on an event }
	$ \Omega_{ b }(  A ) \subset \Omega_1$.
	Notice that, the columns are still jointly independent. 
	The set $\mathcal{J}$ is determined by $ b $. 
	And for $j \in [ n]$, 
	$ \mbox{supp}({\bf C}_j(A)) $ is uniformly chosen among subsets 
	of $[n]$ with size $b_j$. 

	For $j_1, j_2 \in \mathcal{J}$, by 
	the Hypergeometric distribution tail \eqref{eq:HyperGeomTail} 
	{from Proposition \ref{prop: HyperGeom}}
	, 
	\begin{align*}
		\mathbb{P} \left\{ 
			\left| 
				\opn{supp}({\bf C}_{j_1}(  A )) \cap 
				\opn{supp}({\bf C}_{j_2}(  A ))	
			\right| \ge 1 \,\vert\, \Omega_{b}(A)
		 \right\} 
	\le 
		C_{\rm hg} \frac{3b_{j_1}b_{j_2}}{n}
	\le
		3C_{\rm{hg}} \frac{ s_0^2 }{ n },
	\end{align*}
	{where $C_{\rm{hg}} >1 $ is an universal constant appeared in 
	Proposition \ref{prop: HyperGeom}.}

	As a subevent of $\Omega_1$, 
	\begin{align*}
		| \mathcal{J} | = |  \mathcal{J}(  s_0 ) |
	\le
		\log^2( n ) 
			\left( \frac{e  pn  }{  s_0 } 
			\right)^{  s_0  } 
			\exp( - pn ) n
	 \le 
		\log^2( n ) 
			\left( \frac{e  pn  }{  s_0 } 
			\right)^{  s_0  }
	\end{align*}
	where the last inequality  {holds} due to $  pn  \ge \log n $. 
	Applying the union bound we have 
	\begin{align*}
			&\mathbb{P} \left\{ 
				\Omega_{\mathcal{J}} \, |\, 
				\Omega_{ b }(  A )
			 \right\} 
	\le 
		{ 
			\log^2( n ) 
			\left( \frac{e  pn  }{  s_0 } 
			\right)^{  s_0  }
			\choose 2 
		}	
		3C_{\rm{hg}}\frac{  s_0^2 }{ n } \\
	&\hspace{2cm} \le
		\log^4(n) 
		\big( \frac{ e  pn  }{  s_0  } \big)^{ 2 s_0  }
		3C_{\rm{hg}}	\frac{  s_0 ^2 }{n}
	\le 
		(\log( n))^{ 3 s_0  } n^{-1}.
	\end{align*}
	Since this inequality holds for each 
	$\Omega_{ b }(  A ) \subset \Omega_1$ and 
	$\Omega_1$ can be partitioned into disjoint union of $\Omega_{b}(A)$, the proof follows.  
\end{proof}

\begin{lemma} \label{lem: Omega_G}
	For any $\phi>1$,the following probability bound holds:
	\begin{align*}
		\mathbb{P} \left\{ \Omega_{ G }^c 
			\,| \,	\Omega_1
		\right\} 
	\le 	
		n^{-1}
	\overset{\eqref{eq: J0smallp01}}{=}
		o(\Prob\{ \Omega_{\rm RC}^c(\beta)\})
	\end{align*}	
	when $n$ is sufficiently large. 
\end{lemma}

{
	For $  b  \in \{0,1,\dots, n\}^{ n}$ and 
	$ I \subset [ n]$, let 
	\begin{align} \label{eq: defOmegabI}
		 \Omega_{ b , I}
		\mbox{ be the subevent of } \Omega_{ b } 
		\mbox{ (described in \eqref{eq: defOmegabA}) that } 
	\end{align}
	\begin{enumerate}
	\item	$\forall j \in [ n],\,  \nonumber
		| \mbox{supp}( {\bf C}_j(  A ))| = b_j, \mbox{ and }$
	\item 	 $\mathcal{I} = I, \mbox{ where }  
		\mathcal{I} = \left\{ 
		i\in [ n]\, :\, \exists j 
		\in \mathcal{J} \mbox{s.t.}\, a_{ij}=1
		 \right\}.$
	\end{enumerate}
}

\begin{proof}
	Observe that, $\Omega_1$ can be partitioned into subevents 
	of the form $\Omega_{b,  I }$. From now on, 
	we fix a subevent $\Omega_{b,  I } \subset \Omega_1$.
	
	Conditioning on $\Omega_{b, I}$, $\{ {\bf C}_j(  A )\}_{j \in [n] \backslash \mathcal{J}}$ 
	are independent. 
	For each $j \in [ n] \backslash \mathcal{J}$,
	the expected size of 
	$ \mathcal{I} \cap \opn{supp}({\bf C}_j(A)) $ is 
	$ \frac{|\mathcal{I}| |\opn{supp}({\bf C}_j(A))|}{n}$.
	As a subevent of $\Omega_1$, both 
	$ | I | \le \phi  pn  |\mathcal{J}| $ 
	and $ |\mbox{supp}({\bf C}_j(A))| \le  \phi pn $ 
	are of order in polynomial of $ \log n $.
	By the Hypergeometric distribution tail 
	\eqref{eq:HyperGeomTail} {from Proposition \ref{prop: HyperGeom}},
	for $j \in [n] \backslash \mathcal{J}$,
	\begin{align} \label{eq: Omega_W Hyper}
		\mathbb{P} \left\{ \left| \mathcal{I} 
			\cap \mbox{supp}({\bf C}_j(A)) \right| \ge 
			l \, | \, \Omega_{b, I  }( A )
		\right\} 
	\le	
		C_{\rm hg} \left( \frac{3|\mathcal{I}||\mbox{supp}({\bf C}_j(A))|
			}{nl} \right)^l 
	\le  
		n^{-0.9 l}.
	\end{align}	
	Setting $l=3$ and taking the union bound over all 
	columns, we obtain
	\begin{align*}
		\mathbb{P} \left\{ 
			\exists j \in [ n] \backslash \mathcal{J} , \, 
			\left| \mathcal{I} 
			\cap \mbox{supp}({\bf C}_j(A)) \right| \ge 
			3 \, | \, \Omega_{b,  I }
			( A )
		\right\} 
	{\le 
		n \cdot n^{-2.7} }
	\le 	n^{-1}.
	\end{align*}

	Since this bound does not depend on the choice of 
	$\Omega_{ b ,  I  }(  A ) \subset \Omega_1$, 
	the proof is completed. 
\end{proof}

\begin{lemma} \label{lem: Omega_W} For any 
	$\phi >1$,  if $ s_0 \ge 1$ 
	and $\gamma >1$
	are large enough, we have
	\begin{align*}
		\mathbb{P} \left\{ \Omega_{ W }^c \,|\,
			\Omega_1 \cap \Omega_{ G }
		\right\}  \le n^{-2}
\overset{\eqref{eq: J0smallp01}}{=}	
o(\Prob\{ \Omega_{\rm RC}^c(\beta)\}). 
	\end{align*}	
\end{lemma}
\subsubsection{Proof of Lemma \ref{lem: Omega_W}}
Given a pair $b',b'' \in [0,n]^n$ and $I \subseteq [n]$, 
we define $\Omega_{b',I,b''}(A)$ to be the event that 
the event that
\begin{enumerate}
	\item $I = {\cal I}(A)$
	\item $\forall j \in [n]$, $| \opn{supp}( {\bf C}_j(A))| = b_j'$ and  
	$| \opn{supp}( {\bf C}_{j}(A) \cap I ) | = b''_j$.  
\end{enumerate}

Notice that, the event $\Omega_1 \cap \Omega_G$ can be partitioned into disjoint union of $\Omega_{b',I,b''}$. 

From now on, we fix a subevent $\Omega_{b',I,b''} \subseteq \Omega_1 \cap \Omega_G$. 

For a given pair of subsets $J_1\subset J_2 \subseteq [n] \backslash {\cal J}$, we want to establish a lower bound on the size of 
\begin{align*}
	&I_W(J_1,J_2) \\
:=& 
	\Big\{ i \in [n] \backslash I\,:\,
		\exists j_0 \in [n] \setminus {\cal J} \mbox{ s.t. }
		a_{ij_0}=1 \mbox{ and } a_{ij}=0 
		\mbox{ for } j\in J_2 \backslash \{j_0\}
	\Big\},
\end{align*}
for every pair of $J_1 \subset J_2 \subseteq [n] \backslash {\cal J}$ with a given size profile $(n_1,n_2)$ such that $|J_1|=n_1$ and $|J_2|=n_2$.  For each pair $(n_1,n_2)$, the number of pairs $J_1 \subseteq J_2 \subseteq [n] \backslash {\cal J}$ with the given size profile can be coarsely bounded by 
\begin{align} \label{eq: n2n1size}
	{ n_2 \choose n_1} {n \choose n_2} 
\le 
	2^{n_2} \left(\frac{en}{n_2}\right)^{n_2}
\le 
	\exp\left(
		\log\left( \frac{2en}{n_2}\right) n_2
	\right).
\end{align}

Our goal is provide an uniform upper estimate for $\Prob\{ I_W(J_1,J_2)\,\vert\, \Omega_{b',I,b''}\}$ for every pair $J_1\subseteq J_2 \subseteq [n] \setminus {\cal J}$ with $|J_1|=n_1$ and $|J_2|=n_2$ using Lemma \ref{lem:I(J_1,J_2)}. Then, we will apply the union bound argument over all possible $(J_1,J_2)$ and all possible $(n_1,n_2)$ to derive the estimate for $\Prob\{ \Omega_W\,\vert\, \Omega_1 \cap \Omega_G\}$. 

Before we proceed, we need a coarse lower bound on the size of $[n] \backslash I$:

\begin{lemma} \label{lem: nminusIsize}
For each $\Omega_{b',I,b''}(A) \subseteq \Omega_1 \cap \Omega_G$, if $n$ is sufficiently large,   
$$
	\big| [n] \setminus I\big| = n - |I| \ge n/2. 
$$	
\end{lemma}
\begin{proof}
Within the event $\Omega_{b',  I , b''}(  A ) \subseteq \Omega_1 \cap \Omega_G$, 
$$ I = {\cal I} = \bigcup_{j \in {\cal J}} \opn{supp}({\bf C}_j(A)).$$
Given that $s_0$ is an universal constant. If $n$ is sufficiently large, 
\begin{align*}
	|I| 
\le& 
	\sum_{ j \in {\cal J}} |\opn{supp}({\bf C}_j(A))|
\overset{\Omega_1(A)}{\le}
	\log^2(n) \Big( \frac{epn}{s_0}\Big)^{s_0} \exp(-pn)n  \cdot s_0 \\
& \phantom{AAA AAA AAA AAA AAA}
\overset{ pn \ge \log n }
= O({\rm polylog}(n)).
\end{align*}
Therefore, the proof follows. 
\end{proof}

Now, we are ready to derive Lemma \ref{lem: Omega_W}.
\begin{proof}[Proof of Lemma \ref{lem: Omega_W}]

Conditioning on an subevent $\Omega_{b',I,b''} \subseteq \Omega_1 \cap \Omega_G$. For $j \in [n] \backslash {\cal J}$, the support of ${\bf C}_j(W)$ is uniformly distributed among subsets $[n] \backslash I$ of size $b''_j$. Hence, we may apply Lemma \ref{lem:I(J_1,J_2)} with $ B =  W $, $b = b''|_{[n] \backslash J}$, 
$$
	J_1' = \big\{ j \in J_1\,:\, |\opn{supp}({\bf C}_j(W))| \ge r \big\}
\mbox{ and } 
	J_2' = J_2,
$$
for some $r \ge \min_{j \in [n] \setminus {\cal J} } b_j''$. 
In order to apply the lemma, we need to show 
\begin{enumerate}
	\item $\frac{2\|b''\|_\infty^2}{ n - |I|} \le 1$, and
	\item $r \ge \frac{24\|b\|^2_\infty}{n-|I|} |J_2|$.
\end{enumerate}
Let us establish the first condition since it does not depend on the choice of $r$. By Lemma \ref{lem: nminusIsize} and 
$$
	\|b''\|_\infty \overset{\Omega_1}{\le} \phi pn,
$$
where $\phi$ is a constant which depends only on $\beta$, we conclude that 
\begin{align} \label{eq: omegaW00}
	\frac{2\|b''\|_\infty^2}{ n - |I|}
=
	\frac{4 \phi^2 \log^2(n)}{n}
\le 
	1
\end{align}
when $n$ is sufficiently large. 
As for the second condition, we need to verify it and apply the Lemma case by case, due to the value of $r$ depends on the given pair $(n_1,n_2)$:

\underline{Case 1:} $ 1 \le n_1 \le  \lceil  \exp( \frac{  \log n   }{ 
	\log^2( \log n  ) } ) \rceil  $ and $ n_2 \le 1000 n_1$.
	Here we simply take 
	$$r:= \min_{[n] \setminus J}b_j'' 
	\overset{\Omega_G}{\ge} b_j' - 2
	\overset{\Omega_1}{\ge} s_0 - 2 \ge s_0/2.$$ 
	In particular, we have 
	$$J'_1 = J_1.$$
	Notice that $|J_2| = n_2 \le 1000 n_1 \le \frac{n}{\log^{10}(n)}$ when $n$ is sufficiently large. Together with \eqref{eq: omegaW00},
	\begin{align*}
		\frac{24 \|b\|_\infty^2 }{n-|I|} |J_2|	
	\le 
		\frac{48\phi^2 \log^2(n)}{n} \cdot \frac{n}{\log^{10}(n)}
	\le 
		1
	\le 
		r, 
	\end{align*}
	which verifies the second condition in Lemma \ref{lem:I(J_1,J_2)}.
	Now, we can apply the lemma to get 
\begin{align}
\nonumber 
	\mathbb{P} \Big\{ 
		|I_W(J_1, J_2)|  <  \frac{ |J_1|s_0}{8} \Big\vert\, 
			\Omega_{b',I,b''}(A) 
	\Big\}
\le&
	\mathbb{P} \Big\{ 
		|I_W(J_1, J_2)|  <  \frac{ |J_1|r}{4} \Big\vert\, 
			\Omega_{b',I,b''}(A) 
	\Big\} \\
\le &
	C^{|J_1|} \exp \left( -  
		\log \left(  \frac{ r (n-|I|) }{ 24 |J_2| \|b\|_\infty^2} 
		\right) \frac{r |J_1|}{8} 
	\right).
\end{align}

From the of $n-|I|$ from Lemma \ref{lem: nminusIsize}, $r \ge s_0/2$, and $pn \le \big(1 + \frac{1}{2\beta}\big) \log n  \le 2\log n $, 
\begin{align*}
	\log \left(  \frac{ r (n-|I|) }{ 24 |J_2| \|b\|_\infty^2} 
			\right)
\ge 
	\log  \left(  \frac{ s_0/2 \cdot  n/2 }{ 24 |J_2| (\phi \log n )^2} 
			\right)
\ge
	\log \left( \frac{n}{|J_2| \log^3(n)} \right)
\end{align*}
and hence, 
\begin{align*}
	    \mathbb{P} \Big\{ 
			|I_W(J_1, J_2)|  <  \frac{ |J_1|s_0}{8} \Big\vert\, 
 			\Omega_{b',I,b''}(A) 
		\Big\}  \le&  
		C^{|J_1|} \exp \left( -  
			\log \left(  \frac{n}{|J_2|\log^3(n)} 
			\right) \frac{r |J_1|}{8} 
		\right) \\
		\le& 
 \exp \left( -  
			\log \left(  \frac{2en}{|J_2|} 
			\right) \frac{s_0 |J_1|}{32} 
		\right),
\end{align*}
where the last inequality follows from the fact that $ \frac{n}{|J_2|} \ge \log^{10}n$. 
Finally, if $s_0$ is chosen large enough, then 
$$
	s_0 |J_1|/32 = \frac{s_0}{32} n_1 \ge 2000 n_1 \ge 2n_2. 
$$ 
Now we apply the union bound over all pairs of $J_1 \subseteq J_2 \subseteq [n] \setminus {\cal J}$ with $|J_1|=n_1$ and $|J_2|=n_2$, together with \eqref{eq: n2n1size} we conclude that 
	\begin{align*}
		& \mathbb{P} \bigg\{  
			\exists J_1 \subset J_2 \subset 
			[ n] \setminus \mathcal{J} 
		\,:\,  \\
		& \phantom{AAA AAA}
			|J_1| = n_1, \, |J_2| = n_2 
			\mbox{ and } |I(J_1, J_2)| < 
			\frac{ |J_1| s_0 }{8}
		\biggm \vert
			\Omega_{b',  I , b''}
			(  A )	 \bigg\} \\
	\le &
		\exp \left( 
			\log \left(  \frac{2e n }{ n_2 } 
			\right) 
			\left( n_2 - \frac{s_0}{32}n_1 
		\right) 
		\right)
	\le 
		\exp \left( - n_2 \right) 
    \le   n^{-1000}.
	\end{align*}

	\underline{Case 2:} $
		\frac{1}{10}	
		\lceil  
			\exp( \frac{  \log n   }{ 
			\log^2( \log n  ) } ) 
		\rceil  
	\le 
		n_1 
	\le 
		\lceil \frac{ n }{ \log^{12}(n) } \rceil
	$ and 
	$	n_2 
	\le 
		  \frac{  \log n  }{ \log^{ 3 }(  \log n   ) } n_1
	$.

	In this case, we set $ r = \frac{ 100  \log n  }{ 
	\log^3(  \log n   ) }$. Comparing to Case 1, the value of $r$ has been increased, while the same estimate for $\frac{24\|b\|^2_\infty}{n-|I|} |J_2|$ is still valid. Therefore, the conditions for applying the Lemma \ref{lem:I(J_1,J_2)} still hold.
	
	On the other hand, in this case, $J_1'$ does not necessary equal $J_1$. To estimate the size of $J_1'$, recall that as we have conditioned on a sub-event of $\Omega_1$, 
	$ \left|  \mathcal{J}(s) \right| < 
			\log^2( n ) 
			\left( \frac{e  pn  } {t_2} \right)^s
	$ for $1 \le  s \le \frac{1}{2}  pn $. It is straightforward to 
	verify that 
	\begin{align*}	
		\left|  \mathcal{J}(r)\right|
	< \frac{1}{10} n_1.
	\end{align*}

	Notice that 
	$$
		J_1' = J_1 \setminus {\cal J}(r), 
	$$
	which in turn implies 
	$$
		|J_1'| \ge n_1 - \frac{1}{10}n_1 \ge \frac{9}{10}n_1. 
	$$

	Now, we apply Lemma \ref{lem:I(J_1,J_2)} to get 
\begin{align}
\nonumber 
	\mathbb{P} \Big\{ 
		|I_W(J_1, J_2)|  <  \frac{ |J_1|s_0}{8} \Big\vert\, 
			\Omega_{b',I,b''}(A) 
	\Big\}
\le&
	\mathbb{P} \Big\{ 
		|I_W(J_1, J_2)|  <  \frac{ |J_1|s_0}{8} \Big\vert\, 
			\Omega_{b',I,b''}(A) 
	\Big\} \\
\nonumber
\le &
	\mathbb{P} \Big\{ 
		|I_W(J_1', J_2)|  <  \frac{ |J_1'|r}{4} \Big\vert\, 
			\Omega_{b',I,b''}(A) 
	\Big\} \\
\nonumber
\le &
	C^{|J_1'|} \exp \left( -  
		\log \left(  \frac{ r (n-|I|) }{ 24 |J_2| \|b\|_\infty^2} 
		\right) \frac{r |J_1'|}{8} 
	\right).
\end{align}
With essentially the same estimate as in Case 1 with the difference that $r \le \log n $ in this time, we can show 
$$
	\log \left(  \frac{ r (n-|I|) }{ 24 |J_2| \|b\|_\infty^2} 
		\right)
\ge 
	\frac{1}{2} \log\Big( \frac{2en}{n_2}\Big),
$$ 
and hence, 
\begin{align}
\nonumber 
	\mathbb{P} \Big\{ 
		|I_W(J_1, J_2)|  <  \frac{ |J_1|s_0}{8} \Big\vert\, 
			\Omega_{b',I,b''}(A) 
	\Big\}
\le&
	\exp \left( -  
		\log \left( \frac{2en}{n_2} 
		\right) \frac{100\log n }{\log^3(\log n )} \frac{n_1}{32}. 
	\right).
\end{align}
Similarly, here we rely on
$$
	\frac{100\log n }{\log^3(\log n )} \frac{n_1}{32} 
\ge 
	2 n_2
$$
and \eqref{eq: n2n1size} to estimate the union bound to conclude  
\begin{align*}
		& \mathbb{P} \bigg\{  
			\exists J_1 \subset J_2 \subset 
			[ n] \setminus \mathcal{J} 
		\,:\,  \\
		& \phantom{AAA AAA}
			|J_1| = n_1, \, |J_2| = n_2 
			\mbox{ and } |I(J_1, J_2)| < 
			\frac{ |J_1| s_0 }{8}
		\biggm \vert
			\Omega_{b',  I , b''}
			(  A )	 \bigg\} \\
	\le &
		\exp \left( - n_2 \right) 
    \le   n^{-1000}.
\end{align*}	

\underline{Case 3:}
$ 
	\lceil \frac{ n }{ \log^{12}(n) } \rceil
		\le n_1 \le 
		\lfloor \frac{1}{ \gamma p} \rfloor $ and 
	$n_2 \le \max \Big\{ 
		\lfloor \frac{1}{ \gamma p} \rfloor, \,
		 \frac{\log n }{\log^3(\log n )} n_1
	\Big\}$.

	In this case, we set  $r = \lfloor \frac{1}{2}pn \rfloor$. 

	Since the upper bound for $|J_2|=n_2$ is also changed, we need to verify the second condition in Lemma \ref{lem:I(J_1,J_2)}.

	Now, by \eqref{eq: omegaW00}, we have 
	\begin{align*}
		\frac{2\|b''\|_\infty^2}{ n - |I|} |J_2|
	\le 
		\frac{4 \phi^2 \log^2(n)}{n} \frac{1}{\gamma p}
	\le 
		\frac{4 \phi^2 \log n }{\gamma}.
	\end{align*}
	On the other hand, 
	$$
		r = \lfloor \frac{1}{2} pn \rfloor \ge \frac{1}{3}\log n . 
	$$

	Therefore, if we impose {\bf the first assumption} on $\gamma$ that $ 12 \phi^2 \le \gamma$, then the second condition of Lemma \ref{lem:I(J_1,J_2)}
	holds.  

	Before we apply the lemma, let us estimate the size of $|J_1'|$. Recall that 
	\begin{align*}
		\left|  \mathcal{J}(r) \right| 
	\overset{\Omega_1}{\le}
		\log^2( n) 
		\exp\left( 
			 \frac{\log(2e)}{2}   pn 
			 -  pn \right)  n
	\le
		\exp\left( 
			 -0.1  pn 
		 \right) n
	\le 
		n^{0.9}.
	\end{align*}
	Therefore, 
	$n_1 > 
	100 | \mathcal{J}(r)   
	| $, which in term implies
	$$
		|J_1'| = |J_1 \setminus {\cal J}(r) |
		\ge 0.99 n_1. 
	$$	
	Now, we apply the Lemma to get 	
	\begin{align}
	\nonumber 
		\mathbb{P} \Big\{ 
			|I_W(J_1, J_2)|  <  \frac{ |J_1|s_0}{8} \Big\vert\, 
				\Omega_{b',I,b''}(A) 
		\Big\}
	\le &
		\mathbb{P} \Big\{ 
			|I_W(J_1', J_2)|  <  \frac{ |J_1'|r}{4} \Big\vert\, 
				\Omega_{b',I,b''}(A) 
		\Big\} \\
	\nonumber
	\le &
		C^{|J_1'|} \exp \left( -  
			\log \left(  \frac{ r (n-|I|) }{ 24 |J_2| \|b\|_\infty^2} 
			\right) \frac{r |J_1'|}{8} 
		\right).
	\end{align}

	If we impose the {\bf second assumption on } $\gamma$ that $24 e \phi^2 \le \gamma$, then 
	$$
		\log \left(  \frac{ r (n-|I|) }{ 24 |J_2| \|b\|_\infty^2} 
			\right) \ge e,
	$$
	and hence, 
	\begin{align}
	\nonumber 
		\mathbb{P} \Big\{ 
			|I_W(J_1, J_2)|  <  \frac{ |J_1|s_0}{8} \Big\vert\, 
				\Omega_{b',I,b''}(A) 
		\Big\}
	\le &
		\exp \left( - \frac{r n_1}{32} \right)
	\le 
		\exp \Big( - \log^2(n) n_2\Big).
	\end{align}
	Finally, we apply the union bound together with the estimate \eqref{eq: n2n1size}, we conclude that 
	\begin{align*}
		& \mathbb{P} \bigg\{  
			\exists J_1 \subset J_2 \subset 
			[ n] \setminus \mathcal{J} 
		\,:\,  \\
			& \phantom{AAA }|J_1| = n_1, \, |J_2| = n_2 
			\mbox{ and } |I(J_1, J_2)| < 
			\frac{ |J_1| s_0 }{8}
		\biggm \vert
			\Omega_{b',  I , b''}
			(  A )	 \bigg\} 
    \le   n^{-1000}.
	\end{align*}

	Now, since we have treated all possible cases of the pair $(n_1,n_2)$, and for each case we obtain 
	\begin{align*}
		& \mathbb{P} \bigg\{  
			\exists J_1 \subset J_2 \subset 
			[ n] \setminus \mathcal{J} 
		\,:\,  \\
	& \phantom{AAA }
			|J_1| = n_1, \, |J_2| = n_2 
			\mbox{ and } |I(J_1, J_2)| < 
			\frac{ |J_1| s_0 }{8}
		\biggm \vert
			\Omega_{b',  I , b''}
			(  A )	 \bigg\} 
    \le   n^{-1000}.
	\end{align*}
	Now we apply the union bound over all pairs $(n_1,n_2)$, we conclude that 
	\begin{align*}
		& \mathbb{P} \left\{  
			\Omega_W
					\biggm \vert
			\Omega_{b',  I , b''}
			(  A )	 \right\} 
    \le   n^2 \cdot n^{-1000}
\overset{\eqref{eq: J0smallp01}}{=}
o(\Prob\{ \Omega_{\rm RC}^c(\beta)\}).
	\end{align*}
	And since this bound is valid for all 
	$ \Omega_{b',  I , b'' }(  A )
	\subset \Omega_1 \cap \Omega_{ W }$, 
	the proof is completed. 
	\end{proof}

At this point, we are ready to prove Lemma \ref{lem: Omega_2}:

\begin{proof}
	By definition, 
	\begin{align*}
		\Omega_{\ref{thm: SteepExpansion}}^c 
	& = 
		\Omega_1^c \cup 
		\Omega_{H}^c \cup
		\Omega_{G}^c \cup
		\Omega_{W}^c \cup
		\Omega_{row}^c	 \\
	& = 
		\Omega_1^c \cup 
		(\Omega_{H}^c \cap \Omega_1) \cup
		(\Omega_{G}^c \cap \Omega_1) \cup
		(\Omega_{W}^c \cap 
			(\Omega_1 \cap \Omega_{ G })
		) \cup
		\Omega_{\rm row}^c.\\
	\end{align*}

	By Lemma \ref{lem: Omega_1}, Proposition \ref{prop: Omega_H}, and Lemma \ref{lem: Omega_W}, we have 
	\begin{align*}
		&\max \Big\{ 
			\mathbb{P}\left\{ \Omega_1^c \right\} ,\,
			\mathbb{P} \{ \Omega_{H}^c \,| \, \Omega_1 \},\,
			\mathbb{P} \left\{ \Omega_{G}^c 
				\,|\, \Omega_1  \right\},\,
			\mathbb{P} \left\{ \Omega_{ W }^c \,|\, 
				\Omega_1 \cap \Omega_{ G } \right\} ,\,
			\mathbb{P} \left\{ \Omega_{\rm row}^c \right\} 
		 \Big\} 	\\
	= &
		o(\Prob\{ \Omega_{\rm RC}^c(\beta)\}).
	\end{align*}

	Together with the standard inequality $ \mathbb{P} \left\{ O_1 \cap O_2 \right\}  \le \mathbb{P} \left\{ O_1 \vert O_2 \right\} $ for any probability events $O_1$ and $O_2$ with $\Prob\{ O_2\}>0$, combining these we get 
	\begin{align*}
		\mathbb{P} \left\{  \Omega_{\ref{thm: SteepExpansion}}^c \right\} 
	&= 
o(\Prob\{ \Omega_{\rm RC}^c(\beta)\}).
	\end{align*}	
\end{proof}

\subsection{ Proof of Theorem \ref{thm: SteepExpansion} and Theorem \ref{thm: T_1s}}

Let $P_{\mathcal{J}^c}$ be the orthogonal projection to 
the span of $ \left\{  e_j \right\}_{ 
j \in [n] \backslash \mathcal{J}}$. 

\begin{lemma} \label{lem: Hvector}
	For any $\phi >1$, suppose $  s_0 \ge 1$ and $ \gamma> 1$ 
	are large enough. There exists $C\ge 1$ and $c \in (0,1)$ so that the following holds: Fix any sample $A$ from the event $\Omega_{\ref{thm: SteepExpansion}}$. 
	Suppose a vector $x \in \R^n$ satisfying 
	\begin{enumerate}
		\item $ \opn{supp}(x) \cap {\cal J}(0) = \emptyset$, and
		\item there exists $J \subseteq \cal J$ so that  
			$$ \min_{j \in J} |x_j| > C pn x^*_{6|J|}.$$  
	\end{enumerate}
	Then, either 
	\begin{enumerate}
		\item there exists $i \in {\cal I}$ such that 
		$|(Ax)_i| \ge \frac{1}{4} \min_{j \in J} |x_j|$, or 
		\item the vector $y = P_{\cal J^c} x$ satisfies 
		\begin{align*}
		y^*_{ \max\{ \frac{|J|}{2} , 1\} } & \ge  
		2\phi  pn 
		y^*_{ 6|J| }, &
		y^*_{ \max\{  \frac{|J|}{2} ,  1\} } & \ge  
		c	
		\min_{j \in J } |x_{j}|.
	\end{align*}
	\end{enumerate}

% 	When $ C_{\mathcal{T}_1} > C_{ \ref{lem: Hvector}} \phi$ for a
% 	universal constant $C_{\ref{lem: Hvector}} >1$,  
% 	the following holds: 
% 	Fix a sample of $A$ from $\Omega_{\ref{thm: SteepExpansion}}$. 
% 
% 	Let $x \in \mathbb{R}^n$ be a vector such that 
% 	$ \opn{supp}(x)\cap  \mathcal{J}(0)= \emptyset$ and 
% 	there exists $J \subseteq \mathcal{J}$ satisfying 
% 	$ \min_{j \in J } |x_{j}| > 
% 	 C_{\mathcal{T}_1}  pn  x^*_{6|J|}$.
% 	
% 	Then, either $ \| A_{ \cal I, [n] }x \| > \frac{1}{4}
% 	\min_{j \in J } |x_{j}|$  where 
% 	$A_{ \mathcal{I}(A) [ n]} =  
% 		(a_{ij} )_{ i \in \mathcal{I}(A), j \in [ n]}	
% 	$
% 	or the vector $ y= P_{\mathcal{J}}x$ satisfies 
% 	\begin{align*}
% 		y^*_{ \max\{ \frac{|J|}{2} , 1\} } & \ge  
% 		2\phi  pn 
% 		y^*_{ 6|J| }, &
% 		y^*_{ \max\{  \frac{|J|}{2} ,  1\} } & \ge  
% 		c_{\ref{lem: Hvector}}
% 		\min_{j \in J } |x_{j}|
% 	\end{align*}
% 	where 
% 	$	c_{\ref{lem: Hvector} } >0 $ is a
% 	universal constant. 
\end{lemma}

\begin{proof}
	The values of $ s_0 \ge 1$ and $ \gamma >1$ will be chosen 
	large enough so that Lemma \ref{lem: Omega_2} holds. 
	Let $ J_2 := \sigma_x(6|J|) $.
	Let $$ \cal I(x) = \left\{ i \in \mathcal{I} \,:\,  
	\exists j \in J \mbox{ s.t. } a_{ij}=1 \right\} .$$  
	
	Due to 	
	$ \mbox{supp}(x)\cap  \mathcal{J}_A(0) = \emptyset$, 
	${\bf C}_j( A) \neq \vec{0}$ for $j\in J$. 
	Hence, $|\mathcal{I}(x)| \ge |J|$.
	For $i \in \mathcal{I}(x)$, let $j_0 \in J $ be an index 
	such that $ a_{ij_0} = 1$. 
	Recall that 
	$ \big\{ \mbox{supp}({\bf C}_{j}(A)) \big\}_{j\in \mathcal{J}}$
	are mutually disjoint due to 
	$A$ is a sample $ \Omega_{\ref{thm: SteepExpansion}} \subset 
	\Omega_{H}$
	,  and hence $j_0$ is 
	unique. In particular, $  a_{ij} = 0$ for 
	$j \in \mathcal{J} \backslash \{ j_0 \}$. 
	
	Now we fix $ i \in \mathcal{I}(x)$,
	\begin{align*}
		|\left( A x \right)_i| & = 
		\Big| \sum_{j \in [ n ]} a_{ij} x_j \Big| 
	 \ge 
		|x_{j_0}| 
		- \Big| \sum_{ j \in J_2 \backslash \mathcal{J} } a_{ij}x_j \Big|
		- \Big| \sum_{ j \in [ n] \backslash ( J_2 \cup \mathcal{J} ) }
		a_{ij} x_j \Big|.
	\end{align*}

	As a subevent of $\Omega_{\rm row}$, we know $|\opn{supp}({\rm R}_i(A))| \le \phi pn$. Further, by the assumption of $x$, $|x_j| \le \frac{1}{Cpn}|x_{j_0}|$ for $j \notin J_2$. 
	Combining these two estimates we have 
	\begin{align*}
	\Big| \sum_{ j \in [ n] \backslash ( J_2 \cup \mathcal{J} ) }
		a_{ij} x_j \Big|
	\le 
		\phi pn \cdot \frac{1}{Cpn}|x_{j_0}|
	\le 
		\frac{\phi}{C} |x_{j_0}|, 
	\end{align*}	
	and hence,
	\begin{align*}
		|\left( A x \right)_i| &   
	 \ge 
		\Big( 1- \frac{\phi}{C} \Big) |x_{j_0}| 
		- \Big| \sum_{ j \in J_2 \backslash \mathcal{J} } a_{ij}x_j \Big|.
	\end{align*}

	If $ | \sum_{ j \in J_2 \backslash \mathcal{J} } a_{ij}x_j | 
	\le \frac{1}{4} \min_{j \in J } |x_{j}|	$, then 
	\begin{align*}
		| (Ax)_i | 
	\ge 
		\frac{1}{4} \min_{j \in J } |x_{j}|,
	\end{align*}
	which is the first part of the statement in the Lemma. 

	\medskip

	Now suppose 
	$| \sum_{ j \in J_ 2\backslash \mathcal{J} }  A _{ij}x_j | 
	> \frac{1}{4} \min_{j \in J } |x_{j}|		
	$ for all $ i \in  \mathcal{I}(x)$.
	For $ j \in J_2 \backslash \mathcal{J}$, every column of $ G $ has support size at most $2$ due to $A$ is a sample from $\Omega_{ G }$. Thus,
	\begin{align}
		\label{eq: Hvector00}
		\left| \left\{  ( i , j ) \in \mathcal{I}(x) \times 
			( J_2  \backslash \mathcal{J} ) \, : \,
			a_{ij} = 1
		\right\}  \right|  \le 
		2|J_2| . 
	\end{align}

	Let 
	$$
		{\cal I}'(x) := \Big\{ i \in {\cal I}(x)\,:\,  
			| \left\{  j\in J_2 \backslash \mathcal{J}\,:\,  
	a_{ij} = 1 \right\} | \le \frac{ 2|J_2| }{ \frac{1}{2} |\mathcal{I}(x)| }\Big\}.
	$$	
	We {\bf claim} that $|\cal I'(x)| > \frac{1}{2} |\cal I(x)|$. Suppose the claim fails. Consider the complement of $\cal I'(x)$ is at least half of the size of ${\cal I}(x)$: 
	$|\cal I(x) \setminus \cal I'(x)| \ge \frac{1}{2} |\cal I(x)|$.
	Then, 
	$$
	\left| \left\{  ( i , j ) \in \mathcal{I}'(x) \times 
			( J_2  \backslash \mathcal{J} ) \, : \,
			a_{ij} = 1
		\right\}  \right|
	>
		\frac{1}{2} |\cal I(x)| \cdot  \frac{2|J_2|}{\frac{1}{2}|\cal I(x)|}
	= 
		|J_2|,
	$$
	which contradicts to \eqref{eq: Hvector00}. Therefore, the claim holds. 

	Now, for each $ i \in \mathcal{I}'(x)$, 
	we pick $j_i \in J_2 \backslash \mathcal{J}$
	such that 
	$$
	|x_{j_i}| \ge \frac{1}{ 24 } \cdot \frac{ 1}{ 4} 
	\min_{j \in J } |x_{j}|
	:= c \min_{j \in J } |x_{j}|. $$
	and let $J_1$ be the collection of such $j_i$:
	\begin{align*}
		J_1 := \left\{ j_i \,:\, 
			i \in \mathcal{I}'(x)
		\right\}.
	\end{align*}
	Notice that it is possible that for $i_1 \neq i_2$, 
	$ j_{i_1} = j_{i_2}$. Nevertheless, as a subevent of $\Omega_{ G }$, 
	for $j \in [ n ] \backslash \mathcal{J}$, 
	$$ | \{ i \in \mathcal{I}'(x) \, : \,  j_i = j \} | \le 2.$$
	Therefore, 
	$ |J_1| \ge \max \{ \frac{|J|}{2}, \, 1 \}$.

	Now, let  $y = P_{\mathcal{J}^c}x$. We have $ |y_j| \ge 
	c\min_{j' \in J } |x_{j'}|  $ for 
	$j \in J_1$. And for $ j \notin J_2 $, 
	we have $  C pn |y_j| <  \min_{j' \in J } |x_{j'}|$. 
	Now we conclude 
	$$ y^*_{ \min \{ \frac{|J|}{2} ,\, 1 \} } 
	\ge 2\phi  pn  y^*_{6|J|}$$ when 
	$ C \ge \frac{2\phi}{c} $.
\end{proof}

\begin{proof}[Proof of Theorem \ref{thm: SteepExpansion}] 
	By Lemma 
	\ref{lem: Omega_2}, we have 
	$$
		\mathbb{P} ( \Omega_{\ref{thm: SteepExpansion} }^c ) 
	= 
o(\Prob\{ \Omega_{\rm RC}^c(\beta)\}).
	$$

	Now we fix a sample $A$ from $\Omega_{\ref{thm: SteepExpansion}} \cap \Omega_{\rm RC}(\beta) $. 
	And due to $A$ is a sample of $\Omega_1$ and $\Omega_{row}$, 
	$$ | \mathcal{I}| \le  s_0 |\mathcal{J}|
	\le  \log^{2s_0}(n).$$

	Further, let $C$ denote the constant $C$ introduced in Lemma 
	\ref{lem: Hvector}.
	Let $x \in \cal T^*$ with $\opn{supp}(x) \cap {\cal J}(0) = \emptyset$. We will derive the theorem according to the pair $(n_1,n_2)=(n_1(x),n_2(x))$ such that 
	$$
		x_{n_1}^* \ge Cpn x_{n_2}^*.
	$$

	\underline{Case 1}: $ 1 \le n_1 \le 
	\lceil \exp( \frac{ \log n  }{ 
		\log^2(  \log n   ) } ) \rceil
	$
	and $n_1 < n_2 \le 3 n_1$. 
	
	Suppose $ \left| \sigma_x([n_1]) \cap \mathcal{J} \right| \ge \frac{n_1}{2} $.
	Let 
	$ J = \sigma_x([n_1]) \cap \mathcal{J}$. Then, $n_2 \le 6 |J|$. 
	By Lemma \ref{lem: Hvector}, one of the following is true:
	Either there exists $i \in \cal I$ such that $  |(Ax)_i|  
	> \frac{1}{4} x^*_{n_1}$, or $y = P_{\mathcal{J}^c}x$ 
	satisfies 
	\begin{align*}
		y^*_{ \max \{ \frac{n_1}{ 4 } , 1 \} } & >
		2 \phi  pn  y^*_{ n_2 } & 
		y^*_{ \max \{ \frac{n_1}{ 4 } , 1 \} } & >	
		c_{\ref{lem: Hvector}} x^*_{n_1},
	\end{align*}
	where $c_{\ref{lem: Hvector}}$ is the constant $c$ introduced in Lemma \ref{lem: Hvector}. If the first condition happens, then the theorem follows. 
	
	\medskip 

	Now let us assume the statement about $y$ holds. 
	Let $J_1 = \sigma_y( [ \max \{ \frac{n_1}{ 4 } , 1 \} ] ) \subset 
	[ n ] \backslash \mathcal{J}$ and 
	$J_2 =  \sigma_y( [n_2] ) \subset 
	[ n]\backslash \mathcal{J}$.  
	
	As a subevent of  $ \Omega_{  W  }$, 
	$ I_{ W }(J_1,J_2) > \frac{  s_0 }{8}$. 
	Fix $i \in I_{ W }(J_1,J_2)$ and 
	let $j_0  \in J_1$ be the unique index such that 
	$ a_{ij_0} = 1$. 
	
	Due to $i \in [ n] \backslash \mathcal{I}$, 
	we have $a_{ij} = 0$ for $ j \in \mathcal{J}$ and thus 
	$ (Ax)_i = (Ay)_i$. 

	Next,
	\begin{align*}
		 &| (Ay)_i |  
	=
		| a_{ij_0} y_{j_0}  + \sum_{j \notin J_2} a_{ij}y_j | 
	\ge  
		|y_{\max \{ \frac{n_1}{4} , 1 \}}| -  |\sum_{j \notin J_2 } a_{ij}y_j | \\
	& \hspace{1cm} \ge 
		y^*_{\max \{ \frac{n_1}{4} , 1 \}}
		- \frac{\phi  pn }{2\phi  pn } y^*_{n_1} 
	\ge
	 	\frac{1}{2} c_{\ref{lem: Hvector}} x^*_{n_1}. 
	\end{align*}
	where in the second inequality we used $ | \mbox{supp} ({\bf R}_i(A) ) | 
	\le \phi  pn $ due to $A$ is a sample from $\Omega_{row}$. 
	Therefore, we conclude that 
	\begin{align} \label{eq: (Ax)_i count}
		| \left\{  i \in [n] \,:\,
			| (Ax)_i | > \frac{1}{2} c_{\ref{lem: Hvector}} x^*_{n_1}. 
		\right\} | > \frac{ s_0}{8}.
	\end{align}	
	
	In the case 
	$ \left| \sigma_x([n_1]) \cap \mathcal{J} \right| \le \frac{n_1}{2} $, let 
	$ y = P_{\mathcal{J}}x$. Then, we have 
	\begin{align*}
		y^*_{ \lceil \frac{n_1}{2} \rceil } 
	& \ge 
		 C pn 
		y^*_{ n_2  }, & 
		y^*_{ \lceil \frac{n_1}{2} \rceil } 
	& \ge 
		x^*_{ n_1}. 
	\end{align*}
	Applying the same argument we obtain \eqref{eq: (Ax)_i count}.

	\medskip

	\underline{Case 2} $
		\frac{1}{10} \lceil \exp( \frac{ \log n  }{ 
		\log^2(  \log n   ) } ) \rceil
		\le  n_1 \le  
		\frac{1}{ \gamma  p } $ 
	and $
		n_1 < n_2 \le 
		\min \{
			\frac{  pn }{  \log^3 ( pn )}n_1, 
			\frac{1}{ \gamma   p } 
		\} $. 
	In this case, we simply consider $y = P_{\mathcal{J}^c} x$. 
	Due to $|\mathcal{J} | \le 
		\frac{1}{2} 	\lceil \exp( \frac{ pn }{ 
		\log^2(  pn  ) } ) \rceil $,
	we have 
	$y^*_{n_1 / 2} >  C  pn  y^*_{n_2}$ 
	and $y^*_{n_1/2 } \ge x^*_{n_1}$. 
	Again, by a similar approach we obtain \eqref{eq: (Ax)_i count}.
\end{proof}
\begin{proof}[Proof of Theorem \ref{thm: T_1s}]
	Let $\Omega( A )$ and 
	$\Omega(  A ^\top)$ be the events described in 
	Theorem \ref{thm: SteepExpansion} for $ A $ and 
	$  A ^\top$ respectively
	with $ s_0 > 8 \beta$. 	

	We will show that $\Omega_{\cal T^*}$ is an subevent of   
	$\Omega_{\rm RC}(\beta) \cup \Omega(  A ^\top) \cup 
	\Omega( A )
	$.  
	From now on, we fix a sample $A$ in  $ \Omega_{\rm RC}(\beta) \cup \Omega(  A ^\top) \cup 
	\Omega( A )$. We have the following two sets: 
	\begin{align*}
		\mathcal{I}_A =& \left\{ 
				i \in [ n] \, :\, 
				\exists j \in  \mathcal{J}_A(  s_0 ) 
				\mbox{ s.t. } a_{ij} = 1 
		 \right\}, \\	
		\mathcal{I}_{A^\top}  =& \left\{ 
				i \in [ n] \, :\, 
				\exists j \in  \mathcal{J}_{A^\top} (  s_0 ) 
				\mbox{ s.t. } a_{ij} = 1 
		 \right\}.
	\end{align*}
	Clearly, we have 
	$$
		\cal I_A \cap J_{A^\top}(0) = \cal I_{A^\top} \cap J_A(0) = \emptyset. 
	$$
	Now, we pick $I$ and $J$ with $|I|= |J| = n - \beta + 1$ in the following way:
	\begin{align*}
		\mathcal{I}_A \subset &I \subset [ n ] \backslash 
		 \mathcal{J}_{A^\top}(0), \\
		\mathcal{I}_{A^\top} \subset &J \subset [ n ] \backslash 
		 \mathcal{J}_{A}(0).
	\end{align*}
	Such $I$ and $J$ exists because $| \mathcal{J}_A(0)| < \beta$ 
	due to $\Omega_{\rm RC}(\beta)$ and  
	\hh{$| \mathcal{I}(A) |, | \mathcal{I}(A^\top) | <
		(\log( n))^{ 2  s_0}
	$} due to $\Omega( A )$ and $\Omega( A ^\top)$.

	Next, let $x \in \mathcal{T}^*$ with $\mbox{supp}(x) \subset J$. 	
	By Theorem \ref{thm: SteepExpansion} and our choice of $s_0$, 
	at least one of the following holds:
	\begin{itemize}
		\item Either $\exists i \in \cal I(A)$ such that $ |(Ax)_i|  
				\ge c x^*_{n_1}$, or
		\item 	$ | \left\{ 
				i \in [ n] \,: \, |(Ax)_i| \ge
			c x^*_{n_1}	
		 \right\}  | > \beta $.
	\end{itemize}
	In the first case, we have 
	\begin{align*}
	  \left\| A_{I,J}x \right\| \ge \left\| A_{\mathcal{I}(A),J}x \right\|  
	  = \left\| A_{\mathcal{I}(A), [n]}x \right\|  
	\ge cx^*_{n_1}
	\end{align*}
	due to $ I(\mathcal{A}) \subset I $ and $\mbox{supp}(x) \subset J$. 

	When second condition holds, there exists 
	$$i_0 \in I \cap \left\{ 
				i \in [ n] \,: \, |(Ax)_i| \ge
			n^{-1 + o_{n}(1)}\left\| x \right\| 
		 \right\} $$
	since  $ |I| + | \left\{ 
				i \in [ n] \,: \, |(Ax)_i| \ge
			n^{-1 + o_{n}(1)}\left\| x \right\| 
		 \right\}  | > n$. 
		 Then, 
	\begin{align*}
		\left\| A_{I,J}x \right\| 
	\ge 
		|(A_{I,J}x)_{i_0}| 
	=
		|(Ax)_{i_0}|
	\ge 
		c x^*_{n_1}	,
	\end{align*}
	where the equality holds due to $ \mbox{supp}(x) \subset J$. By symmetry, the same argument shows 
	\begin{align*}
	\left\| A^\top_{I,J}x \right\| 
	\ge 
		c x_{n_1}^*
	\end{align*}
	for $x \in \mathcal{T}^*$ with $\mbox{supp}(x) \subset I$. Hence, we conclude 
	that $\Omega( \mathcal{T}_1) 
	\subset  \Omega_{\rm RC}(\beta) \cup \Omega(  A ^\top) \cup 
	\Omega( A )
	$.	
	\end{proof}

\bigskip

\section{ Typical $ A $ for $ p \ge 
(  1 + \frac{ 1 }{ 2\beta } )\frac{ \log( n ) }{ n   }$. }
The objective and the structure of this section is similar to that from the previous section, except some of the technical parts are different. 
\label{sec: T*+}
For a given pair of constants $C \ge 1$ and $c \in (0,1)$, let  
	\begin{align*}
	&	\cal T^*  = \cal T^*(C,\gamma)
	:=&  \\
	 &	\bigg\{ x \in \R^n\,:\, 
		x_{n_1}^* \ge Cpn x^*_{n_2} \\ 
	\nonumber
& \phantom{AAA}
		\mbox{ for a pair of }(n_1,n_2) = (n_1(x),n_2(x)) \in \N^2 \mbox{ satisfying either }\\
	\nonumber
	& \phantom{AAA AA} (1)  \quad n_1 = 1,\, n_2 = 2,  \\
	\nonumber
	& \phantom{AAA AA} (2) \quad
		2   \le  n_1 \le \frac{1}{\gamma p}, \mbox{ and }
		 n_1 <  n_2 \le 
		\max \Big\{
			 \frac{1}{ \gamma  {p}}  ,\,
			\frac{  pn  }{ \log^3( pn )} n_1
		\Big\}.
		\bigg\},
	\end{align*}
	and consider the event 
	\begin{align*}
		& \Omega_{\cal T^*}(C,\gamma,c) \\
	:= &\Big\{ 
		\exists I,J \subseteq [n] \mbox{ with } |I| =|J|=n-\beta+1 \mbox{ such that }		\\
		& \phantom{AAA} \forall x \in \cal T^*(C) \mbox{ with } \opn{supp}(x) \subseteq [J], \quad
			 \|A_{I,J}x\| \ge c x^*{n_1(x)} \mbox{ and }
		 \\
		& \phantom{AAA} \forall x \in \cal T^*(C) \mbox{ with } \opn{supp}(x) \subseteq [I], \quad
			 \|A_{I,J}^\top x\| \ge c x^*{n_1(x)}
		\phantom{ AAA AA} \Big\} .
	\end{align*}
	We remark that the definition of $\cal T^*$ has been modified comparing to the definition of $\cal T^*$ in the case when $\frac{\log n }{n} \le p \le \Big( 1 + \frac{1}{2\beta}\Big) \frac{\log n }{n}$, while $\Omega_{\cal T^*}(C,\gamma, c)$ has the same definition as before.  
	The main goal is to establish the following theorem. 
\begin{theorem} \label{thm: T_1+}
	For a positive integer $\beta$, 
	there exists $c_\beta, c \in ( 0, 1/2 ),  C>1$, and $\gamma >1$
	so that the following holds: Suppose $ p $ 
	satisfies 
	$$
		\Big( 1 + \frac{1}{2\beta} \Big) 
		\frac{ \log( n ) }{ n } 
	\le
		p 
	\le
		c_\beta.
	$$
	Then, 
	\begin{align*}
		\mathbb{P}\left\{ \Omega^c_{\cal T^*}(C,\gamma, c) \right\}  
	=
		( 1 + o_{ n}(1))
		\mathbb{P} \left\{ \Omega_{\rm RC}^c(\beta) \right\} .
	\end{align*}
\end{theorem}

\begin{theorem} \label{thm: SteepExpansion tau>10} 
	For a fixed positive integer $\beta$, 
	there exists $c_\beta, c \in ( 0, 1/2 )$, 
	$ C \ge 1$, 
	$\gamma \ge 1$, and $\phi_0>0$ so that the following holds. 
	Suppose $ p $ satisfies 
	\begin{align*}
		( 1 + \frac{ 1 }{ 2\beta } ) 
		\frac{ \log( n ) }{ n  }
	\le
		p 	
	\le 
		c_\beta.
	\end{align*}

	There exists an event 
	$ \Omega_{\ref{thm: SteepExpansion tau>10}} $ with
	$$ \mathbb{P} \left\{  \Omega_{\ref{thm: SteepExpansion tau>10}}^c \right\}  
	= o_{ n} ( \mathbb{P}\{\Omega_{\rm RC}^c(\beta)\} )
	$$
	satisfying the following description:
	Fix any sample $A \in \Omega_{\ref{thm: SteepExpansion tau>10}} \cap \Omega_{\rm RC}(\beta)$, 
	we have $  \mathcal{J}_A(\phi_0  pn ) \le \beta$. 
	Further, for $x \in \cal T^*$, 
	\begin{enumerate}
		\item If $n_1(x) = 1$ and $\sigma_x(1) \in {\cal J}(\phi_0 pn)$, then $\exists i \in \opn{supp}({\bf C}_{\sigma_x(1)}(A))$ such that 
	$$
		|(Ax)_i| \ge cx_{n_1}^*.
	$$
		\item Otherwise, 
	\begin{align*}
		|\left\{ i \in [n] \,:\, 
		 | (Ax)_i | > c x_{n_1}^*  
		\right\} | > \beta.
	\end{align*}

	\end{enumerate}	
	
	%If $ x\in \mathcal{T}_{1j}$ with $j >1$
	%or $ x \in \mathcal{T}_{11}$ with $\sigma_x(1) \notin 
	% \mathcal{J}(\phi_0  pn )$. Then, 
	%\begin{align*}
	%	|\left\{ i \in [n] \,:\, 
	%	 | (Ax)_i | > n^{-1 + o_{n}(1)} \left\| x \right\| 
	%	\right\} | > \beta.
	%\end{align*}
	%For any $x \in \mathcal{T}_{11}$ with $\sigma_x(1)\in 
	% \mathcal{J}(\phi_0  pn )$, then 
	%$\forall i \in \mbox{supp}( {\bf C}_{\sigma_x(1)}(A))$ we have
	%\begin{align*}
	%	|\left( A x \right)_i|  
	% \ge 
	%	n^{-1 + o_{n}(1)} \left\| x \right\|. 
	%\end{align*}

\end{theorem}

Similar to the previous section, we will still use $\mathcal{J}, \mathcal{I}$, and 
the corresponding submatrices $H,G,$ and $W$. 
However, the way we define $\mathcal{J}$ and the corresponding high probability events are different. 

For a fixed positive integer $\beta$, 
Let $\phi_0 \in ( 0, \frac{ 1 }{ 2\beta } )$ be a sufficiently small constant and set 
$$\mathcal{J} =   \mathcal{J}(\phi_0  pn).$$

Let $\Omega_1 = \Omega_1(\phi_0,\phi)$ 
be the event that  
$|\mathcal{J}| \le \beta$ and $   \mathcal{J}(\phi  pn)  = [n]$.
Let $\Omega_{\rm row}$ be the event that 
$  \mathcal{J}_{ A ^\top }( \phi  pn ) = [ n ]
$.
\begin{proposition} \label{prop: BadColumnC}
	For a fixed positive integer $\beta$,  when $  p 
	\ge  ( 1 + \frac{ 1 }{ 2\beta } ) \frac{ \log( n ) }{ n  }$,
	there exists $0 < \phi_0 < \frac{ 1 }{ \beta } < 1 < \phi_1$ which depends only on $\beta$ 
	so that the following estimate holds: 
	\begin{align*}
		\mathbb{P} \left\{  \Omega_1^c
			\cup \Omega_{\rm row}^c	
		\right\} 
	 \le
	 \big( ( 1- p  )^{ n } n \big)^{ \beta +\frac{1}{2}}
	=
		o(\Prob\{ \Omega_{\rm RC}^c(\beta)\})	.
	\end{align*}
\end{proposition}

\begin{proof}
	We apply Proposition \ref{prop: L_kBeta} with $c = \frac{ 1}{ 2\beta }$. 
	There exists $\phi'>0$ depending on $\beta$ such that for any $ 0 < \phi_0 < \phi'$
	\begin{align*}
		\mathbb{P} \left\{ 
			| \mathcal{J}_{\phi_0  pn }| > \beta
		 \right\} 	
	\le 
		\big(( 1- p  )^{ n } n \big)^{\beta + 3/4}. 
	\end{align*}
	From now on, we fix $ \phi_0 = \min\{ \phi', \frac{ 1 }{ 2 \beta }\}$. 

	By the Binomial tail bound \eqref{eq:BnmUp} from Proposition \ref{prop: Binomial} and the union bound argument, 	
	if $\phi\ge e^3$, 
	\begin{align*}
		\mathbb{P} \left\{ 
			 \mathcal{J}(   \phi  pn  ) <  n
		 \right\} 	
	\le 
		n \cdot 2
		\left( \frac{e}{\phi} \right)^{\phi  pn }
	\le  
		2 \exp \left( 
			-  2\phi { pn } 
			+\log( n)
		 \right) 
	\le
		2 \exp( - \phi pn). 
	\end{align*}
	By choosing $\phi$ to be at least proportional to $\beta$, 
	we may conclude that 
	$$
		\Prob\{ {\cal J}(\phi pn) < n \}
	\le 
		\exp(-\phi pn)
	\le 
		\exp( - 4pn \beta)
	\le 
		((1-p)^n)^{2\beta}
	\le 
		((1-p)^n n)^{2\beta}.
	$$

	Now taking the union bound of the above events we obtain the statement of this proposition.
\end{proof}

Next, we introduce the high probability event that is related to the submatrices $G$ and $W$. 
Let 
\begin{align*}
	\Omega_G = \Omega_G(\phi_0) := \Big\{ j \in [n] \backslash {\cal J}\,:\, 
			|\opn{supp}( {\bf C}_j (  G 	))	|
	  \le \frac{1}{2}\phi_0  pn 	
		\Big\}
\end{align*}
and 
\begin{align*}
 \Omega_W =& \Omega_W(\phi_0)	\\
:= &\Big\{ 
		\forall n_1 \in [1, 1/\gamma p], 
		\forall J_1 \subseteq J_2 \subseteq [n] \setminus \cal J 
		\mbox{ with } |J_1|=n_1, \\
		& \phantom{AAA} |J_2|\le \max \big\{ 1/\gamma p, \log n  / 3\log^3(\log n ) \big\} \mbox{ and } I_W(J_1,J_2)|>\beta 
		\Big\}.
\end{align*}

In the end, we define the event $\Omega_{\ref{thm: SteepExpansion tau>10}} =  
	\Omega_1 
	\cap \Omega_{ G }  \cap 
	\Omega_{  W  }  \cap \Omega_{\rm row} $.

\begin{lemma} \label{lem: Omega_2 tauLarge}
	For a fixed positive integer $\beta$, and $(\phi_0, \phi)$ 
	from Proposition \ref{prop: BadColumnC} which satisfies 
	$0 < \phi_0 < \frac{ 1}{ 2\beta } < \phi$. There exists 
	$c_\beta>0 $ and $ \gamma \ge 1$  which depends on $\beta, \phi_0$, and 
	$\phi$ such that for 
	\begin{align*}
		( 1 + \frac{ 1}{ 2\beta } )\frac{ \log( n ) }{ n   } 
	\le 
		p 
	\le 
		c_\beta,
	\end{align*}
	the following probability estimate holds:
	\begin{align*}
		\mathbb{P} \left\{ \Omega_{\ref{thm: SteepExpansion tau>10}}^c \right\} 
	    	= 
o(\Prob\{ \Omega_{\rm RC}^c(\beta)\}).
	\end{align*}
\end{lemma}

\subsection{Probability Estimate for $\Omega_{\ref{thm: SteepExpansion tau>10}}^c$.}

The proof of Lemma \ref{lem: Omega_2 tauLarge} 
will be break into the following two lemmas. 
\begin{lemma} 
	\label{lem: Omega_G tauLarge}
	For any fixed positive integer $\beta$ and $0<\phi_0< \frac{ 1 }{ 2\beta }<\phi$, 
	there exists a $c_\beta \in ( 0, \frac{ 1 }{ 2 } )$ which depends on $ \beta, \phi_0$, 
	and $\phi$ such that  for
	\begin{align*}
		( 1 + \frac{ 1}{ 2\beta } )\frac{ \log( n ) }{ n   } 
	\le 
		p 
	\le 
		c_\beta,
	\end{align*}
	the following holds when $ n $ is sufficiently large:
	\begin{align*}
		\mathbb{P} \left\{  \Omega_{  G }^c \, |\, 
			\Omega_1 
		\right\} 		
	=
		o(\Prob\{ \Omega_{\rm RC}^c(\beta)\}).
	\end{align*}
\end{lemma}

\begin{proof}
	The proof is similar to that of Lemma \ref{lem: Omega_W} with 
	the exception that $\beta$ is involved in this case. 

	First, the event $\Omega_1$ can be partitioned into subevents 
	of the form $\Omega_{b,  I }$. Now we
	condition on a subevent $\Omega_{  b ,  I } 
	\subset \Omega_1$. As a subevent of $\Omega_1$, 
	\begin{align*}
		|  I  | \le \phi_0  pn  \beta \le \frac{ n }{ 2 },
	\end{align*}
	where the last inequlity is due to $ \phi_0 \le \frac{ 1 }{ 2\beta }$. 
	By the Hypergeometric distribution tail 
	\eqref{eq:HyperGeomTail} from Proposition \eqref{prop: HyperGeom},
	for $j \in [n] \backslash \mathcal{J}$,
	\begin{align} \label{eq: Omega_W Hyper}
		\mathbb{P} \left\{ \left|  I  
			\cap \opn{supp}({\bf C}_j(  A )) \right| 
		\ge 
			l \, \big \vert  \, \Omega_{  b , I  }( A )
		\right\} 
	\le	
		C_{\rm hg} \left( \frac{3 | I | |\opn{supp}({\bf C}_j(A))|
			}{nl} \right)^l 
	\end{align}	for $l \ge 1$. 
	By setting $l = \frac{1}{2}\phi_0  pn $,
	\begin{align*}
		 \frac{3 | I | |\mbox{supp}({\bf C}_j(A))|
		}{nl}
	\le
		\frac{  3\phi_0 { pn } \beta  \phi_1 { pn }   
		}{ n \frac{ 1 }{ 2 } \phi_0 { pn }   }
	=
		6 \beta \phi_1 p
	\le 
		6 \beta \phi_1 c_\beta.
	\end{align*}

	Taking the union bound over all columns we obtain 
	\begin{align*}
		\mathbb{P} \Big\{ 
			\Omega_{ W }^c
			\biggm\vert \Omega_{  b ,  I }
			( A )
		\Big\} 
    	\le  
   	 	n C_{ \rm hg }	
		\big( 6\beta \phi_1 c_\beta \big)^{ \frac{ 1 }{ 2 }\phi_0 pn }.
    	\end{align*}

	Now, we can set $c_\beta>0$ to be sufficiently small so that  
	\begin{align*}
	 n C_{\rm hg}\big( 6\beta \phi_1 c_\beta \big)^{ \frac{ 1 }{ 2 }\phi_0 pn}
	\le&  n C_{\rm hg} \exp( - 6 \beta pn) 
	\overset{pn \ge \log n }{\le} 	\exp( - 4\beta pn )\\
	& \phantom{AAA AAA}\le 
		((1-p)^n)^{2\beta}
	\overset{\text{Prop.} \ref{prop: zeroColumns}}{=}
	o(\Prob\{ \Omega_{\rm RC}^c(\beta)\}).
	\end{align*} 
\end{proof}

\begin{lemma} \label{lem: Omega_W tauLarge}
	For any fixed positive integer $\beta$, $0 < \phi_0 < \frac{ 1 }{ 2\beta } < \phi$, and $c_\beta \in ( 0, \frac{ 1 }{ 2 } )$, 
	there exists $ \gamma $ which depends on these constants such that 
	for
	\begin{align*}
		\Big( 1 + \frac{ 1 }{ 2\beta } \Big)
		\frac{ \log( n ) }{ n   } 
	\le 
		p 
	\le 
		c_\beta,
	\end{align*}
	the following holds when $ n $ is sufficiently large:
	\begin{align*}
		\mathbb{P} \left\{ \Omega_{ W }^c \,|\,
			\Omega_1 \cap \Omega_{ G }
		\right\}  
	= 	
		o(\Prob\{ \Omega_{\rm RC}^c(\beta)\}).
	\end{align*}	
\end{lemma}

\begin{proof}
	The proof is the same as the proof of Lemma \ref{lem: Omega_W},
	except some technical details. We will only outline the 
	proof and the specific difference. 

	From now on we condition on a subevent 
	$\Omega_{b',  I , b''}(  A )	
	\subset \Omega_1 \cap  \Omega_{G}$.
	Since it is a subevent of $\Omega_1 \cap  \Omega_{G}$,
	for  $j \in [ n] \backslash \mathcal{J}$, 
	\begin{align}	
	  \frac{ \phi_0 pn }{2}  
	\le b_j \le \phi  pn , \mbox{ and  }  |  I | \le \frac{  n  }{ 2 }.
	\end{align}

		Now we fix $n_1$ and $n_2$ (which represents the size 
	of $J_1$ and $J_2$) where $n_1$ satisfies the 
	assumption in the lemma and 
	$n_2 \le \max \{
	\lfloor \frac{1}{ \gamma  \mathfrak{p}} \rfloor ,\,
	\frac{  pn  }{ 3 \log^3( pn )} n_1
	\}$. 
	For a fixed pair of $J_1 \subset J_2$, we want to apply Lemma \ref{lem:I(J_1,J_2)}
	with an appropriate choice of $r$. To apply the lemma,  
	we need to verify that 
	\begin{align*}
		r \ge |J_2| \frac{ 24 \big(\max_{ n \in [ n ] \backslash \mathcal{J} }
		\{ |\mbox{supp}({\bf C}_{j}(  W )|\}\big)^2 }{ n - |  I  |   }.
	\end{align*}
	Notice that 
	\begin{align*}
		|J_2| \frac{ 24 \big(\max_{ n \in [ n ] \backslash \mathcal{J} }
		\{ |\mbox{supp}({\bf C}_{j}(  W )|\}\big)^2 }{ n - |  I  |   }
	\le 
		n_2 \frac{ 24 (  \phi { pn })^2  }{ \frac{  n }{ 2}  }
	\le 
		\frac{ 48 { pn }  }{ \gamma }. 
	\end{align*}

	Let us impose the {\bf first assumption on $\gamma$} that 
	$ \gamma \le \frac{ 96 }{ \phi_0 }$. Then, we can apply Lemma \ref{lem:I(J_1,J_2)} with $ r = \frac{ 1 }{ 2 }\phi_0  pn$. 
	
	\medskip

	By Lemma \ref{lem:I(J_1,J_2)}, we have
	\begin{align*}
		 & \mathbb{P} \left\{  | I_{  W  }( J_1, J_2 ) | > \frac{ n_1 }{8} \phi_0 { pn } 
		\biggm\vert
			\Omega_{b',  I , b''} 
		\right\}  \\
	\le &
		C_{\rm{hg}}^{ n_1 } 	
		\exp \left( -
		 	\log\left(  
				\frac{\frac{1}{2}\phi_0 pn \cdot n/2}{ 24 n_2 (\phi pn)^2 }	
			\right) \frac{  n_1 }{ 16 } \phi_0{ pn } 
		\right) \\
	= &
		C_{\rm{hg}}^{ n_1 } 	
		\exp \left( -
		 	\log\left(  
				\frac{2en}{n_2} \cdot 
				\frac{\phi_0}{192e \phi^2 pn}
		\right) \frac{  n_1 }{ 16 } \phi_0{ pn } 
		\right) \,.
	\end{align*}

	Next, we impose the {\bf second assumption on $\gamma$}, requiring it to be sufficiently large, which may depend on $\beta$, $C_{ \rm hg }, \phi_0, $ and $\phi$, so that 
	\begin{align*}
		\frac{2en}{n_2} \ge 2e \gamma pn   		
	\ge 
		2 \cdot \frac{192 e\phi^2 pn}{\phi_0},
	\end{align*}
	which in turn implies that 
	$$
		\log\left(  
				\frac{2en}{n_2} \cdot 
				\frac{\phi_0}{192e \phi^2 pn}
		\right)
	\le 
		\frac{1}{2} \log\left( \frac{2en}{n_2} \right).
	$$
	Then, 
	\begin{align*}
		 \mathbb{P} \left\{  | I_{  W  }( J_1, J_2 ) | > \frac{ n_1 }{8} \phi_0 { pn } 
		\biggm\vert
			\Omega_{b',  I , b''} 
		\right\}  
	\le &			
		\exp \left( -
		 	\log\left(  
				\frac{2en}{n_2} 		\right) \frac{  n_1 }{ 32 } \phi_0{ pn } 
		\right),
	\end{align*}
	when $n$ is sufficiently large. 

	To apply the union bound over all possible $J_1$ and $J_2$ with the given size profile $(n_1,n_2)$, let us first recall the estimate: 	
	\begin{align*} 
		\left| \left\{  J_1 \subset J_2 \subset [n]
			\backslash  \mathcal{J} \, : \, 
		|J_1| = n_1 \mbox{ and } |J_2| = n_2  \right\} \right|
	\le 
		\exp \left( 
			\log( \frac{ 2e n }{ n_2 } )
			n_2
		\right) . 
	\end{align*}
	Given that the assumption $n_2 
	\le \frac{  \log n  }{ \log^3( \log n  )  }n_1$, when $ n $ is sufficiently large, we have 
	\begin{align*}
		 & \mathbb{P} \Big\{   
			\exists J_1 \subset J_2 \subset 
			[ n] \backslash \mathcal{J} 
		\,:\,  \\
		& \phantom{AAA}	|J_1| = n_1, \, |J_2| = n_2 
			\mbox{ and } |I(J_1, J_2) | > \frac{ n_1 }{ 8 } \phi_0 { pn } 
		\,|\,
			\Omega_{b',  I , b''}
		\Big\} \\
	\le &
		\exp \left( 
			\log( \frac{ 2e n }{ n_2 } )
			n_2
		 	-\log\left(  
				\frac{2en}{n_2} 		\right) \frac{  n_1 }{ 32 } \phi_0{ pn } 
		\right)\\
	\le 	&	
		\exp \left( 
		 	-\log\left(  
				\frac{2en}{n_2} 		\right) \frac{  n_1 }{ 64 } \phi_0{ pn } 
		\right)
	\le 
		\exp \left( 
		 	-\log\left(  2e\gamma pn
				 		\right) \frac{  n_1 }{ 64 } \phi_0{ pn } 
		\right).
	\end{align*}		
	Finally, we impose the {\bf third assumption on $\gamma$}, requiring it to be sufficiently large so that  
	\begin{align*}	
		\exp \left( 
		 	-\log\left(  2e\gamma pn
				 		\right) \frac{  n_1 }{ 64 } \phi_0{ pn } 
		\right)
	\le 
		\exp( - 100\beta pn n_1)
	\le& 
		n^{-2} ((1-p)^n)^{2\beta}	 \\
	=&
		o\big(n^{-2} \Prob\{ {\cal J}(0) \ge \beta \}\big).
	\end{align*}

	If we further take the union bound over all possible choices of $n_1, n_2$ described 
	in $ \Omega_{ W }$, we obtain 
	\begin{align*}
		\mathbb{P} \left\{ \Omega_{ W }^c \, |\, 		
		\Omega_{b',  I , b''} \right\}
	=
		o\big(\Prob\{ {\cal J}(0) \ge \beta \}\big).
	\end{align*}
		
	Since the estimate does not depends on our choice of 
	$\Omega_{b',  I , b''} 
		\subset \Omega_1 \cap \Omega_{  G  }$, 
	the proof is completed. 
\end{proof}

\begin{proof}[ Proof of Lemma \ref{lem: Omega_2 tauLarge}]
	This is a consequence of 
	Proposition \ref{prop: BadColumnC}, 
	Lemma  \ref{lem: Omega_G tauLarge}, and
	Lemma \ref{lem: Omega_W tauLarge}.
\end{proof}

\subsection{Proof of Theorem \ref{thm: SteepExpansion tau>10} and Theorem \ref{thm: T_1+}}

\begin{proof}[Proof of Theorem \ref{thm: SteepExpansion tau>10}]
	We will choose $\phi_0, \phi_1, c_\beta,$ and $\gamma$ so that 
	the probability estimate of Lemma \ref{lem: Omega_2 tauLarge} holds. Notice that 
	these constants depend on $\beta$, but not $ p $ nor $ n $. 

	Now fix a sample $A$ from $\Omega_{\ref{thm: SteepExpansion tau>10}} \cap \Omega_{\rm RC}(\beta) $ and let $x \in \cal T^*$ be a vector satisfying $ \opn{supp}(x) \cap  \mathcal{J}_A(0) = \emptyset $.  

		% \item $ |\mbox{supp}(x) \cap \mathcal{J} | \le 1$, and
        % \item $ x^*_{n_1} \ge  C  pn  x^*_{n_2}$ for 
		% $(n_1, n_2)$ satisfies one of the following two conditions:
		% \begin{itemize}
			% \item $ n_1 = 1$, $ n_2 = 2$,
			% \item $2   \le  n_1 \le \frac{1}{
					% \gamma
		% \mathfrak{p}}$, 
		% $ n_1 <  n_2 \le 
		% \max \{
			% \lfloor \frac{1}{ \gamma  \mathfrak{p}} \rfloor ,\,
			% \frac{  pn  }{ \log^3( pn )} n_1
		% \}$.
		% \end{itemize}
	% \end{enumerate}
	
	% We remark that every $x \in \mathcal{T}_1$ satisfies the third condtion descibed.

	\underline{Case 1: $n_1 >1$}
	In this case, we take  
	$y:=P_{ \cal J^c }x$. Then, we have 
	\begin{align*}
		y^*_{ \lfloor \frac{n_1}{2} \rfloor } & 
	\ge  C  pn y^*_{ n_2 },  \,\,\,\,\,\,\,\,
	\mbox{                   and }& 
		y^*_{ \lfloor \frac{n_1}{2} \rfloor } & 
	\ge x^*_{n_1}.
    \end{align*}
	Next, we set $J_1 = \sigma_y( \lfloor \frac{n_1}{2} \rfloor)$
	and $J_2 = \sigma_y( n_2)$. Notice that $ |J_1| \le 
	\frac{  pn }{ \log^2(   pn  ) }	
	|J_2|$ from the definition of $n_1$ and $n_2$.
	Thus, from the above choice of $J_1$ and $J_2$ we have 
	$|I_{  W  }(J_1,J_2) | > \beta$ due to $A \in \Omega_{ W }$, 

	Now, for $i \in I_{  W }(J_1,J_2)$, we have 
	$ i \in 
	[ n]  \backslash \mathcal{J}$ and thus  
	$(Ax)_i  = (Ay)_i$.  Let $j_0$ be the unique index in $J_1$
	which $a_{ij_0} = 1$. We have 
	\begin{align*}
		|  (Ax)_i | 
	\ge 
		|(Ay)_i| \ge |y_{j_0}| - 
		\sum_{ j \in [n] \backslash \mathcal{J} }
		a_{ij}|y_j|
	\ge 
		y^*_{ \lfloor \frac{n_1}{2} \rfloor }	
		- \phi  pn  
		\frac{1}{ C  pn } 
		y^*_{ \lfloor \frac{n_1}{2} \rfloor }.
	\end{align*}
	By setting $  C \ge 2\phi$, together with 
	$y^*_{ \lfloor \frac{n_1}{2} \rfloor } \ge x^*_{n_1}$,
	we may conclude that 
	\begin{align*}
		| ( Ax )_i | \ge 	
		\frac{1}{2}x^*_{n_1}.
	\end{align*}
	Again, the choice of $\phi$ depends only on $\beta$, and thus 
	$ C$ is a $\beta$-dependent value. 

	\underline{Case 2: $n_1=1$}
	In this case, we have $n_2=2$. 
	If $\sigma_x(1) \notin \mathcal{J}$, the simple modification of the above argument shows
	\begin{align*}
	\left| \left\{ 
		i \in [ n ] \,:\, 
		|(Ax)_i | > \frac{1}{2} x^*_1 
	 \right\}  \right|  > \beta.
	\end{align*}

	Now suppose $\sigma_x(1) \in \mathcal{J}$.
	Notice that $ 
	\mbox{supp}( {\bf C}_{\sigma_x(1)}A )	
	\neq \emptyset$ due to $\mbox{supp}(x) \cap
	 \mathcal{J}_A(0) = \emptyset $.

	For each $ i \in \mbox{supp}( {\bf C}_{\sigma_x(1)}A)$, 
	\begin{align*}
		|\left( A x \right)_i| & = 
		| \sum_{j \in [ n ]} a_{ij} x_j | 
	 \ge 
		x^*_{1}
		- | \sum_{ j \neq \sigma_x(1)  } a_{ij}x_j |
	 \ge 
		x^*_{1} - \frac{ \phi pn }{
			 C  pn 
		} x^*_1
     \ge \frac{1}{2} x^*_1.
	\end{align*}

	% In the end, suppose $ x \in \mathcal{T}_{1j}$ for $ 1 \le j \le   {t_2}  $.
	% We may take $n_1 = n_{j-1}$ and $n_2 = n_{j}$. 
	% And we finish the proof due to $x^*_{n_1} \ge 
	% n^{-1 + o_{n}(1) } \left\| x \right\| $ 
	% by Proposition \ref{prop: xNorm tauLarge}. 

\end{proof}

\begin{proof}[Proof of Theorem \ref{thm: T_1+}]
	Let $\Omega( A )$ and 
	$\Omega(  A ^\top)$ be the events described in 
	Theorem \ref{thm: SteepExpansion tau>10} for $ A $ and 
	$  A ^\top$ respectively.

	\underline{Case 1:}
	Suppose there exists 
	$$j_0 \in  \mathcal{J}_A(\phi_0  pn ) \backslash 
	 \mathcal{J}_A(0)\quad \mbox{and} \quad i_0 \in  \mathcal{J}_{A^\top}(\phi_0  pn ) \backslash 
	 \mathcal{J}_{A^\top}(0)$$ such that $$a_{i_0j_0}=1\,.$$
	Now, we pick $I$ and $J$ with $|I|= |J| = n - \beta + 1$ in the following way:
	\begin{align*}
		j_0 \in &J \subset [ n ] \backslash 
		 \mathcal{J}_A(0),  \mbox{ and }  J \cap 
		 \mathcal{J}_A(\phi_0  pn ) = j_0\\
		i_0 \in &I \subset [ n ] \backslash 
		 \mathcal{J}_{A^\top}(0), 
		\mbox{ and }  I \cap  \mathcal{J}_{A^\top}(\phi_0  pn ) = i_0.
	\end{align*}
	The existence of $J$ is due to $  \mathcal{J}_A(0) 
	\subset  \mathcal{J}_A(\phi_0  pn ) $ and 
	$  \mathcal{J}_A( \phi_0  pn ) \le \beta$. 
	For the same reason, such index set $I$ also exists. 

	Let $x \in \cal T^*$ with  $\opn{supp}(x) \subseteq J$. Here we will apply Theorem \ref{thm: SteepExpansion tau>10} to $A$.
	In the case where $n_1(x)=1$ and $\sigma_x(1) \in {\cal J}(\phi_0 pn)$, then necessarily $\sigma_x(1)=j_0$ due to 
	$$
		\opn{supp}(x) \cap {\cal J}_A(\phi_0 pn) \subseteq J \cap {\cal J}_A(\phi_0 pn) = \{j_0\}. 
	$$
	Thus, by Theorem \ref{thm: SteepExpansion tau>10} and $a_{i_0j_0}=1$, 
	$$
		\|A_{I,J}x\| \ge 
		|(Ax)_{i_0}|  \ge c x_{n_1}^*. 
	$$
	In the case where $x$ does not satisfies both $n_1(x)=1$ and $\sigma_x(1) \in {\cal J}(\phi_0 pn)$, the theorem shows the existence of     
	$$i \in I \cap \left\{ 
				i \in [ n] \,: \, |(Ax)_i| \ge
				c x_{n_1}^*
		\right\}, $$ 
	due to the set on the R.H.S. has size strictly greater than $\beta$. Thus, in both cases, there exists $i \in I$ 
	\begin{align} \label{eq: AIJx}
		\left\| A_{IJ}x \right\|_2  \ge |(Ax)_i| 
	\ge 
		c_1 x_{n_1}^*.
	\end{align}
	If we take $x \in \cal T^*$ with $\mbox{supp}(x) \subseteq I$, then following the same argument we have 
	\begin{align*} 
		\left\| A_{IJ}^\top x \right\|_2 
	\ge 
		c_1 x_{n_1}^*.
	\end{align*}
	\medskip 
	\underline{Case 2:}
	Next, we assume such pair $(i_0,j_0)$ described above does not exists. Let us consider the following:  
	\begin{enumerate}
		\item If $ | \mathcal{J}_A(\phi_0  pn )|= \beta$,  fix
			any $j_0 \in  \mathcal{J}_A(\phi_0  pn ) \backslash 
				 \mathcal{J}_A(0)$ and \\$i_{j_0} \in 
				\left\{ i \in [ n] \,:\,
					a_{ij_0}=1	
				\right\} $. 
		\item If $ | \mathcal{J}_{A^\top}(\phi_0  pn )|= \beta$,  fix
			any $i_0 \in  \mathcal{J}_{A^\top }(\phi_0  pn ) 
				\backslash 
				 \mathcal{J}_{ A^\top }(0)$ and $j_{i_0} \in 
				\left\{ j \in [ n] \,:\,
					a_{i_0j}=1	
				\right\} $. 
	\end{enumerate}	
	We remark that 	$i_{j_0} \notin  \mathcal{J}_{A^\top}(\phi_0  pn )$
	if $i_{j_0}$ exists. 
	Similarly, $j_{i_0} \notin  \mathcal{J}_{A}(\phi_0  pn )$ if 
	$j_{i_0}$ exists. 

	Now, it is not hard to verify that 
	we can pick $I$ and $J$ satisfying the following conditions:
	\begin{align*}
		J \cap  \mathcal{J}_A( \phi_0  pn ) &=
		\begin{cases}
			\left\{ j_0 \right\}  & \mbox{if }
			| \mathcal{J}_A(\phi_0  pn )|= \beta ,\\
			 \emptyset & \mbox{otherwise.}
		\end{cases}	 &
		\mbox{ and } j_{i_0} &\in J \mbox{ if }
		| \mathcal{J}_{A^\top} (\phi_0  pn )|= \beta; \\
		I \cap  \mathcal{J}_{A^\top}( \phi_0  pn ) &=
		\begin{cases}
			\left\{ i_0 \right\}  & \mbox{if }
			| \mathcal{J}_{A^\top} (\phi_0  pn )|= \beta ,\\
			 \emptyset & \mbox{otherwise.}
		\end{cases}	 &
		\mbox{ and } i_{j_0} &\in I \mbox{ if }
		| \mathcal{J}_{A} (\phi_0  pn )|= \beta.
	\end{align*}

	Then, $A_{I,J}$ will be the matrix which satisfies the condition of the event $\Omega_{\cal T^*}$. 
	The verification is the same as the that in the previous case. 
	We will skip it here. 
	Hence, we conclude that  
	$ \Omega_{\mathcal{T}^*}^*$ is a subevent of $ \Omega_{\rm RC}(\beta) \cup 
		\Omega(  A ) \cup \Omega(  A ^\top)$.
	And the probability estimate provided in Theorem 
	\ref{thm: SteepExpansion tau>10} implies we have the desired 
	bound for $\mathbb{P}(\Omega^c_{\mathcal{T}^*})$. 

\end{proof}

\bigskip

\part{Vector Decomposition and Proof of main theorem}
\section{Decomposition of $ \mathbb{R}^{n}$: Overview}

As stated in the introduction, here we will decompose $\R^n$ into three types of vectors: $\cal V$-vector, $\cal T$-vector, and $\cal R$-vector. For each type of vectors, we will also state the corresponding results about bounding $\|Ax\|$ or $\|A_{I,J}x\|$ quantitatively away from $0$ for $x$ in each type of vectors. 

Due to technical reason, we construct the vectors in the following orders: $\cal T$, $\cal R$, and $\cal V$. Also due to the technical complexity of their definition, it will be non-trivial to show that the union of $\cal T$, $\cal R$, and $\cal V$ type of vectors covers all non-null vectors in $\mathbb{R}^n$. 

Before that, to give a better picture of the decomposition, let us give a precise definition of $\cal V$ vector without fixing the corresponding parameters.  

Let us begin with a minimal definition of a non-sparse vector. 
\begin{definition}
Consider a parameter	
\begin{align}
\label{eq: tau}	
 \tau \in (0,1).
\end{align}
We define 
\begin{align}
	\label{eq: defY}
	\mathcal{Y}( \tau ) 
&:= 
	\big\{ 
		x \in \mathbb{R}^{n}\,:\,
		x^*_{ \lfloor \tau n \rfloor } = 1
	\big\}.
\end{align}

Clearly, for each vector $x \in \R^n$ with support at least $\lfloor \tau n \rfloor$, we can "normalize" it in the way that $ \frac{x}{x^*_{\lfloor \tau n \rfloor}} \in \mathcal{Y}(\tau)$. 
Further, let  
\begin{align}
	\label{eq: rho}
	\rho \in (0,1),
\end{align}
be another parameter, and consider the subcollection of $\cal Y(\tau)$, 
\begin{align}
	\label{eq: defAC} 
    \mathcal{AC}( \tau, \rho) &:=
    	\Big\{ x \in \mathcal{Y}(\tau) 	\,:\, 
       \big|  \big\{ i\in [ n ] \,:\, \big|x_i - x^*_{\lfloor \tau n \rfloor} \big| < 
            \rho \big\}  \big|  >  n- 
            \lfloor \tau n \rfloor    
        \Big\}. 
\end{align}
\end{definition}
In other words, this is the subcollection of those vectors in $\cal Y(\tau)$ with the majority of its components are close to $1$. These are vectors that looks almost like a constant if we ignore a small portion of its components. \\

Roughly speaking, a vector $x$ belongs to the complement of "sparse vectors" from the framework of Litvak-Tikhomirov is a vector with support at least $\lfloor \tau n\rfloor$ ($x\in \cal Y(\tau)$) and the growth of $x^*_i$ is stable as $i \searrow 1$. Here is the technical definition:    

\begin{definition} \label{def: g}
A {\bf growth function} $  g  $ is a non-decreasing function from $ [1,\infty) $ into $ [1,\infty) $.
For $\tau >0, \delta>0, \rho>0$, and a growth function ${ g }$, 
we define the gradual non-constant vectors with 
these parameters to be the set
\begin{align}
	\label{eq: defV}
	\cal V =& \mathcal{V}(\tau ,{ g }, \delta, \rho) \\ 
	\nonumber
	:= &
	\bigg\{  x \in \mathcal{Y}( \tau )\,:\,
		\forall i \in [ n ],\,x^*_i \le 
			{ g }( n /i) 	
		\mbox{ and }
		\exists Q_1,Q_2 \subseteq [ n ]
		\mbox{ such that }  \\
		& \phantom{AAA AAA AAA AAA AA}|Q_1|,|Q_2|\ge \delta  n 
		\mbox{ and }
		\max_{i\in Q_2}x_i \le \min_{i\in Q_1}x_i-\rho
	\bigg\}. \nonumber
\end{align}

\end{definition}

% \begin{align}
% 	\label{eq: defAC} 
%     \mathcal{AC}( \tau, \rho) &:=
%     	\bigg\{ x \in \mathcal{} 	
% 		\in \mathbb{R}^n\,:\,
%         \exists \lambda\in \mathbb{R} \mbox{ s.t } 
% 	|\lambda| = x^*_{\lfloor \tau n \rfloor}  \\
% 	& \hspace{2cm} \mbox{ and }
%         \left|  \left\{ i\in [ n ] \,:\, |x_i - \lambda |< 
%             \rho \lambda \right\}  \right|  >  n- 
%             \lfloor r n \rfloor    
%         \bigg\}. \nonumber
% \end{align}

The following is the concrete statement of \eqref{eq: introLTpartial} introduced in the Introduction, which can be obtained from Litvak-Tikhomirov \cite{LT20} under suitable modification: 
\begin{theorem}  \label{thm: LT}
 	For given $\tau,\delta, \rho \in (0,1)$, $s>0$
 	, and $R, K\ge 1$. There are $n_0 \in \mathbb{N}$ and $ C\ge 1$ 
 	depending on $\tau ,\delta,\rho,R,K$ such that 
 	the following holds. Let $n\ge n_0$ and  
	$\frac{\log n }{n} \le p\le C^{-1}$.
 	Let $ { g }:[1,\infty) \rightarrow 
 	[1, \infty)$ be an increasing growing 
 	function satisfying 
 	\begin{align} \label{eq: growthCondition}
 		\forall a \ge 2, \, \forall t\ge 1\,\,\,
 		{ g } (at) \ge { g }(t)+a 
 		\mbox{ and } 
 		\prod_{j=1}^\infty { g }(2^j)^{j2^{-j}}\le K.
 	\end{align}
 
 	Let $A$ be an $n\times n$ Bernoulli($p$) random 
 	matrix. Let $\Omega( \mathcal{V}^c )$ be the event that 
	$ \forall x \notin \cup_{\lambda \ge 0} \lambda \mathcal{V}_{n}( \tau, 
	{ g }, \delta, \rho)$, 
	$ \left\| A^\top x \right\| > 0 $. 
	Then,  
 	\begin{align*}
		& \mathbb{P} \left\{  
			\exists x \in \mathcal{V}( \tau, 
	{ g }, \delta, \rho) \mbox{ s.t. }
			\left\| Ax \right\| \le t b_n
		\biggm \vert
			\forall y \notin \cup_{ \lambda \ge 0} \lambda
			\mathcal{V}( \tau, 
	{ g }, \delta, \rho),\, 
			\left\| A^\top y \right\| > 0
 		\right\} \\
	\le &
	\frac{ CR }{ r^R ( \mathbb{P} \left\{  \Omega( \mathcal{V}^c) \right\}  )^2  }
	\left( 
	\exp( - R pn )  + \frac{ b_n}{ \sqrt{pn} }t \right),
 	\end{align*}
	where $b_n = \sqrt{\sum_{i=1}^n g(\frac{n}{i})^2}$. 

 \end{theorem}

\textbf{The choice of parameters $\tau,\rho,g,\delta$}, or the collection of $\cal V$ will be determined after other types of vectors: the steep vectors $\cal T$ (a better partition of $\cal T^*$ type vectors), the collection of vectors $\cal R$ where we can apply Rogozin's theorem (Theorem \ref{thm:Rogozin}). 
We will first construct $\cal T$, following by  $\cal R$, and then $\cal V$. The parameters will be chosen so that $\R^n$ can be decomposed into the above 3 types of vectors, at the same time satisfying the conditions stated in the theorem above.  

\section{ Definition of $\mathcal{T}$-vectors and $\cal R$-vectors}
The $\mathcal{T}$-vectors will be breaked into 3 subcollection: $\mathcal{T}_1$, $\mathcal{T}_2$, and $\mathcal{T}_3$. 

We will start with $ \mathcal{T}_1$ vectors. Basically, we are partitioning vectors of $\cal T^*$ type described in the previous part, according to $(n_1(x),n_2(x))$. 

For a fixed positive integer $\beta$, the construction of 
$ \mathcal{T}_1$-vectors are different when 
$ \frac{ \log( n ) }{ n }
\le p  \le 
( 1 + \frac{ 1 }{ 2\beta } )\frac{ \log( n ) }{ n }$
and when 
$ p \ge  
( 1 + \frac{ 1 }{ 2\beta } )\frac{ \log( n ) }{ n }$.

\subsection{ $ \mathcal{T}_1$ for
	$\frac{ \log( n )  }{ n  }
\le 	p 
\le 
	( 1 + \frac{ 1 }{ 2 \beta } ) 
	\frac{ \log( n )  }{ n  } $ }

\begin{definition} \label{def: T_1smallp}
	
For a fixed positive integer $\beta$, 
let $ \gamma ,  C_{\mathcal{T}_1} \ge 1$ be two parameters. 
For $ 	\frac{ \log( n )  }{ n  }
\le 	p 
\le 
	( 1 + \frac{ 1 }{ 2 \beta } ) 
	\frac{ \log( n )  }{ n  } $,
let $t_1$ and $ {t_2} $ be 
the positive integers defined by the following formulas:
First, $t_1$ is the unqiue integer satisfying  
$$3^{t_1-1} 
		<
			\lceil 
				\exp(  \frac{  \log n   }{ \log^2(  \log n  ) } )
			\rceil
		\le 
			3^{t_1}\,.$$
Once $t_1$ is determined, we define $t_2$ to be the unique integer satisfying	 
\begin{align*}
	\left\lceil 
		\exp\left( \frac{ \log n }{ \log^2(\log n) } \right)
	\right\rceil&
	\left(  
		\frac{ \log n }{ \log^3(\log n) }
	\right)^{t_2 - t_1 - 1} \\
	& \phantom{AAA AA} < 
	\left\lceil \frac{1}{\gamma p} \right\rceil \\
	&\phantom{AAA AAA}\le 
	\left\lceil 
		\exp\left( \frac{ \log n }{ \log^2(\log n) } \right)
	\right\rceil
	\left(  
		\frac{ \log n }{ \log^3(\log n) }
	\right)^{t_2 - t_1}\,.
\end{align*}

%	 \begin{align*}
%			\lceil 
%				\exp(  \frac{  \log n   }{ \log^2(  \log n  ) } )
%			\rceil
%			\left(  
%				\frac{  \log n   }{ \log^3(  \log n   ) }& 
%			\right)^{ {t}_2 -t_1 - 1}
%			& < \lceil \frac{1}{ \gamma  p } \rceil \\
%		& \le 
%			\lceil 
%				\exp(  \frac{  \log n   }{ \log^2(  \log n  ) } )
%			\rceil
%			\left(  
%				\frac{  \log n   }{ \log^3(  \log n   ) } 
%			\right)^{ {t}_2 - t_1 }
%	\end{align*}
Next, let $n_0 = 1$. For $1\le j \le  {t_2} $, we define
\begin{align*}  
	j &< t_1  &  
	n_j & = 3^{j} 
\\
	j &= t_1  &  
	n_{ t_1 } & = 
		\Big\lceil 
			\exp \Big(  \frac{  \log n   }{ \log^2(  \log n  ) } \Big)
		\Big\rceil 
\\ 
	t_1+1 \le j & \le   {t_2}  -1 & 
	n_j &=
		\Big\lceil 
			\exp \Big(  \frac{  \log n   }{ \log^2(  \log n  ) } \Big)
		\Big\rceil 
		\left(  
				\frac{  \log n   }{ \log^3(  \log n   ) } 
		\right)^{ j - t_1 } 
\\ 
	j & =   {t_2}   &  
	n_{  {t_2}  } &=  
	\Big\lceil \frac{1}{ 
		\gamma  p  
	}  \Big\rceil
\end{align*}

We define inductively $\cal T_{1j}$ for $j$ from $1$ to $t_2$:
\begin{align}  \label{eq: T_1smallp}
	\mathcal{T}_{1j} 
&= 
	\mathcal{T}_{ 1j }( \gamma ,  C_{\mathcal{T}_1} ) 
: = 
	\left\{ x\in \mathbb{R}^n \,:\,
	x\notin \cup_{i=1}^{j-1} \mathcal{T}_{1i} 
	\mbox{ and } x_{ n_{j-1} }^* >  C_{\mathcal{T}_1}  pn  
	x_{ n_{j}  }^*\right\},
\\
	\mathcal{T}_1 &  = \mathcal{T}_1( \gamma ,  C_{\mathcal{T}_1} )
	 := 
	\cup_{j=1}^{ {t_2}}\mathcal{T}_{1j}.
\end{align}

\end{definition}

\noindent
\textbf{Remark: }Notice that for $ x \in \cal T_{1j}$, while the ratio 
$x^*_{n_{j-1}}/x^*_{n_j}$ is big, the gap $x^*_{n_{i-1}}/x^*_{n_i}$ is bounded by $C_{\cal T_1} pn$ for $i < j$. This allow us to show  
$$
	x_{n_{j-1}}^* \ge n^{-1-o_n(1)} \|x\|.
$$  
% \HH{Comparing to the theorem .. }
\textbf{Remark: } Using $  \frac{1}{ \gamma  p  } \le n$, the following 
bounds on $t_1, t_2 - t_1$, and $t_2$ hold:

\begin{align} \label{eq: t_0t_1sEstimate}
	t_1  
\le 
	\frac{  \log n   }{ \log^2(  \log n   ) }, \qquad
	t_2 - t_1 
\le &
	(1+o_{n}(1)) \frac{ \log( n )}{ \log(  \log n  )  }, 
	\mbox{ and }\\
	\nonumber
	 & \phantom{AAA AAA AAA} {t_2}  
\le 
	(1+o_{n}(1)) 
	\frac{ \log( n )}{ \log( \log n  )  }. 
\end{align}

\begin{proposition} \label{prop: xNorm}
	Suppose $ \frac{ \log( n ) }{ n   } \le p \le 
	(  1 + \frac{ 1 }{ 2\beta } ) \frac{ \log( n ) }{ n   }$.
	For $ j \in [  {t_2} ]$ and $x \in \mathcal{T}_{1j}$,
	\begin{align}
		\frac{  \left\| x \right\| 	}{
			x^*_{ n_{j-1}}	
			}
	\le 
		n^{ 1 + o_{n}(1) }.
	\end{align}
	And for $ x \notin \mathcal{T}_1$, 
	\begin{align}
		\frac{ \left\| x \right\| }{
			x^*_{ n_{  {t_2}}}	}
	& \le 
		n^{ 1 + o_{n}(1) } , &
		\frac{ x^*_{  1} }{
			x^*_{ n_{   {t_2}  }}
			}
	& \le 
	n^{ 1 + o_{n}(1) }.
	\end{align}

\end{proposition}

\begin{proof}
	We will prove the first statement. As we will see later, the second statement follows from the same argument.
	
	Now we fix $ j \in [   {t_2}   ]$ and $ x \in \mathcal{T}_{1j}$.
	Due to the fact that $ x \notin \mathcal{T}_{1i}$ for 
	$i \in [j-1] $, 
	\begin{align*}
		\frac{ x^*_{ n_{i-1} } }{
			x^*_{ n_{j-1}} 
		} 
	\le 
		(  C_{\mathcal{T}_1}  pn  )^{ j - i }.
	\end{align*}

	Using the estimate $ n_{i} \le 
	\left( \frac{  \log n   }{ \log^3(  \log n   ) } \right)^i$, we get 
	
	\begin{align*}
		\sum_{l = 1}^{ n_j } (x^*_l)^2
	& \le 
		\sum_{i = 1}^{j} n_i (x^*_{ n_{i-1} } )^2
	\\ 
	& \le 
	\sum_{i = 1}^j 
		\left( \frac{  \log n   }{ \log^2(  \log n   ) } \right)^i
		\left(   C_{\mathcal{T}_1}  pn  \right)^{2 ( j- i) } 
		(x^*_{ n_{j-1} })^2
	\\
	& = 
		\left(   C_{\mathcal{T}_1}  pn  \right)^{2j  } 	
		(x^*_{ n_{j-1} })^2
		\left(  \sum_{ i = 1 }^j 
		\left( \frac{  \log n   }{ \log^2(  \log n   ) } \right)^i
		\left(   C_{\mathcal{T}_1}  pn  \right)^{-2 i } 	
		\right) 
	\\
	& \le 
		\left(   C_{\mathcal{T}_1}  pn  \right)^{2j  } 	
		(x^*_{ n_{j-1} })^2
	\end{align*}
	where the second to last inequality relies on the geometric decay of the 
	summation
	$$ \left(  \sum_{ i = 1 }^j 
		\left( \frac{  \log n   }{ \log^3(  \log n   ) } \right)^i
		\left(   C_{\mathcal{T}_1}  pn  \right)^{-2 i } 	
		\right) 
		\overset{\log n  \le pn}{\le}
		\sum_{i=1}^j \left( 
			\frac{ 1}{\log^3( \log n  )  C_{\mathcal{T}_1}^2  pn  }
		 \right)^i = o(1).
	$$
	For the remaining terms,  using $x^*_l \le x^*_{ n_j} 
	\le \frac{1}{  C_{\mathcal{T}_1}  pn   } x^*_{ n_{j-1}}$ 
	for $l \ge n_{j-1}$, 
	\begin{align*}
		\sum_{ l = n_j + 1}^{ n } 
		(x^*_l)^2
	& \le 
		n \frac{ 1 }{ (  C_{\mathcal{T}_1}  pn  )^2  } 
		(x^*_{ n_{j-1} })^2.
	\end{align*}
	
	Combining these two inequalities we obtain 
	\begin{align} \label{eq: normxbound}
		\left\| x \right\| ^2 
	& \le	\left(  		 
			\left(  C_{\mathcal{T}_1}  pn  \right)^{2j } +
			\frac{ n }{ (  C_{\mathcal{T}_1}  pn  )^2 } 
		\right)  
		(x^*_{ n_{j-1} })^2.
	\end{align}
	Using $ j \le   {t_2}   \le 
		( 1 + o_{n} (1) )  
		\frac{ \log( n ) }{ 
		\log( \log n  ) }$, we can bound the term 
	\begin{align*}
		\left(  C_{\mathcal{T}_1}  pn  \right)^{2j }
	\le 
		\left(  2C_{\mathcal{T}_1}  \log n   \right)^{2j }
	\le &
		\exp\left( 
			(2+ o_{n}(1) ) \log( 2C_{\mathcal{T}_1}  \log n  )
			\frac{\log n }{\log( \log n  )}
		 \right) \\
	\le&
		\exp \left( 
			(2+ o_{n}(1) )
			\left( 1+ \frac{\log( 2C_{\mathcal{T}_1})}{\log( \log n  )} \right) 
			\log n 	
		 \right) \\
	\le& 
		n^{2 + o_{n}(1)}
	\end{align*}
	and therefore, \eqref{eq: normxbound} becomes
	\begin{align*}
		\left\| x \right\| ^2 
	& \le  
		n^{ 2+ o_{n}(1) }
		(x^*_{ n_{j-1} })^2.
	\end{align*}
\end{proof}

\subsection{$ \mathcal{T}_1$ for $ p  
\ge ( 1 + \frac{ 1 }{ 2\beta } )\frac{ \log( n ) }{ n }$}

\begin{definition} \label{def: T1+}

For a fixed positive integer $\beta$, 
let $ \gamma,  C_{\mathcal{T}_1} \ge 1$ be two parameters. 
For  $$ ( 1 + \frac{ 1 }{ 2\beta } )\frac{ \log( n ) }{ n } \le
p \le \frac{ 1 }{ 2\gamma },$$
let $  {t_2}   $ be the positive integer defined by the following inequalites: 
\begin{align*}
			2
			\left(  
				\frac{  pn  }{ \log^3(  pn  ) } 
			\right)^{   {t_2}   - 2}
		<
			\Big\lceil \frac{1}{ \gamma  p  } \Big\rceil
		\le 
			2
			\left(  
				\frac{  pn  }{ \log^3(  pn  ) } 
			\right)^{   {t_2}   - 1 }.
\end{align*}

Let $ n_0 = 1$, $ n_1 = 2$, and 
\begin{align*}
	2 \le j & <   {t_2}   &
	n_j & = 2 \left(  
				\frac{  pn  }{ \log^3(  pn  ) } 
			\right)^{ j - 1}
\\ 
	j &=   {t_2}   &  
	n_{  {t_2}  } &=  \Big\lceil \frac{1}{ \gamma  p  }  \Big\rceil
\end{align*}

For $1 \le j \le   {t_2}  $, let 

\begin{align*}
	\mathcal{T}_{1j} &= \mathcal{T}_{ 1j }( \gamma ,  C_{\mathcal{T}_1} ) 
	 : = \left\{ x\in \mathbb{R}^n \,:\,
	x\notin \cup_{i=1}^{j-1} \mathcal{T}_{1i} 
	\mbox{ and } x_{ n_{j-1} }^* >  C_{\mathcal{T}_1}  pn  
	x_{ n_{j}  }^*\right\},
\\
	\mathcal{T}_1 &= \mathcal{T}_1( \gamma ,  C_{\mathcal{T}_1} ) : = \cup_{j=1}^{  {t_2}  }\mathcal{T}_{1j}.
\end{align*}
\end{definition}
{ \bf Remark:} Notice that 
\begin{align} \label{eq: slargep}
	 {t_2}   
\le 
	( 1+ o_{ n ( 1 ) } )\frac{ \log( \frac{ 1 }{ \gamma p  } ) }{ { \log( pn) }   }	+ 2.
\end{align}
And for $ p >> n^{-1/2}$, then  $  {t_2}   = 2$ since ${t_2}$ is an integer. The following proposition is an analogue of Proposition \ref{prop: xNorm}.

\begin{proposition} \label{prop: xNorm tauLarge}
	For $ j \in [   {t_2}   ]$ and $x \in \mathcal{T}_{1j}$,
	\begin{align}
		\frac{  \left\| x \right\| 	}{
			x^*_{ n_{j-1}}	
			}
	\le 
		n^{ 1 + o_{n}(1) }.
	\end{align}
	And for $ x \notin \mathcal{T}_1$, 
	\begin{align}
		\frac{ \left\| x \right\| }{
			x^*_{ n_{   {t_2}  }}	}
	& \le 
		n^{ 1 + o_{n}(1) } , &
		\frac{ x^*_{ 1 } }{
			x^*_{ n_{   {t_2}  }}
			}
	& \le 
	n^{ 1 + o_{n}(1) }.
	\end{align}

\end{proposition}

\begin{proof}
	We skip the proof here, since it is essentially the same as that of 
	Proposition \ref{prop: xNorm}. 
\end{proof}

\subsection{ $ \mathcal{T}_2, \mathcal{T}_3, $ and $ \mathcal{T}$ 
(for $ p \ge \frac{ \log( n ) }{ n } $) }

Once $ \mathcal{T}_1$ is constructed, the way we define $ \mathcal{T}_2$,
$ \mathcal{T}_3$, and $\mathcal{T}$ is the same regardless of the value 
of $ p $. Let 
\begin{align*}
	n_{   {t_2}   + 1 }  =  
	\sqrt{ \frac{ n }{  p  } } 
\quad \mbox{and} \quad
	n_{   {t_2}   + 2 }  =
	\tau n 
\end{align*}
where $ \tau \in (0,1)$ is the constant introduced in the definition of $\cal Y (\tau)$ and $\mathcal{V}(\tau ,{ g }, \delta, \rho)$ (See Definition \ref{def: g}). We remark that 
$$
	n_{t_2} < n_{t_2+1} < n_{t_2+2}
$$  
due to $pn \ge \log n $. Now we define
\begin{align*}
	\mathcal{T}_{2} 
&= 
	\mathcal{T}_{2}( \gamma ,  C_{\mathcal{T}_1}, C_{  \mathcal{T}_2 }, \tau  )
:= 
	\left\{ x\in \mathbb{R}^n \,:\,
	x\notin  \mathcal{T}_1
	\mbox{ and } x_{n_{  {t_2}  }}^* > 
		C_{\mathcal{T}_2}\sqrt{ pn  }
	x_{n_{  {t_2}  +1}}^*\right\}, 
\\
	\mathcal{T}_{3} 
&=
	\mathcal{T}_3( \gamma ,  C_{\mathcal{T}_1}, C_{  \mathcal{T}_2 }, \tau  )
: = 
	\left\{ x\in \mathbb{R}^n \,:\,
	x\notin  \mathcal{T}_1\cup \mathcal{T}_2
	\mbox{ and } x_{n_{  {t_2}  +1}}^* > 
	C_{  \mathcal{T}_2}\sqrt{ pn  }
	x_{n_{  {t_2}  +2}}^*\right\}, 
\\
	\mathcal{T} 
&=
	\mathcal{T}( \gamma ,  C_{\mathcal{T}_1}, C_{  \mathcal{T}_2 }, \tau  )
:= 
	\cup_{i=1}^3 \mathcal{T}_i.
\end{align*}

\begin{proposition}  \label{prop: TnormBound}
	For $x \notin \mathcal{Y}( \tau ) \backslash \mathcal{T}$, 
	we have 
	\begin{align*} 		
		x^*_{ n_{  {t_2}   + 2} } & = 1 &
		x^*_{ n_{   {t_2}   + 1} } &
	\le 
		C_{\mathcal{T}_2} \sqrt{  pn  } &
		x^*_{ n_{   {t_2}   } } &
	\le 
		C_{\mathcal{T}_2}^2  pn  \\
		x^*_{ n_{   {t_2}   - j } }  &
	\le 
		C_{\mathcal{T}_2}^2  pn 
		(  C_{\mathcal{T}_1}  pn  )^j
		\mbox{   for   } j \in [   {t_2}   ] &
		\left\| x \right\|_\infty & 
	\le 
		n^{ 1 + o_{n}(1) } { pn } 
	&
		\left\| x \right\|& 
	\le 
		n^{ 1 + o_{n}(1) }  { pn } 
	\end{align*}
\end{proposition}
While $\cal T_2$, $\cal T_3$ are similar to vectors in the form of $\cal T^*$ described in the Section \ref{sec: T*} and Section \ref{sec: T*+}, the technical tools to show  $\|Ax\| / \|x\|>0$ for $x \in \cal T_2$ or $x\in \cal T_3$ are different. 

\begin{proof}
	This is an immediate consequence of Proposition \ref{prop: xNorm}, 
	\ref{prop: xNorm tauLarge}, and the definition of 
	$ \mathcal{T}_2, \mathcal{T}_3, $ and $ \mathcal{T}$.
\end{proof}

\subsection{ $ \mathcal{R} $-vectors}

For $n_{  {t_2}  } \le k \le \frac{n}{
\log^2( pn )}$, let 
\begin{align*}
	\mathcal{R}_k^1 &:= \bigg\{ 
		 x \notin \mathcal{T}\backslash AC( \tau,\rho)
		 \,:\, x^*_{  \tau n  }=1, \, 
		 \frac{\left\| x_{\sigma_x([k,n])} \right\| }{
			 \left\| x_{\sigma_x([k,n])} \right\|_\infty 
		 } \ge \frac{2C_{\rm {Rgz}}}{\sqrt{p}} \mbox{ and } \\
		 & \phantom{AAA AAA AAA AAA AAA AAA AAA} 
		 \sqrt{ \frac{ n }{2} } \le \left\| x_{\sigma_x([k,n])} \right\| 
		 \le C_\mathcal{T}\sqrt{{pn^2}}
	 \bigg\} \\
	\mathcal{R}_k^2 &:= \bigg\{ 
		 x \notin \mathcal{T}		 
		 \,:\, x^*_{  \tau n  }=1, \, 
		 \frac{\left\| x_{\sigma_x([k,n])} \right\| }{
			 \left\| x_{\sigma_x([k,n])} \right\|_\infty 
		 } \ge \frac{2C_{\rm {Rgz}}}{\sqrt{p}} \mbox{ and }\\
		 & \phantom{AAA AAA AAA AAA AAA AAA AAA A} 
		 \frac{ 2 \sqrt{n} }{ \tau } 
		 \le \left\| x_{\sigma_x([k,n])} \right\| 
		 \le C^2_{  \mathcal{T}_2}  p 
		 n^{\frac{3}{2}}
	 \bigg\}  \\
	 \mathcal{R}_k &:= \mathcal{R}_k^1 
	 	\cup \mathcal{R}_k^2,
\end{align*}
where $C_{\rm{rgz}}>0$ is the constant appeared in 
Theorem \ref{thm:Rogozin}. In the end, we set 
\begin{align*}
	\mathcal{R} 
: = 
	\cup_{ n_{ {t_2} }}
		\le k 
		\le 	\frac{n}{
			\log^2( pn )} 
	\mathcal{R}_k .
\end{align*}

\section{Statement about $\cal T$- and $\cal R$-vectors}

For $ I, J \subset [n]$ with $ |I| = |J| = n- \beta +1$,
	let $\Omega_{I,J}(\mathcal{T},\mathcal{R})$ be the event that 

	\begin{align*}
		&\Omega_{I,J}(\mathcal{T},\mathcal{R})\\
	:=&
		\Big\{ \,\forall x \in \cal T \cup \cal R 
		\mbox{ with } \opn{supp}(x) \subseteq J, \quad
		\left\|  A _{I,J}x \right\|  
		 \ge n^{-1+ o_{ n}(1)} ( pn )^{-1}\left\| x \right\| 
		 , \, \mbox{and} \\
		&\phantom{\Big\{ \,}
		\forall x \in \cal T \cup \cal R
		\mbox{ with } \opn{supp}(x) \subseteq I, \quad
		\left\|  A _{I,J}^\top x \right\|   
		 \ge n^{-1+ o_{ n}(1)} ( pn )^{-1}\left\| x \right\| 
		 \phantom{AAA} \Big\}.
			\end{align*}
	Further, let 
	$$\Omega(\mathcal{T},\mathcal{R}) = \cup_{I,J} 
	\Omega_{I,J}(\mathcal{T},\mathcal{R})$$ 
	where the union is taking over all possible $I,J$ of size $n- \beta +1$.
	The goal of this section is to show 
	\begin{theorem} \label{thm: mainMerge}
		For fixed $\beta$, there exists 
		$ C_{\mathcal{T}_1}, \gamma , C_{\mathcal{T}_2}> 1$ and 
		$\tau, \rho, \delta, c_{ \beta } \in (0,1)$ which
		depends on $\beta$ so that the following holds: 
		For $ p $ satisfying 
		\begin{align*}
			\frac{ \log( n ) }{ n  }		 
		\le
			p 		
		\le 
			c_\beta,
		\end{align*}
		let ${ g }$ be the growth function defined in 
		Definition \ref{def: g}. 
			If $ n$ is sufficiently large, 
		\begin{align*}
			\mathbb{P}\{ \Omega(\mathcal{T}, \mathcal{R})^c \} 
		= 
			(1 + o_{n}(1))  
			\mathbb{P}\{ \Omega^c_{\rm RC} \} . 
		\end{align*}

	\end{theorem}

\begin{remark} \label{rem: T1T*}
Notice that from the definition of $\cal T_1$ and $\cal T^*$, we have  
$$\cal T_{1}(C_{\cal T_1}) \subseteq \cal T^*(C_{\cal T_1}) .$$
\end{remark}

Let us state a corollary of Theorem \ref{thm: T_1s}, Theorem \ref{thm: T_1+}, and Proposition \ref{prop: xNorm tauLarge}, we have  
\begin{corollary} \label{cor: T1}
For a positive integer $\beta$, there exists $c_\beta$, $C_{\cal T_1} \ge 1$, and $\gamma >1$ so that the following holds: 
For $p$ satisfying 
$$
	\frac{\log n }{n} \le p \le c_{\beta}, 
$$	
we have 
\begin{align*}
		&\Prob\bigg\{ 
		\mbox{There is no } I,J \subseteq [n] \mbox{ with } |I| =|J|=n-\beta+1 \mbox{ such that }		\\
		&\phantom{AA AAA} \forall x \in \cal T_1 \mbox{ with } \opn{supp}(x) \subseteq [J], \quad
			 \|A_{I,J}x\| \ge n^{-1-o(1)}\|x\| \mbox{ and }
		 \\
		&\phantom{AA AAA} \forall x \in \cal T_1 \mbox{ with } \opn{supp}(x) \subseteq [I], \quad
			 \|A_{I,J}^\top x\| \ge n^{-1-o(1)}\|x\|	\phantom{AAA AA} \bigg\} \\
= &o(\Prob\{ \Omega_{\rm RC}^c(\beta) \}).
\end{align*}
\end{corollary}

Next, let us states two results about $\cal T_1$, $\cal T_2$, and $\cal R$ vectors. 

\begin{theorem} \label{thm:T_2T_3-vector}
	For any fixed $\beta
	 \ge 1$, $ C_{\mathcal{T}_1} >1 $ and $ \gamma > 1$, 
	suppose $C_{\mathcal{T}_2}>1$ are sufficiently large, 
	$ c'_{ \beta }, \tau >0$ are sufficiently small. 
	For 
	\begin{align*}
		\frac{ \log( n ) }{ n  } \le p 
		\le  c'_{ \beta },
	\end{align*}
	and sufficiently large $ n $,  we have 
	\begin{align*}
		\mathbb{P}\left\{ \exists x\in \mathcal{T}_2 
			\mbox{ s.t. } 
			\left\|  A x \right\| 
		\le  
			n^{ - \frac{1}{2} - o_{n} (1) }
			( { pn }  )^{-1}
			\left\| x \right\| 
		\right\}  
	& \le 
		( 1- p )^{ 2 \beta n }
	=o(\Prob\{ \Omega_{\rm RC}^c(\beta)\}), 
	\\
		\mathbb{P}\left\{ \exists x\in \mathcal{T}_3 
			\mbox{ s.t. } \left\|  A  x \right\| 
		\le  
			n^{ - \frac{1}{2} - o_{n} (1) }
			( pn  )^{-1}
			\left\| x \right\| 
		\right\}  
	& \le 
		( 1- p )^{ 2 \beta n }
	=o(\Prob\{ \Omega_{\rm RC}^c(\beta)\}). 
	\end{align*}
\end{theorem}

\begin{theorem} \label{thm:R-vector}
	For any fixed $\beta
	 \ge 1$, $ C_{\mathcal{T}_1} >1 $ and $ \gamma > 1$, 
	suppose $C_{\mathcal{T}_2}>1$ are sufficiently large, 
	$ c'_{ \beta }, \tau , \rho \in ( 0,1 )$ are sufficiently small. 
	Then, for
	\begin{align*}
		\frac{ \log( n ) }{ n  } \le p 
		\le  c'_{ \beta }. 
	\end{align*}
	and sufficiently large $ n $, we have 
	\begin{align*}
		\mathbb{P} \left\{  
			\exists x \in \mathcal{R} \mbox{ s.t. }
			\left\|  A  x \right\| \le  
			n^{ - \frac{1}{2} - o_{n} (1) }
			(  { pn }  )^{-1/2}
			\left\| x \right\| 
		\right\} \le 			
		  ( 1- p )^{ 2 \beta n }
		=
o(\Prob\{ \Omega_{\rm RC}^c(\beta)\}).
	\end{align*}
\end{theorem}

We remark that $  C_{\mathcal{T}_1}, \gamma , C_{\mathcal{T}_2}$, and 
$\tau>0$ are parameters appeared in the definition of 
$\mathcal{T}$ vectors and $\mathcal{R}$ vectors. 

The proof of the above two theorems follows the work of Litvak-Tikhomirov \cite{LT20}. Due to the difference of definition vectors, we are not able to directly cite their results. Instead, we include the proof of these two theorems in the appendix. 

\begin{proof}[Proof of Theorem \ref{thm: mainMerge}]
The theorem is a corollary of 
Corollary \ref{cor: T1}, 
Theorem \ref{thm:R-vector}, 
Theorem \ref{thm:T_2T_3-vector}, and 
Theorem \ref{thm: vecPartition}.
\end{proof}

\section{ $\mathcal{V}$-vectors and Partition of $\R^n$}
Given the definition of $\cal T$ and $\cal R$ vectors in the previous section. The goal of this section is the following: We will define correspondingly the $\cal V$ vectors with suitable parameter choices so that 
\begin{enumerate}
	\item $\cal V$ satisfies the conditions described in Theorem \ref{thm: LT} (See Proposition \ref{prop: specificGrowthFunction}), and 
	\item $\R^n$ can be decomposed into these three types of vectors (See Theorem \ref{thm: vecPartition}).
\end{enumerate}

Recall that for $r>0, \delta>0, \rho>0$ and a growth 
function ${ g }$, 
\begin{align*}
	& \mathcal{V}_n(r,{ g }, \delta, \rho)  \\
	= &\Biggm\{ x\in \mathcal{Y}(r)\,:\,
		\forall i \in [n],\,x^*_i \le { g }(n/i) 	
		\mbox{ and }\\
		& \phantom{AAA}\exists Q_1,Q_2 \subset [n]
		\mbox{ such that } 
		|Q_1|,|Q_2|\ge \delta n 
		\mbox{ and }
		\max_{i\in Q_2}x_i \le \min_{i\in Q_1}x_i-\rho
	\Biggm\}.
\end{align*}

We will define our function ${  g }$ piecewisely on 
the intervals
$$ 
	[ 1,\, \frac{ n }{ n_{   {t_2}   + 2} } ) ,\quad
 [ \frac{ n }{ n_{   {t_2}   + 2} }, \, 
	\frac{ n }{ n_{   {t_2}   + 1} } ),\quad
 [ \frac{ n }{ n_{   {t_2}   + 1} }, \, 
	\frac{ n }{ n_{   {t_2}   } } ),\quad
 [ \frac{ n }{ n_{   {t_2}   } }, \, 
	\frac{ n }{ n_{   {t_2}  -1 } } ),\quad
 \dots \,, 
 [ \frac{ n }{ n_{ 1 } }, \, 
\infty) \,.$$

\begin{definition} \label{def: g}
	Consider the partition of $[1,\infty)$ into the following $t_2+2$  
	\begin{align*}
		U_{-2} :=& \Big[ 1, \frac{n}{n_{t_2+1}}\Big), \\
		U_{j} :=& \Big[ \frac{n}{n_{t_2-j}}, \frac{n}{n_{t_2-j-1}}\Big)
		\mbox{ for index } j \in \{-1,0,1, \ldots, t_2-2\}\\
		U_{t_2-1} :=&  \Big[ \frac{n}{n_{1}}, \infty \Big)
	\end{align*}

	Let ${  g  }$ be the function defined piecewisely on these intervals 
	\begin{align*}
		g(t) := \begin{cases}
			2t^{3/2} & t \in U_{-2},\\
			2t^3	& t \in U_{-1}, \\
			\frac{t}{n/n_{t_2-j}} (C_{\cal T_1} pn)^j (pn)^4  &
			t \in U_j \mbox{ for } j \in \{0,1,\dots, t_2-1\}. 
		\end{cases}	
	\end{align*}

	In particular, $ {  g  }$ is a function depending on 
	$n,  p , \tau,  C_{\mathcal{T}_1}, \gamma $ 
	and $ C_{\mathcal{T}_2}$.
\end{definition}

\begin{proposition} \label{prop: specificGrowthFunction}
	There exists a universal constant $K_3>0$ such that when $n$ is sufficiently large,  
	the following holds:
	\begin{enumerate}
		\item The function ${  g  }$ defined in Definition \ref{def: g}
	is a {\bf growth function} satisfying \eqref{eq: growthCondition} :
	\begin{align*}
		\forall a \ge 2, \, \forall t\ge 1\,\,\,
		{ g } (at) \ge { g }(t)+a 
		\mbox{ and } 
		\prod_{j=1}^\infty { g }(2^j)^{j2^{-j}}\le K_3.
	\end{align*}

	\item For $ x \in \mathcal{Y}_n( \tau ) \backslash 
	\mathcal{T}$, 
	\begin{align*}
		x^*_i \le { g } \Big( \frac{n }{ i } \Big)
	\end{align*}
	for $ \frac{ n }{ n_{   {t_2}   + 1} }
	\le i \le n$.

	\item We have $ b_n := \sqrt{\sum_{ i \in [ n ] } 
	(  { g } ( \frac{ n  }{ i }  ))^2}  = 
	n^{ 1 + o_{n} (1) }(  pn  )^7$, 
	which is the upper bound of $ \left\| x \right\|  $  
	for all $x \in \mathcal{V}$. 
	\end{enumerate}
\end{proposition}

\begin{proof}
	First, it is clear that $g$ is a piecewisely increasing function on each interval $U_i$. To show that ${ g }$ is a non-decreasing funciton, it is sufficient to show that 
	\begin{align} \label{eq: g GapMatch}
		\lim_{ s \nearrow n 
		/ n_{   {t_2}   - j } }
		{ g }(s) 
	\le 
		{ g } (n 
		/ n_{   {t_2}   - j  } ) .
	\end{align}
	for $j \in \{ -1,\, 0,\, 1,\, \dots,   {t_2}   - 1 \}$.

	First of all, this is immediate when $j = - 1$. 
	When $j = 0$, due to $n_{   {t_2}  } = \big\lceil \frac{1}{ \gamma   p } \big\rceil$
	we have
	\begin{align*}
		\lim_{ s \nearrow n 
			/ n_{   {t_2}    } }
		{ g }(s) 
	& = 
		2 ( \gamma  pn  )^3 	
		\le (  pn  )^4 \le 
		{ g  } \Big( \frac{ n }{ n_{   {t_2}   } }  \Big).
	\end{align*}

	For $j \in [   {t_2}   - 1 ]$, since  
	$
		\frac{ n_{   {t_2}   - j +1 
		}}{ n_{   {t_2}   - j } } \le  pn  
		\le  C_{\mathcal{T}_1}  pn 
	$ (See Definition \ref{def: T_1smallp} and Definition \ref{def: T1+}) and therefore
	\begin{align*}
		\lim_{ i \nearrow n 
			/ n_{   {t_2}   - j } }
		{ g }(i) 
	& = 
		\frac{ n_{   {t_2}   - j + 1}}{ n_{   {t_2}   - j } } 
		(  C_{\mathcal{T}_1}  pn  )^{ j - 1}
		( pn )^{4}
		\le {  g } \Big(  \frac{ n }{ n_{   {t_2}   - j }} \Big).
	\end{align*}
	Hence, we conclude that ${ g }$ is a non-decreasing function. 

	\medskip 

	Next, observe that ${  g }$ satisfies ${  g  }(at) \ge a{  g } (t)$ 
	with $a \ge 1$ in each interval which ${ g }$ defined piecewisely.
	And this property automatically extended to the whole domain 
	$ [1 , \infty )$ due to \eqref{eq: g GapMatch}. 
	We conclude that for $a \ge 2$ and $ s \ge 1$,
	\begin{align*}
		{ g }(as) \ge a { g }(s) \ge 
		{ g }(s) + a,
	\end{align*}
	where the second inequality relies on $ { g }(s) \ge 2$ for $s \ge 1$.
	For the second condition of \eqref{eq: growthCondition}, 
	it is not hard to see that $ g $ is growing like a polynomial.
	Specifically, we {\bf claim} that  
	$$ {  g  }(s) \le 2s^{10}.$$
	The claim follows immediately for $t \in U_{-2} \sqcup U_{-1}$ from its definition. 

	Next, notice that  
	$$
		\Big(\frac{n_{t_2+1}}{n_{t_2}}\Big)^5 
	\ge
		(\gamma \sqrt{pn})^{2.5}
	\ge 
		(C_{\cal T_1} pn)^2,
	$$	
	and for $j \in \{2,3,\dots, t_2-t_1\}$, 
	we have the rough estimate 
	$$
		\Big(\frac{n_{t_2-j}}{n_{t_2-j+1}}\Big)^5
	\ge	
		(pn)^3
	\ge 
		(C_{\cal T_1} pn)^2,
	$$
	though is not necessarily true when $j = 1$. 
	Hence, $j \in \{0,1,\ldots, t_2-t_1\}$ and $s \in U_{j}$, 
	\begin{align*}
		s^6 
	\ge 
		\Big(\frac{n}{n_{t_2-j}}\Big)^6 
	= &
		\Big(\frac{n}{n_{t_2+1}}\Big)^6 
		\Big(\frac{n_{t_2+1}}{n_{t_2}}\Big)^6 
		\Big(\frac{n_{t_2-j}}{n_{t_2}}\Big)^6 \\
	\ge &
		1^6 \cdot (C_{\cal T_1}pn)^2 \cdot 	(C_{\cal T_1} pn)^{2(j-1)}\\
	\ge&
		(C_{\cal T_1}pn)^{2j}\,.
	\end{align*}
	Therefore, the claims hold for $s \in U_j$ with $j\in \{0,1,\ldots, t_2-1\}$, 
	$$
		g(s) \le \frac{s}{n / n_{t_2-j}} \cdot s^5 \cdot s^4 \le s^{10}.  
	$$
	Therefore, the claim holds. 
	
	\medskip 	

	Following from the claim, for $j \ge 1$, 
	\begin{align*}
		\log( {  g } ( 2^j )^{ j 2^{-j} } )
	\le 
		\log\Big( 2^{11j^22^{-j}} \Big)
	&\le
		11 j^2 2^{-j}
	\end{align*}
	and therefore,  $ \sum_{j=1}^\infty {  g } ( 2^j )^{ j 2^{-j} } < K_3$
	for some sufficiently large $K_3>1$.

	It remains to show the comparison $ x^*_i \le { g }( \frac{ n }{ i} )$
	for $ x \in \mathcal{Y}(\tau ) \backslash \mathcal{T}$. 
	First of all, by Propsition \ref{prop: TnormBound},
	\begin{align*}
		x^*_{ n_{ {t_2} }  } \le C_{ \mathcal{T}_2 }^2  pn .
	\end{align*}
	
	For $ n_  {t_2}   \le i \le n_{   {t_2}   + 1 }$, 
	\begin{align*}
		 {  g  }\Big( \frac{n }{i} \Big) 
	\ge 
		{  g } \Big( \frac{ n }{ n_{   {t_2}   +1 } } \Big)
	= 
		{  g } \Big( \sqrt{  pn  } \Big)
	\ge 	
		2 (  pn  )^{3/2} 
	\ge 
		C_{ \mathcal{T}_2}  pn 
	\ge 
		x^*_{i}.
	\end{align*}

	Next, consider the regime 
	$ n_{   {t_2}   - j } \le i < n_{   {t_2}   - j + 1 }$
	for $j \in [   {t_2}   ]$.
	By Proposition \ref{prop: TnormBound}, $x^*_{ n_{   {t_2}   - j } }
	\le C_{ \mathcal{T}_2}^2  pn  (  C_{\mathcal{T}_1}  pn  )^j$.
	Thus, 
	\begin{align*}
		 {  g  } \Big( \frac{n }{i} \Big) 
	\ge 
		{  g } \Big( \frac{ n }{ n_{   {t_2}   - j + 1 } } \Big)
	=
		(  C_{\mathcal{T}_1}  pn  )^{j-1} (  pn  )^4	
	\ge 	
		x^*_{ n_{  {t_2}   - j } }
	\ge 
		x^*_{i}.
	\end{align*}

	It remains for us to estimate $ \sum_{ i=1 }^{ n } 
	g( \frac{ n  }{ i })^2$. 
	First, 
	\begin{align*}
		\sum_{ i= n_{  {t_2}  +1 } }^{ n } 
		g \Big( \frac{ n  }{ i } \Big)^2
	\le &
	n \cdot 4 (  pn  )^{3/2}, \mbox{ and  }  & 
		\sum_{ i= n_{  {t_2}   } 
		}^{ n_{   {t_2}  +1 } } 
		g \Big( \frac{ n  }{ i } \Big)^2
	\le &
		\sqrt{ \frac{ n }{ p   } }
		\cdot	4( \gamma  pn  )^6. & 
	\end{align*}

	Next, 
	\begin{align*}
		\sum_{ i=1 }^{ n_{   {t_2}   } }		
		g \Big( \frac{ n  }{ i } \Big)^2 
	\le &
		\sum_{ j=0 }^{  {t_2}  -1 } 
		\sum_{ i = n_{   {t_2}  - j -1 }  }^{
		n_{  {t_2}  -j } }
			g \Big( \frac{ n  }{ i } \Big)^2\\
	\le &
		\sum_{ j=0 }^{  {t_2}  -1 } 
		n_{  {t_2}  -j }	
		g\Big( \frac{ n  
		}{ n_{  {t_2}  -j -1 } } \Big)^2\\
	\le&
		\sum_{ j=0 }^{  {t_2}  -1 } 
		(  pn  )^{  {t_2}  -j }
		(  C_{\mathcal{T}_1} pn  )^{ 2j+2} 
		(  pn )^{8} \\
	= &
		\sum_{ u=1 }^{  {t_2}   } 
		(  pn  )^{ u }
		(  C_{\mathcal{T}_1} pn  )^{ 2(  {t_2}  -u )+2} 
		(  pn )^{8}\\
	\le &
		(   C_{\mathcal{T}_1} pn  )^{2  {t_2}   + 10 }
		\sum_{ u=1 }^{  {t_2}  } 
		(  pn )^{ -u } \\
	\le &
		(   C_{\mathcal{T}_1} pn  )^{2  {t_2}   + 10 }.
	\end{align*}
	For $( 1 + \frac{ 1 }{ 2\beta } )\frac{ \log( n ) }{ n } \le
	p \le \frac{ 1 }{ 2\gamma }$, applying \eqref{eq: slargep} we have 
	\begin{align*}
		(   C_{\mathcal{T}_1} pn  )^{2  {t_2}   + 10 }
	\le  &
	(  C_{ \mathcal{T}_1 } pn)^{14} \cdot 
	\exp\bigg( 2( 1+o_n(1)) \log( C_{ \mathcal{T}_1 } pn )
	\frac{ \log( 1/\gamma p  )}{ \log( pn)  } \bigg)\\
	\le &
		(  C_{ \mathcal{T}_1 } pn)^{14}(  \frac{ 1 }{ \gamma p  } )^{2(1+ o_n(1)} \\
	\le &
		n^{2 + o(1) } (  pn  )^{ 14 }. 
	\end{align*}
	 For $\frac{ \log( n ) }{ n } 
	\le  p 
	\le ( 1 + \frac{ 1 }{ 2\beta } )\frac{ \log( n ) }{ n }$, we could apply \eqref{eq: t_0t_1sEstimate}
	to get the same result. 
	Together we conclude that 
	$$ \sum_{ i=1 }^{ n } 
		g\Big( \frac{ n  }{ i } \Big)^2 
	\le 
	n^{2 + o_{  n  }(1) } (  pn  )^{14}. $$
	\end{proof}

	In the end, we will wrap up this section with the last theorem:
\begin{theorem} \label{thm: vecPartition}
	For a sufficiently large $ n$, we have the following:

	Let $ \tau \in (0, \frac{1}{10})$, $\delta 
	\in (0, \frac{ \tau}{3})$, $\rho\in (0, \frac{1}{10})$. Then, 
	\begin{align*}
		\mathbb{R}^n \backslash 
		\big(\cup_{\lambda >0} \lambda \mathcal{V}
		( \tau, { g }, \delta, \rho)\big) 
	\subset 
		(\cup_{\lambda >0} \lambda \mathcal{R})
		\cup \mathcal{T} \cup \{ \vec{0} \} 
	\end{align*}
	where ${ g }$ is the growth function defined in 
	Definition \ref{def: g}.
\end{theorem}

\begin{proof}
	First of all, recall that
	$\mathcal{Y}( \tau ) := \left\{ 
		x \in \mathbb{R}^{n}\,:\,
		x^*_{ \lfloor  \tau n  \rfloor } = 1
	\right\}$.
	For
	$ x\notin \cup_{\lambda >0} \lambda \mathcal{Y}( \tau )$,
	$ x \in \mathcal{T}\cup \left\{ \vec{0} \right\} $.
	So we only need to consider 
	$ x\in \cup_{\lambda >0} \lambda \mathcal{Y}( \tau )$.
	To prove the theorem, it is sufficient to show the following statement:
	$$ 
	\mathcal{Y}_{ n }( \tau) 
	\backslash (\mathcal{V}_n
		(r, { g }, \delta, \rho)  \cup \mathcal{T})
		\subset \mathcal{R}. $$

	Now fix $x \in \mathcal{Y}_{ n }( \tau) 
	\backslash (\mathcal{V}_n
		(r, { g }, \delta, \rho)  \cup \mathcal{T})$.
	Since $ x \notin \mathcal{V}_n
		(r, { g }, \delta, \rho)  $, either 
	\begin{enumerate}
		\item 	there 
	exists no $Q_1,Q_2 \subset [n]$ with 
	$|Q_1|,|Q_2| \ge \delta n $ such that 
	$\max_{i\in Q_2} \le \min_{i\in Q_1} x_i-\rho$, or 
		\item 	$x^*_{i_0} > { g }(\frac{ n }{i_0})$
	for some $i_0 \in [ n ]$. 
	\end{enumerate}

	\medskip

	\underline{Case 1}: 
	In this case, as an intermediate step, we will first show $x \in \cal{AC}(\tau, \rho)$, and prove $x \in \cal R$ later.  
	
	First, there exists a subset $I$ of 
	size $ n - 2 \delta n$
	such that for $i,j\in I$, $ |x_i-x_j|\le \rho $.
	From the assumption $\delta < \frac{\tau}{3}$ stated in the theorem, there exists $j \in \sigma_x
	(\lfloor  \tau n  \rfloor)\cap I$ and $j' \in I \backslash \sigma_x
	(\lfloor  \tau n  \rfloor)$.\\

	Suppose $x_j>0$. Following from 
	$x^*_j \ge x^*_{\lfloor  \tau n  \rfloor}=1$, we have $x_j > 1$. 
	As a consequence, for $i \in I$, 
	$$x_i \ge x_j-\rho \ge 1-\rho.$$
	In particular, the above inequality applied to $i=j'$ as well: $ x_{j'} \ge 1-\rho$. Hence, for $i\in I$, $$x_i \le x_{j'}+\rho \le 1+\rho.$$ 
	Therefore, $x\in \mathcal{AC}( \tau,\rho)$. 
	If $x_j<0$, it can be treated similarly to show 
	$x\in \mathcal{AC}( \tau,\rho)$. We have 	
	$$\left\| x^*_{
        [ \lfloor  \tau n  \rfloor, n]
    } \right\| \ge 
    \sqrt{  ( n -\lfloor  \tau n  \rfloor - 2 \delta n)}(1-\rho) 
    \ge \frac{1}{2}\sqrt{ n },$$
	where we relied on $\tau,\rho < \frac{1}{10}$.

	Further, for  $ k>n_{  {t_2}  +1} $, relying on $x^*_{n_{   {t_2}  +1}} \le \gamma \sqrt{ pn }$ from $x \notin \mathcal{T}$, we have 
	$$\left\| x^*_{[k, n]} \right\|
    \le C_\mathcal{T} \sqrt{pn} \sqrt{n} .$$ 

	To summarize, we have shown that 
	\begin{align*}
		x \in \mathcal{AC}( \tau,\rho) \quad \mbox{ and } \quad
		\frac{1}{2}\sqrt{n} \le 
		\left\| x^*_{[k, n]} \right\|
		\le C_\mathcal{T} \sqrt{pn} \sqrt{n} 
		\mbox{ for }  
			n_{   {t_2}  +1} 
			\le k \le 
			\lfloor  \tau n  \rfloor.
	\end{align*}
	If we can show that there exists $k$
	with $n_{   {t_2}  +1} 
			\le k \le 
			\lfloor  \tau n  \rfloor$
	such that 
	$\frac{ \left\| x^*_{[k, n ]} \right\| 
		}{ \left\| x^*_{[k, n]} \right\|_\infty } \ge 
	\frac{2C_{\rm {Rgz}}}{\sqrt{  p  }}$, 
	then $x \in \mathcal{R}_{k}^1 \subset \mathcal{R}$. 
	Let $m_0:= \frac{n}{\log^2(pn)}$. Notice that 
	we have $m_0 \ge 2n_{  {t_2}  +1} $ when $n$ is sufficiently 
	large. Here we will argue according to the value of $x_{m_0}^*$: 	
	\begin{enumerate}
		\item $x^*_{m_0} \le \log^2(  pn  )$

			First, 
			$\left\| x^*_{[m_0, n ]} \right\| \ge 
			\left\| x^*_{[ \lfloor  \tau n  \rfloor, n]} \right\|  
			\ge \frac{1}{2}\sqrt{n}$. 	
			Second, $\left\| x^*_{[m_0, n]} \right\|_\infty
				\le  \log^2( pn ).
			$ 
			
			Then, 
			$\frac{ \left\| x^*_{[m_0, n ]} \right\| 
		}{ \left\| x^*_{[m_0, n]} \right\|_\infty }  \ge 
		\frac{\sqrt{ n }}{2\log^2(  pn  )} \ge
		\frac{2C_{\rm{Rgz}}}{\sqrt{  p  }}$. Therefore, 
		$x \in \mathcal{R}_{m_0}^1$.

		\item $x^*_{m_0} > \log^2(  pn )$
	
			First, 
			$\left\| x^*_{[ n_{  {t_2}  +1}, n ]} \right\| 
			\ge 
			 \left\| x^*_{[ n_{  {t_2}  +1}, m_0 ]} \right\| 
			\ge
			 \sqrt{\frac{1}{2}m_0 } x^*_{m_0}
			\ge 
				\sqrt{\frac{n}{2}} \log( pn )$.
				
			Second, $\left\| 
				x^*_{[ n_{  {t_2}   + 1}, n]} 
			\right\|_\infty
			\le  \gamma \sqrt{ pn }$.

			Then, 
			$\frac{ \left\| x^*_{[n_{  {t_2}   + 1}, n ]} \right\| 
				}{ \left\| 
				x^*_{[n_{  {t_2}   + 1}, n]} 
			\right\|_\infty }  \ge 
				\frac{ \log( pn )}{ \sqrt{2} \gamma \sqrt{p}}	
				\ge
			\frac{2C_{\rm {Rgz}}}{\sqrt{  p  }}$. Therefore, 
			$x \in \mathcal{R}_{ n_{  {t_2}  +1}}^1$.
	\end{enumerate}
		
	\medskip

	\underline{Case 2:}	
	Now, let us turn to the second case that 
	$x^*_{i_0} > { g }(\frac{ n }{i_0})$
	for some $i_0 \in [ n ]$. If there are multiple indices
	satisfying this inequality, let $i_0$ be the 
	smallest index among them. 
	First of all, since $x^*_{\lfloor  \tau n  \rfloor}=1$, 
	$i_0 \le \lfloor  \tau n  \rfloor$.
	On the other hand, we also have a lower bound of $i_0$ by Proposition \ref{prop: specificGrowthFunction}: 
	$i_0 > n_{   {t_2}   + 1}$.
	Our goal is to show $x \in \cal R_k^2 \subseteq \cal R$ for some suitable value $k \le i_0/2$.

	Now, for $ k \le \frac{i_0}{2}$,
	\begin{align*}
		\left\| x^*_{[k , n]} \right\|  \ge 
		\left\| x^*_{[\frac{i_0}{2},i_0]} \right\|  \ge 
		\sqrt{\frac{i_0}{2}}{ g }(n/ i_0) =  	
		\sqrt{\frac{i_0}{2}} \Big(2\frac{n}{i_0} \Big)^{3/2} 
		\ge 
			\frac{2}{\tau} \sqrt{n},
	\end{align*}	
	and for $k \ge n_{  {t_2}  }$, by Proposition \ref{prop: TnormBound}, 
	$$\left\| x^*_{[k , n]} \right\|  \le 
	\sqrt{n} C_{\mathcal{T}_2}^2  pn  =
		C_{\mathcal{T}_2}^2  p  n^{3/2}.$$ Hence, 
	for $ k \in [n_{t_2}, \min\{ i_0/2, n/\log^2(pn)\}]$, $x$ satisfies the second condition of the collection $\cal R_k^2$:
	$$
		\frac{2}{\tau} \sqrt{n} \le \left\| x^*_{[k , n]} \right\|  \le C_{\cal T_2}^2 pn^{3/2}. 
	$$ 
	It remains to show there exists $k \in [n_{t_2}, \min\{ i_0/2, n/\log^2(pn)\}]$ satisfying 
	$
		\frac{\left\| x^*_{[k,n]} \right\| 
		}{ \left\| x^*_{[k,n]} \right\|_\infty } 
	\ge 
		\frac{2C_{\rm{Rgz}}}{\sqrt{ p }}$
	for some $ n_{  {t_2}} 
	\le k \le \frac{n}{\log^2( pn )}$
	to conclude that $ x \in \mathcal{R}_k^2 \subset \mathcal{R}$.

	Suppose $n_{  {t_2}  _0+1} \le i_0 
	\le 2 \frac{n}{\log^2( pn )}$, 
	let $k = \frac{i_0}{2} \ge \frac{1}{2}n_{  {t_2}  _0+1} 
	\ge n_{  {t_2}  _0}$.
	We have 
	\begin{align*}
	    \frac{\left\| x^*_{[k,n]} \right\| }{
			\left\| x^*_{[k,n]} \right\|_\infty} 
	\ge 
		\sqrt{\frac{i_0}{2}}
		\frac{{ g }(\frac{n}{i_0}) }{ 
			{ g }(\frac{ 2 n}{i_0})}
	= 
		\sqrt{2i_0} 
	\ge 
		 \sqrt{2 \frac{ n }{  p  }} 
	\ge 
		\frac{2C_{\rm {Rgz}}}{\sqrt{ p }}.
	\end{align*}
	In the case $ 2 \frac{n}{\log^2( pn )} 
	\le  i_0 \le \tau n$, 
	let $k = \frac{n}{\log^2( pn )}$.
	We have 	
	\begin{align*}
	    \frac{\left\| x^*_{[k,n]} \right\| }{
			\left\| x^*_{[k,n]} \right\|_\infty} 
	\ge 
		\sqrt{\frac{i_0}{2}}
		\frac{{ g }(\frac{n}{i_0}) }{ 
			{ g }(\frac{n}{k})}
	\ge 
		\sqrt{\frac{i_0}{2}}(\frac{k}{i_0})^{3/2} 
	\ge 
		\frac{k^{3/2}}{\sqrt{2}  \tau n  }
	\ge 
		\frac{\sqrt{ n }
		}{ \sqrt{2} \tau \log^3( pn ) }
	\ge 
		\frac{2C_{\rm{Rgz}}}{\sqrt{ p }}.
	\end{align*}
	
\end{proof}

\section{Proof of Theorem \ref{thm: LT}}
In this section, we will fix constants $(r, \delta, \rho ) \in ( 0,1)$,
$R \ge 1$, and $K \ge 1$, and for each positive integer $n$, we fix an increasing 
function ${ g }: [1, \infty ) \mapsto [1, \infty)$ satisfying 
\begin{align*}
		\forall a \ge 2, \, \forall t \ge 1 : \, \, 
		{ g }(at) \ge { g }(t) + a \quad
		\mbox{ and } \quad  
		\prod_{j=1}^\infty { g } (2^j)^{j2^{-j}}  \le K.
\end{align*}

The following is a partial result of 
	Litvak-Tikhomirov \cite{LT20} which address gradual non-constant 
	vectors (Combination of Theorem 2.1 and Theorem 2.2):

\begin{theorem} \label{thm: LT Partial}
	There are $n_0 \in \mathbb{N}$, $c_\beta >0$  
	depending on $\tau , \delta, \rho, R,  K$ such that 
	the following holds. 
	Let $n \ge n_0$, $p \le c_\beta$, and $\log n  \le pn$. Let 
	$ { g } : [1 , \infty) \mapsto [1,\infty) $ be a growth function 
	(see Definition \ref{def: g}) satisfying 
	\begin{align*}
		\forall a \ge 2, \, \forall t \ge 1 : \, \, 
		{ g }(at) \ge { g }(t) + a 
		 \quad \mbox{ and } \quad  
		\prod_{j=1}^\infty { g } (2^j)^{j2^{-j}}  \le K. 
	\end{align*}	
	Then, with probability at least $1 - \exp(-Rpn)$, one has 
	\begin{align*}
		& \left\{ \mbox{Set of normal  vectors to } 
			{\bf C}_2(A_n), \dots , {\bf C}_n(A_n)
		\right\} 	 
		 \cap \mathcal{V} (\tau , { g }, \delta, \rho) \\
	& \subset 
		\left\{ 
			x\in \cal{Y}(\tau) \,:\, 
			\mathcal{L}\left( 
				\sum_{i=1}^n Y_i x_i, \sqrt{m} t 
			\right) \le C' \max\{t, \exp(-Rpn)\}
			\mbox{ for all } t\ge 0
		\right\} ,
	\end{align*}
	where $C'>0$ may depend on $K$, 
	and $ Y = (Y_1, \dots, Y_n)$ be a random $0/1$-vector 
	in $\mathbb{R}^n$ uniformly distributed on 
	the set of vectors with $m$ ones and $n-m$ zeroes. 
\end{theorem}

We also need a modified version of the "invertibility via distance" lemma in 
Litvak-Tikhomirov \cite{LT20} (Lemma 7.4). For a $n\times n$ random matrix $B$, $i \in [n]$ and $t>0$, consider the event 
\begin{align*}
	\Omega_B(i,t) := \left\{ 
	\mbox{dist}( H_{i}(A), {\bf C}_i(A) ) \le t b_n 
\right\} ,
\quad \text{where} \quad  H_i(A) := \opn{span}\big( \{{\bf C}_j(A)\}_{j \neq i}\big)
\end{align*}
and 
$$
	b_n 
:= 
	\sqrt{\sum_{i=1}^n { g }(\frac{n}{i})^2}
\ge 
	\max_{x \in \cal V} \|x\|.
$$
 
\begin{lemma} \label{lem: invertibilityDistance}
 	Fix a positive integer $k$. For a sufficiently large $n$, 
	let $B$ be an $n \times n$ random marix. 
	Then, for any $t > 0$ we have 
 	\begin{align*}
 		\mathbb{P} \big\{ 
 			\exists x \in \mathcal{V} \mbox{ s.t. }
 			\left\| Bx \right\| \le t \left\| x \right\| 
 		 \big\}  
 	\le 
		\frac{2}{ ( \tau n )^k } \sum_{ \bf i }
		\mathbb{P} \{ \Omega_{B}({\bf i}, t)\},
 	\end{align*}
 	where the sum is taken over all sequence 
	$ {\bf i} = (i_1, i_2, \dots, i_k)$ 
 	with distinct values $i_1,\dots, i_k \in [n]$, and
 	\begin{align*}
		\Omega_B({\bf i}, t) 
		= 	\Omega_{B}( (i_1,\dots,i_k), t ) 
	:=  \cap_{l=1}^k \Omega_B(i_l, t).
	\end{align*}
\end{lemma}

\begin{proof}
	Fix $t>0$. 
	Let ${\bf 1_i}$ be the indicator function of $\Omega_B({\bf i}, t)$. 
 	The condition $ \left\| Bx \right\| \le t \left\| x \right\| $ 
	for some $x \in \mathcal{V}_n$ implies that for every $i \in [n]$, 
 	\begin{align*}
 		|x_i|
 		 \mbox{dist}\left( H_i(B), {\bf C}_i(B) \right)
	=
		\|P_i Bx\|
 	\le
		\left\| Bx \right\|  
 	\le 
 		tb_n,
 	\end{align*}
	where $P_i$ is the orthogonal projection to $H_i(B)^\perp$, and 
 	the last inequality follows from the definition of 
 	$\mathcal{V}$ and $b_n$. Since $x^*_{\lfloor \tau n \rfloor } = 1$ from $x \in \cal Y(\tau)$, 
 	there are 
	$$\prod_{j=0}^{k-1} ( \lfloor \tau n \rfloor
 	- j ) \ge \frac{1}{2} ( \tau n )^k$$ 
	ordered sequence ${\bf i} =(i_1,\dots, i_k)$ with distinct values 
 	such that $ {\bf 1_i} = 1$. As a consequence, 
 	\begin{align*}
		\mathbb{P}\left\{ 
 			\exists x \in \mathcal{V}_n	\mbox{ s.t. }
 			\left\| Ax \right\| \le t \left\| x \right\|   
 		\right\} 
 	\le  
		\mathbb{P}\left\{ 
 			\sum_{{\bf i}} {\bf 1_i} \ge 
 			\frac{1}{2} ( \tau n )^k
 		\right\} .
 	\end{align*}
 
 	Then, apply the Markov's inequality in the right hand side 
 	we obtain the statement of the Lemma. 
\end{proof}

To prove the theorem, there are 3 steps. First, we will reduce 
the probability estimate from the tail of 
$ \min_{x \in \mathcal{V}} \left\|  Ax \right\| $
to the tail of distance to subspaces. In this step, we 
rely on Lemma \ref{lem: invertibilityDistance}. 
Second, we further simplify the estimate to the probability estimate of the tail 
of $| \langle X ,\, Y  \rangle |$ where $X$ 
is a normal vector to span of $ \left\{ {\bf C}_j(A) \right\}_{j\neq 1}  $ 
and $Y = {\bf C}_1(A)$.
And finally, we can apply Theorem \ref{thm: LT Partial} to 
finish the proof.

\begin{proof}
Let $\Omega( \mathcal{V}^c )$ be the event that 
$ \forall x \notin \cup_{\lambda \ge 0} \lambda \mathcal{V}$, 
$ \left\| A^\top x \right\| > 0 $. 
Applying Lemma \ref{lem: invertibilityDistance} with $B$ being 
$(A \big \vert \Omega( \mathcal{V}^c))$, we obtain  

\begin{align*}
	&\mathbb{P} \left\{  
		\exists x \in \mathcal{V} \mbox{ s.t. }
		\left\| Ax \right\| \le t b_n
	\biggm \vert
		\forall y \notin \cup_{ \lambda \ge 0} \lambda \mathcal{V},
		\left\| A^Tx \right\| \ge t \left\| x \right\| 
 	\right\}  \\
\le &
	\frac{2}{ ( rn )^k} \sum_{\bf i} \mathbb{P}\{ \Omega_B({\bf i}, t) \} \\
\le&
	\frac{2}{r^k} \mathbb{P} \{\Omega_B( (1,2,\dots, k) , t) \} \\
=&
	\frac{2}{r^k} 
	\mathbb{P} \left\{ \Omega_A( (1,2,\dots, k), t )
	\biggm \vert \Omega(\mathcal{V}^c) \right\},  \\
\end{align*}
where the second inequality follows from the oberservation that the 
distribution of $M$ is invariant under permutation of columns.

Next, we fix ${\bf i} = ( 1,2,\dots, k )$. For $i \in [n]$, let $U(i)$ be the 
event that $  \frac{1}{8} pn \le |\mbox{supp}({\bf C}_i(A)| \le 8 pn  $. 
Then, 
\begin{align} \label{eq: ProofOFLT OA}
	& \mathbb{P} \left\{ 
		\Omega_A( {\bf i}, t) \biggm \vert \Omega(\mathcal{V}^c) 
	\right\} \\
\nonumber
\le &
	\mathbb{P} \left\{ \cap_{l=1}^k U(l)^c 
		\biggm \vert \Omega(\mathcal{V}^c) 
	\right\} +
	\mathbb{P} \left\{ 
		\cup_{l=1}^k U(l) \cap 	
		O_A({\bf i}, t) \biggm \vert \Omega(\mathcal{V}^c) 
	\right\}.  
\end{align}

For the first summand of \eqref{eq: ProofOFLT OA}, 
\begin{align*}
	\mathbb{P} \left\{ \cap_{l=1}^k U(l)^c 
		\biggm \vert \Omega(\mathcal{V}^c) 
	\right\} 
\le &
	\frac{ 1 }{\mathbb{P} \left\{ \Omega(\mathcal{V}^c) \right\}   } 
	\mathbb{P} \left\{ \cap_{l=1}^k U(l)^c \right\} \\
= &
	\frac{ 1 }{\mathbb{P} \left\{ \Omega(\mathcal{V}^c) \right\}   } 
	\mathbb{P} \left\{ U(1)^c \right\}^k \\
\le &
	\frac{ 1 }{\mathbb{P} \left\{ \Omega(\mathcal{V}^c) \right\}   }
	\exp( -ck pn)
\end{align*}
where the last inequality follows by Proposition \ref{prop: Binomial} and 
$c>0$ is an universal constant. 

For the second summand of \eqref{eq: ProofOFLT OA}, 
\begin{align*}
	\mathbb{P} \left\{ 
		\cup_{l=1}^k U(l) \cap 	
		O_A({\bf i}, t) \biggm \vert \Omega(\mathcal{V}^c) 
	\right\}
\le & 
	\sum_{l=1}^k \mathbb{P} \left\{ 
		 U(l) \cap 	
		O_A({\bf i}, t) \biggm \vert \Omega(\mathcal{V}^c) 
	\right\} \\
\le &
	\sum_{l=1}^k \mathbb{P} \left\{ 
		U(l) \cap 	
		O_A(l, t) \biggm \vert \Omega(\mathcal{V}^c) 
	\right\} \\
= &
	k \mathbb{P} \left\{ 
		U(1) \cap 	
		O_A(1,t) \biggm \vert \Omega(\mathcal{V}^c) 
	\right\},
\end{align*}
where the last equality follows from the fact that the distribution of $A$ 
is invariant under permutation of of columns. 

Let $\Omega( \mathcal{V}^c, 1)$ be the event that 
$ \forall x \notin \cup_{\lambda \ge 0} \lambda \mathcal{V}$, 
$ \left\| (A_{[n],[2,n]})^\top x \right\| \ge t \left\| x \right\| $.
Notice that if the event $\Omega( \mathcal{V}^c, 1)$ holds, 
then any normal vector of $H_1(A)$ is contained in  
$ \cup_{ \lambda \ge 0 }\lambda \mathcal{V}$. 
Further, ${\bf C}_1(A)$ is independent from $\Omega( \mathcal{V}^c, 1)$ and $  \Omega( \mathcal{V}^c) \subset \Omega( \mathcal{V}^c, 1)$.
\begin{align*}
	\mathbb{P} \big\{ U(1) \cap \Omega_A(1,t)
		\big \vert 
		\Omega( \mathcal{V}^c) 	
	\big\} 	
\le &
	\frac{ \mathbb{P} \left\{ U(1) \cap \Omega_A(1,t) \cap 
		\Omega( \mathcal{V}^c,1 )\right\}  
	}{  
		\mathbb{P} \left\{ \Omega( \mathcal{V}^c)\right\}  
	} \\
= &
	\frac{ \mathbb{P} \left\{ \Omega( \mathcal{V}^c, 1 )\right\}  
 	}{ \mathbb{P} \left\{ \Omega( \mathcal{V}^c)\right\}  } 	
	\mathbb{P} \Big\{  U(1) \cap \Omega_A(1,t)
		\Big \vert 
		\Omega( \mathcal{V}^c, 1) 	
	\Big\} \\
\le & 
	\frac{ 1
 	}{ \mathbb{P} \left\{ \Omega( \mathcal{V}^c)\right\}  } 	
	\mathbb{P} \Big\{  U(1) \cap \Omega_A(1,t)
		\Big\vert 
		\Omega( \mathcal{V}^c, 1) 	
	\Big\}.
\end{align*}
Further,
\begin{align*}
	& \mathbb{P} \left\{  U(1) \cap \Omega_A(1,t)
		\biggm \vert 
		\Omega( \mathcal{V}^c, 1) 	
	\right\} \\
\le & 
	\mathbb{P} \left\{  \Omega_A(1,t)
		\biggm \vert 
		\Omega( \mathcal{V}^c, 1) 	
		\cap U(1)
	\right\} \\
\le &
	\max_{ \frac{ 1 }{ 8 } pn \le m \le 8 pn }
		\mathbb{P} \left\{  \Omega_A(1,t)
			\biggm \vert 
			\Omega( \mathcal{V}^c, 1) \cap 
			\left\{ | \mbox{ supp }({\bf C}_1(A)) | = m \right\}
		\right\}. 
\end{align*}

To summarize, at this point we reach 
\begin{align*}
& \mathbb{P} \left\{  
		\exists x \in \mathcal{V} \mbox{ s.t. }
		\left\| Ax \right\| \le t b_n
	\biggm \vert
		\forall y \notin \cup_{ \lambda \ge 0} \lambda \mathcal{V},
		\left\| A^\top x \right\| \ge t \left\| x \right\| 
 	\right\}  \\
\le &
	\frac{ 2 }{ \tau^k	\mathbb{P} \left\{  \Omega( \mathcal{V}^c) \right\} }
	\bigg(
 		\exp(-ck pn) +\\
& \phantom{AAA}
	+
		k \max_{ \frac{ 1 }{ 8 } pn \le m \le 8 pn }
		\mathbb{P} \left\{  \Omega_A(1,t)
			\biggm \vert 
			\Omega( \mathcal{V}^c, 1) \cap 
			\left\{ | \mbox{ supp }({\bf C}_1(A)) | = m \right\}
		\right\}
	\bigg).
\end{align*}

Let us make the following \textbf{claim:}
For any $R>1$, if $n$ is sufficiently large, 
\begin{align*}
& \max_{ \frac{ 1 }{ 8 } pn \le m \le 8 pn }
	\mathbb{P} \left\{  \Omega_A(1,t)
		\biggm \vert 
		\Omega( \mathcal{V}^c, 1) \cap 
			\left\{ | \mbox{ supp }({\bf C}_1(A)) | = m \right\}
		\right\} \\ 
\le & C\frac{tb_n}{\sqrt{pn}} +	
		\frac{ C }{ \mathbb{P} \left\{  \Omega( \mathcal{V}^c) \right\} } 
		 \exp(-Rpn) \,,
\end{align*}
where $C>1$ is a constant that could possibly depends on $K$ and $\tau$. 
Then,  we obtain 
\begin{align*}
& \mathbb{P} \left\{  
		\exists x \in \mathcal{V} \mbox{ s.t. }
		\left\| Ax \right\| \le t b_n
	\biggm \vert
		\forall y \notin \cup_{ \lambda \ge 0} \lambda \mathcal{V},
		\left\| A^Tx \right\| \ge t \left\| x \right\| 
 	\right\}  \\
\le &
	\frac{ 2 }{ \tau^k ( \mathbb{P} \left\{  \Omega( \mathcal{V}^c) \right\}  )^2  }
	\left( \exp(-ck pn) +  
	C\frac{ktb_n}{\sqrt{pn}} + k\exp(-Rpn) \right)
\end{align*}
In the end, by choosing $k=R$ and increase the value of $C$, we obtain 
the statement of the theorem. 

\medskip

Now, it remains to prove the {\bf claim}. 
Let us fix $m \in \Big[\frac{ 1 }{ 8 } pn , 8 pn\Big]$ and let 
$\Omega_{\rm LT}(m)$ be the event described in Theorem 
\ref{thm: LT Partial} that  
\begin{align*}
	& \left\{  
		\mbox{ Sets of normal vectors to }
			{\bf C}_j( A) 
		\mbox{ with $j \neq 1$ } 
	\right\} \cap \mathcal{V} \\
\subset &
	\left\{ 
		x\in \cal {Y}(\tau) \,:\, 
		\mathcal{L}\left( 
			\sum_{l=1}^n Y_l x_l, \sqrt{m} s 
		\right) \le C' \max\{s, \exp(- Rpn  )\}
		\mbox{ for all } s\ge 0
	\right\} 
\end{align*}
{
	wehre $ Y = (Y_1, \dots, Y_n)$ be a random $0/1$-vector 
	in $\mathbb{R}^n$ uniformly distributed on 
	the set of vectors with $m$ ones and $n-m$ zeroes, 
	and $C'$ is the constant described in Theorem 
	\ref{thm: LT Partial} which is compatible 
	with our chocie of growth function ${ g }$ and $K$. 
}

We remark that the event 
$ \Omega_{\rm LT}(m) \cap \Omega( \mathcal{V}^c, 1) $
is independent from ${\bf C}_1(A)$. 
Further, condition on $ \Omega(\mathcal{V}^c, 1)$, 
normal vectors of $H_1(A)$ 
are contained in $\cup_{\lambda \ge 0} \lambda \mathcal{V}$. 
From now on, we condition on 
	$ \Omega_{\rm LT}( m) \cap 
	\Omega( \mathcal{V}^c, 1)$,
	let $X \in \mathcal{V} $ be a measurable (random) vector which 
	is normal to $ \left\{ {\bf C}_j(A) \right\}_{j \neq 1}  $. 
	
	Then, $X$ and $\big( {\bf C}_1(A) \big\vert 
		 |\mbox{supp}({\bf C}_1(A))|=m  \big)$
	 are jointly independent.
	Further, $$\big( {\bf C}_1(A) \big\vert 
		 |\mbox{supp}({\bf C}_1(A))|=m  \big)$$
		fits the description of $Y$ in the event 
	$\Omega_{\rm LT}(m)$. 
	Together with 
		$ | \langle X ,\, {\bf C}_1(A) \rangle | 
		\le
		\mbox{ dist} ( H_1(A), {\bf C}_1(A) ) $, 
	\begin{align*}
		& \mathbb{P} \left\{  \Omega_A(1,t)
		\biggm \vert 
		\Omega_{\rm LT}( m) \cap 
		\Omega( \mathcal{V}^c, 1) \cap 
			\left\{ | \mbox{supp}({\bf C}_1(A)) | = m \right\}
	\right\} \\
	\le &
		\mathbb{P} \left\{ 
			| \langle X ,\, {\bf C}_1(A) \rangle | \le tb_n
	\biggm \vert 
	\Omega_{\rm LT}( m) \cap 
		\Omega( \mathcal{V}^c, 1) \cap 
			\left\{ | \mbox{supp}({\bf C}_1(A)) | = m \right\}
		\right\} \\
	\le & 
	C'\max\Big\{ \frac{ t b_n }{ \sqrt{pn}  } ,\, \exp(-Rpn) \Big\},
	\end{align*}
	where the last inequality follows form the definition of $\Omega_{\rm LT}(m)$. 
	Now, we apply Theorem \ref{thm: LT Partial} to get
	\begin{align*}
	\mathbb{P} \left\{ 
		\Omega_{\rm LT}( m)
	\biggm \vert 
		\Omega( \mathcal{V}^c, 1) \cap 
			\left\{ | \mbox{supp}({\bf C}_1(A)) | = m \right\}
		\right\} 
	= &
		\mathbb{P} \left\{ 
			\Omega_{\rm LT}( m)
	\biggm \vert 
		\Omega( \mathcal{V}^c, 1) 
		\right\} \\
	\le & 
	\frac{ 1 }{\mathbb{P} \left\{ \Omega( \mathcal{V}^c, 1) \right\} }  
 	\mathbb{P} \left\{ \Omega_{\rm LT}( m) \right\} \\
	\le &
		\frac{ 1 }{\mathbb{P} \left\{ \Omega( \mathcal{V}^c) \right\} }  
		\exp(-Rpn)
	\end{align*}
	where the first equality is due to the independence of 
	$ \Omega_{\rm LT}( m)$ and  \\
	$\left\{ | \mbox{supp}({\bf C}_1(A)) | = m \right\}$, 
	the last inequality is due to 
	$ \mathbb{P} \left\{ \Omega( \mathcal{V}^c, 1) \right\} \ge 
		\mathbb{P} \left\{ \Omega( \mathcal{V}^c) \right\}$.
	Finally, we conclude that 
	\begin{align*}
	& \mathbb{P} \left\{  \Omega_A(1,t)
		\biggm \vert 
		\Omega( \mathcal{V}^c, 1) \cap 
			\left\{ | \mbox{supp}({\bf C}_1(A)) | = m \right\}
	\right\} \\
	\le &
		\mathbb{P} \left\{ 
			\Omega_{\rm LT}( m)
	\biggm \vert 
		\Omega( \mathcal{V}^c, 1) \cap 
			\left\{ | \mbox{supp}({\bf C}_1(A)) | = m \right\}
		\right\} \\
	    & +
	\mathbb{P} \left\{  \Omega_A(1,t)
		\biggm \vert 
		\Omega_{\rm LT}( m) \cap 
		\Omega( \mathcal{V}^c, 1) \cap 
			\left\{ | \mbox{supp}({\bf C}_1(A)) | = m \right\}
		\right\} \\
	\le &
	C\frac{ tb_n }{ \sqrt{pn} } + 
	\frac{ C }{ \mathbb{P} \left\{ \Omega( \mathcal{V}^c) \right\}  } 
	\exp( -Rpn ),
	\end{align*}
	where  $C>1 $ is a fixed constant which depends on $C'$.

\end{proof}

\section{ Proof of Main Theorem }
Recall that, for a fixed positive integer $\beta$, recall that 
$\Omega_{\rm RC}(\beta)$ is the event that 
\begin{align*}
	  \max \{|  \mathcal{J}_{ A }(0) |,\,
	  |  \mathcal{J}_{ A ^\top }(0) |\}
<
	\beta.
\end{align*}

Suppose $ \tau, \delta, \rho \in ( 0,1 )$ and a growth funciton 
${ g }$ have been chosen from Definition \ref{def: g}. 
From Theorem \ref{thm: vecPartition}, 
for $ I, J \subset [n]$ with $ |I| = |J| = n- \beta +1$, the event $\Omega_{I,J}(\mathcal{T},\mathcal{R})$ can also be written as  
	\begin{align*}
		&\Omega_{I,J}(\mathcal{T},\mathcal{R})\\
	:= &
		\Big\{ \,\forall x \notin \cup_{\lambda \ge 0} \lambda \mathcal{V}
		\mbox{ with } \opn{supp}(x) \subseteq J, \quad
		\left\|  A _{I,J}x \right\|  
		 \ge n^{-1+ o_{ n}(1)} ( pn )^{-1}\left\| x \right\| 
		 ,\, \mbox{and} \\
		& \phantom{\Big\{ \,} \forall x \notin \cup_{\lambda \ge 0} \lambda \mathcal{V}
		\mbox{ with } \opn{supp}(x) \subseteq I, \quad
		\left\|  A _{I,J}^\top x \right\|   
		 \ge n^{-1+ o_{ n}(1)} ( pn )^{-1}\left\| x \right\| 
		 \phantom{AA} \Big\},
			\end{align*}
	where $\mathcal{V} : = \mathcal{V}( \tau, 
	{ g }, \delta, \rho)$. 
	We also define $\Omega_{I,J}(\mathcal{V}, t)$ to be the event that 
\begin{align*}
	\forall x \in \mathcal{V} \cap \mathbb{R}^J,\, 
	\left\|  A _{I, J}x  \right\| 	 
\ge  
	t \left\| x \right\| .
\end{align*}

Notice that if the parameters $ \tau , \delta, \rho$ satisfies 
$$ \tau \in (0, \frac{1}{10}),\, 
	\delta \in (0, \frac{ \tau}{3}),\, 
	\rho\in (0, \frac{1}{10}).$$
Then, by Theorem \ref{thm: vecPartition},
\begin{align*}
	  \mathbb{R}^{J} \backslash \{\vec{0}\} 
= 
	\big( (\cup_{\lambda >0} \mathcal{V}) \cup \mathcal{T} 
		\cup (\cup_{\lambda>0} \mathcal{R}) \big) \cap \mathbb{R}^J.
\end{align*}

If we condition on the event 
$ \Omega_{I,J}( \mathcal{T},\, \mathcal{R}) \cap 
\Omega_{I,J}( \mathcal{V}, \, t)$ with 
$ t \le n^{-1 + o_{n}(1)}( pn )^{-1}$, 
we have 
$$ s_{ n - \beta +1}(  A _{I,J}) 
\ge t.$$ 

(We remark that 
$s_{ n - \beta +1}(  A _{I,J})$
is the least singular value of $ A _{I,J}$
since $ A _{I,J}$ is an $(n - \beta +1) \times 
(n-\beta +1)$ matrix. 
)

Further, by the min-max theorem for singular values, 
it implies that 
\begin{align}
	\label{eq: singularValueComparison}
	s_{ n -\beta+1}(  A  ) 
\ge 
	\min_{ x \in \mathbb{R}^J \backslash \{ \vec{0} \} }
	\frac{\left\|  A _{[ n ] J}x \right\|  
	}{\left\| x \right\| }
\ge
	s_{ n - \beta + 1 }(  A _{I,J} ) 
\ge 
	t,
\end{align}
provided that $t \le n^{ -1 + o_{ n }(1) }( pn )^{-1}$.

\begin{proof} [Proof of Theorem \ref{thm: main}]

Fix $t \le n^{ -1 + o_{n}(1)}$, 
\begin{align}
	 \mathbb{P} \left\{    
		s_{ n- \beta + 1}(  A  ) 
	\le 
		t 
	\right\}
 \le &
	\mathbb{P} \big \{ \forall
		J \subset [ n] \mbox{ with } 	
		|J| =  n - \beta + 1 , 
		s_{min}( A _{[ n] ,J})  
		\le t 
	 \big \} \nonumber 	\\
\le &
	\mathbb{P} \big \{ \forall
		I, J \subset [ n] \mbox{ with } 	
		|I|=|J| =  n - \beta + 1 , 
		s_{min}( A _{ I ,J})  
		\le t 
      \big \}.  \label{eq: thmMainEq1}
\end{align}

Next, we split the estimate accronding to $\Omega( \mathcal{T}, \mathcal{R} )$ to get 
\begin{align*}
	  \eqref{eq: thmMainEq1}
 \le &
 	\mathbb{P} \{ \Omega( \mathcal{T}, \mathcal{R} )^c \} 
  \\
 &+
 	\mathbb{P} \bigg\{ 
		\left\{   \forall I,J \subset [ n] 
	\mbox{ with } 	
		|I| = |J| =  n - \beta + 1 , 
		s_{min}( A _{I,J})  \le 
		t 
		\right\} \cap 
		\Omega(\mathcal{T}, \mathcal{R})
	 \bigg\} 	\\
 \le &
	(1 + o_{n}) \mathbb{P}( \Omega_{\rm RC}^c)\\
&+
	\mathbb{P} \bigg\{ 
		\left\{   \forall I,J \subset [ n] 
	\mbox{ with } 	
		|I| = |J| =  n - \beta + 1 , 
		s_{\rm min}( A _{I,J})  \le 
		t 
		\right\} \cap 
		\Omega(\mathcal{T}, \mathcal{R})
	 \bigg\} .	
\end{align*}
For the second summand, we could simplify it in the following way 
\begin{align}
	& \mathbb{P} \bigg\{ 
		\left\{   \forall I,J \subset [ n] \mbox{ with } 	
		|I| = |J| =  n - \beta + 1 , 
		s_{\rm min}( A _{I,J})  \le t 
		\right\} \cap 
		\Omega(\mathcal{T}, \mathcal{R})
	\bigg\} \nonumber \\
\le &
	\sum_{I',J' \subset [n]} 
	\mathbb{P} \bigg\{ 
		\left\{   \forall I,J \subset [ n] \mbox{ with } 	
		|I| = |J| =  n - \beta + 1 , 
		s_{\rm min}( A _{I,J})  \le t 
		\right\} \cap 
		\Omega_{I',J'}(\mathcal{T}, \mathcal{R})
	\bigg\} \nonumber \\
\le &
	\sum_{I',J' \subset [n]} 
	\mathbb{P} \bigg\{ 
		\left\{   
		s_{\rm min}( A _{I',J'})  \le t 
		\right\} \cap 
		\Omega_{I',J'}(\mathcal{T}, \mathcal{R})
	\bigg\} \nonumber \\
\le & 
	{n \choose \beta-1}^2 \mathbb{P} \bigg\{ 
		\left\{   
		s_{\rm min}( A _{I_0,J_0})  \le t 
		\right\} \cap 
		\Omega_{I_0,J_0}(\mathcal{T}, \mathcal{R})
	\bigg\}.
\end{align}	
where the summantion is taking over all possible 
$I',J' \subset [n]$ with 
$|I'| = |J'| = n - \beta +1$ and 
$I_0 = J_0 = [ n - \beta + 1]$ .

Next, 
\begin{align*}
	& \mathbb{P} \big\{ 
		\left\{   
		s_{\rm min}( A _{I_0,J_0})  \le t 
		\right\} \cap 
		\Omega_{I_0,J_0}(\mathcal{T}, \mathcal{R})
	\big \} \\
= &
	 \mathbb{P} \big \{ 
		\left\{   
			\exists x \in \mathbb{R}^{J_0} \cap  \mathcal{V} ,\, 
			\left\|  A_{I_0,J_0}x \right\| \le t \left\| x \right\| 
		\right\} \cap 
		\Omega_{I_0,J_0}(\mathcal{T}, \mathcal{R})
	\big \} \\
\le & 
	\mathbb{P} \big \{ 
		\left\{   
			\exists x \in \mathbb{R}^{J_0} \cap  \mathcal{V} ,\, 
			\left\|  A_{I_0,J_0}x \right\| \le t b_n
		\right\} \cap 
		\Omega_{I_0,J_0}(\mathcal{T}, \mathcal{R})
	\big \}
\end{align*}
where $b_{ n } := 
	\sqrt{
		\sum_{i=1}^{ n } g( \frac{ n  }{ i } )^2
	} \ge \max_{x \in \mathcal{V}} \left\| x \right\|  $ .
Next, we will adjust the second event in the inclusion,
\begin{align*}
	& \mathbb{P} \left\{ 
		\left\{   
			\exists x \in \mathbb{R}^{J_0} \cap  \mathcal{V} ,\, 
			\left\|  A _{I_0,J_0}x \right\| \le t b_n
		\right\} \cap 
		\Omega_{I_0,J_0}(\mathcal{T}, \mathcal{R})
	\right\}\\
\le &
	\mathbb{P} \left\{ 
		\left\{   
			\exists x \in \mathbb{R}^{J_0} \cap  \mathcal{V} ,\, 
			\left\|  A _{I_0,J_0}x \right\| \le t b_n
		\right\} \cap 
		\left\{ 
			\forall y \in \mathbb{R}^{I_0} \backslash 
				\cup_{\lambda \ge 0} \lambda\mathcal{V} ,\, 
				\left\|  A^{\top}_{I_0, J_0}y  \right\| >0
		\right\} 
	\right\} \\
\le &
	\mathbb{P} \left\{ 
		\left\{   
			\exists x \in \mathbb{R}^{J_0} \cap  \mathcal{V} ,\, 
			\left\|  A _{I_0,J_0}x \right\| \le t b_n
		\right\} 
	\biggm \vert
		\left\{ 
			\forall y \in \mathbb{R}^{I_0} \backslash 
				\cup_{\lambda \ge 0} \lambda\mathcal{V} ,\, 
				\left\|  A^{\top}_{I_0, J_0}y  \right\| >0
		\right\} 
	\right\}. 
\end{align*}

The next step is to apply Theorem \ref{thm: LT} for the matrix $  A_{ I_0,J_0 }$. 
Let $ \delta' = \frac{ \delta }{ 2 }$,
and set $ \tilde { g }(t):= { g }( \frac{ n}{ n  - \beta  } t )$. 

We apply Theorem \ref{thm: LT} with $R$ to be some large constant multiple of 
$C\beta$ to get 
\begin{align}
	\label{eq: thmMainAfterThmLT}
	& \mathbb{P} \left\{ 
		\left\{   
			\exists x \in \mathbb{R}^{J_0} \cap  \mathcal{V} ,\, 
			\left\|  A _{I_0,J_0}x \right\| \le t b_n
		\right\} 
	\biggm \vert
		\left\{ 
			\forall y \in \mathbb{R}^{I_0} \backslash 
				\cup_{\lambda \ge 0} \lambda\mathcal{V} ,\, 
				\left\|  A^{\top}_{I_0, J_0}y  \right\| >0
		\right\} 
	\right\} \\
	\le &
	\exp(-10C \beta pn) + n^{1 +o_{ n }(1) } (  pn )^7
	t, \nonumber
\end{align}
where the inequality holds due to $ b_n \le n^{1 + o_{ n  }(1)} (  pn  )^7$ 
from Proposition \ref{prop: specificGrowthFunction} and 
\begin{align*}
	\mathbb{P}\{ \Omega( \mathcal{V}^c ) \} = 1 - o_{  n  }(1). 
\end{align*}
The above equality holds is due to 
$ \mathbb{P}\{ \Omega( \mathcal{V}^c ) \} \ge  \mathbb{P}\{\Omega_{  A_{ I_0, J_0 } }( \mathcal{T}, \mathcal{R} )\} $, where
$\Omega_{  A_{ I_0, J_0 } }( \mathcal{T}, \mathcal{R} )$ is the event with the same description as $ \Omega( \mathcal{T}, \mathcal{R})$ 
but applied to $  A _{ I_0, J_0 }$ and $ \beta = 1$.

With the standard estimate for Binoimal coefficients
${ n  \choose \beta -1 }^2 \le \exp(2 ( \beta -1) \log( en) )$, 
we conclude from \eqref{eq: thmMainAfterThmLT}, for $ 0 \le t \le n^{-1 + o_{ n  }(1)}$,
\begin{align*}
	& \mathbb{P} \left\{    
		s_{ n- \beta + 1}(  A  ) 
	\le 
		t 
	\right\} 
 \le 
	(1 + o_{n}) \mathbb{P}( \Omega_{\rm RC}^c)
	+  n^{2\beta-1 + o_{ n }(1) } (  pn  )^{7} t,
\end{align*}
since $ \exp( - C pn ) = o(\mathbb{P}( \Omega_{\rm RC }^c )  )$. 
Now, by a change of variable on $t$ we obtain the statement of the theorem. 

\end{proof}

\appendix
\section{Probability Estimate}
\subsection{Non-invertibility of sparse Bernoulli Matrix with $p$ close to $\log n /n$}
\begin{lemma} \label{lem: nonemptyEstimate}
	Suppose $p= p_n \in (0,1)$ satisfies 
	$ \log n  \le pn \le C\log n  $ where $C>1$ is an arbitrary 
	constant. 

	Let $A$ be an $n\times n$ Bernoulli($p$) matrix.
	Then, when $n$ is sufficiently large,
	\begin{align*}
	&	\mathbb{P} \left\{ 
		\exists i \in [n] \mbox{ s.t. } {\bf R}_i(A) = \vec{0}
		\mbox{ or } 
		\exists j \in [n] \mbox{ s.t. } {\bf C}_j(A) = \vec{0}
		\right\} 	\\
	=& 
		(1 + o_n(1) )
		\left( 1 - \left(  1 - (1- p)^p
		\right)^{2 n} \right) .
	\end{align*}
	Furthermore, 
	$ 
		\left( 1 - \left(  1 - (1- p)^p
		\right)^{2 n} \right) 
	= 
		O_n( n(1-p)^n ) $.
\end{lemma}

To prove Lemma 2.5, we will need the following Proposition: 

\begin{proposition} \label{prop: nonemptyEstimate}
	Let $C>1$ be any constant. 
	Suppose $m$ is a sufficiently large positive integer, 
	and $p$ is a parameter satisfying 
	$ \log(m) \le pm \le C\log(m)$.
	For $ 0 \le k \le \log(m)$, let $A_k$ be a 
	$( m - k ) \times m$ Bernoulli($p$) matrix. We have
	\begin{align*}
		\mathbb{P} \left\{   \mathcal{J}_{A_k}(0) > 0 \right\} 	
	&= 
		( 1 + o_m( m^{-1/2} ) )	
		\mathbb{P} \left\{   \mathcal{J}_{A_0}(0) > 0 \right\} 	.
	\end{align*}
 
\end{proposition}

\begin{proof}
	Let 
	\begin{align*}
		p_k : = \mathbb{P} \left\{  
			{\bf C}_1 (A) = \vec{0}
		 \right\} = (1-p)^{m-k}.
	\end{align*}
	From the constraints of $p$ and $k$, we have 
	$ p_k = p_0 (1 + O(pk)) \le \frac{2}{m}$.

	Then, 
	\begin{align*}
		\mathbb{P} \left\{  \mathcal{J}_{A_k}(0) > 0 \right\} 	
	&=
		\sum_{ j = 1 }^m 
		{ m \choose j } p_k^j ( 1 - p_k )^{ m - j }.
	\end{align*}

	Our goal is to show that the summation is dominated by 
	the first $\log(m)$ terms. 

	For $j \ge \log(m)$, 
	\begin{align*}
		\frac{ { m \choose j+1 } p_k^{ j + 1 }( 1 - p_k )^{ m - j - 1} }{
		{ m \choose j } p_k^j ( 1 - p_k )^{ m - j } }
	& = 
		\frac{m-j}{ j + 1} 
		\frac{ p_k }{ 1-p_k}
	 \le 
		\frac{m}{ j } p_k
	\le
		1/2
	\end{align*}
	due to $ p_k = ( 1- p)^m (1-p)^{-k} 
	\le  \frac{1}{m} (1 + O(pk) ) \le \frac{2}{m}$.
	Thus, the tail of the sequence 
	$ \left\{  { m \choose j } p_k^j ( 1 - p_k )^{ m - j }
	\right\}_{j\ge 1} $ decays geometrically at the rate $\frac{1}{2}$ 
	when $j \ge \log(m)$.

	Then, we compare the $j$th term to the $1$st term with 
	$j\ge \log(m)$:
	\begin{align*}
		\frac{ { m \choose j } p_k^{ j  }( 1 - p_k )^{ m - j } }{
		{ m \choose 1 } p_k ( 1 - p_k )^{ m - 1 } }
	& \le
		\frac{1}{m} \left( \frac{em}{j} \right)^j p_k^{j-1} 
	\le 
		\left( \frac{2e}{j} \right)^j  
    \le \frac{1}{2m},
	\end{align*}
	where $ p_k \le \frac{2}{m}$ is applied.
	Therefore, we have 
	\begin{align*}
		\mathbb{P} \left\{  \mathcal{J}_{A_k}(0) > 0 \right\} 	
	& \le  
		\sum_{ j = 1 }^{ \lceil \log(m) \rceil } 
		{ m \choose j } p_k^j ( 1 - p_k )^{ m - j } 
	+ 
		\frac{1}{m}
		\mathbb{P} \left\{  \mathcal{J}_{A_k}(0) > 0 \right\}. 
	\end{align*}

	Next, we will compare the first $\log(m)$ summands to those 
	in the case when $k=0$.	
	For $j, k \le \log(m)$, the following estimates hold:
	\begin{align*}
		pk & = o_m( m^{-1/2} ), &
		p_k pk m & = o_m( m^{-1/2}), \mbox{       and } &
		pk j & = o_m(m^{-1/2}).
	\end{align*}
	
	With $p_k = p_0 (1-p)^{-k}$, when $ 1 \le j \le \log(m)$, 
	\begin{align*}
		{ m \choose j } p_k^j ( 1 - p_k )^{ m - j } 
	=
		{ m \choose j } p_0 (1 - p_0)^{ m - j } (1 + o( m^{-1/2} ) ).
	\end{align*}

	Therefore, we conclude that 
	\begin{align*}
		\mathbb{P} \left\{   \mathcal{J}_{A_k}(0) > 0 \right\} 	
	&= 
		( 1 + o_m( m^{-1/2} ) )	
		\mathbb{P} \left\{   \mathcal{J}_{A^0}(0) > 0 \right\} 	.
	\end{align*}

\end{proof}

Now we are ready to prove Lemma \ref{lem: nonemptyEstimate}.

\begin{proof}[ Proof of Lemma \ref{lem: nonemptyEstimate}]
	Let $O_R$ be the event that 
	\begin{align*}
		\exists i \in [n] \mbox{ s.t. } {\bf R}_i(A) = \vec{0}
	\end{align*}
	and $O_C$ be the event that 
	\begin{align*}
		\exists j \in [n] \mbox{ s.t. } {\bf C}_i(A) = \vec{0}.
	\end{align*}

	For $ S \subset [ n ]$, let 
	$ P_S $ be the probability that
	\begin{align*}
		\forall i \in S, {\bf R}_i(A) = \vec{0}
		\mbox{ and } 
		\exists j \in [n] \mbox{ s.t. }  {\bf C}_j(A) = \vec{0}.
	\end{align*}
	By the standard inclusion-exclusion formula, we have 
	\begin{align*}
		\mathbb{P} \left\{ O_R \cap O_C		\right\} 
	& = 
		\sum_{ S \subset [ n ], S \neq \emptyset  } (-1)^{|S|+1} P_S.
	\end{align*}

	Due to the distribution of $A$ is invaraiant under 
	row permutations, $P_S$ depends only on $|S|$. 

	For $k \in [n]$, let $A_k$ be an $n-k$ by $k$ Bernoulli($p$)
	matrix. Suppose $S\subset [n]$ is a subset of size $k$. 
	Let $O_S$ be the event that 
	$\forall i \in S$, ${\bf R}_i(A) = \vec{0}$.
	Conditioning on $O_S$, the submatrix of $A$ 
	obtained by restricting its rows to $[n] \backslash S$
	has the same distribution as that of $A_k$. Thus, 
	we have 
	\begin{align*}
		P_S = (1-p)^{nk} 
		\mathbb{P} \left\{  \mathcal{J}_{A_k}(0) >0 \right\}
	\end{align*}
	and we conclude that 
	\begin{align*}
		\mathbb{P} \left\{ 	O_R \cap O_C	\right\} 
	& = 
		\sum_{ k \in  [ n ] } 
		(-1)^{k+1} 
		{n \choose k }
		(1-p)^{nk}
		\mathbb{P} \left\{  \mathcal{J}_{A_k}(0) >0 \right\}.
	\end{align*}

	Notice that by the same argument we can deduce that 
	\begin{align*}
		\mathbb{P} \left\{ O_C \right\} 
	& = 
		\sum_{ k \in  [ n ] } 
		(-1)^{k+1} 
		{n \choose k }
		(1-p)^{nk}.
	\end{align*}

	Roughly speaking, if
	$
	\mathbb{P} \left\{  \mathcal{J}_{A_k}(0) >0 \right\} = (1+o_n(1) )
	 \mathbb{P} \left\{  O_R \right\} $ for all values of $k$. Then, 
	 it implies that 
	 $ 
		\mathbb{P} \left\{ O_R \cap O_C \right\}  
	=	
		(1 + o_n(1) ) \mathbb{P}\{O_R\} \mathbb{P} \{ O_C \}
	 $. In other words, the two events are approximately independent. 
	However, it cannot be true when $k$ is large. 

	What we will do is to  split the summation into two parts: 
	$k < \log n  $ and $k \ge \lceil \log n  \rceil$.

	For $ k \le \log n $, we will show that 
	$\mathbb{P} \left\{  \mathcal{J}_{A_k}(0) >0 \right\} = (1+o_n(1) )
	 \mathbb{P} \left\{  O_R \right\} $. This is presented 
	 separately in Proposition \ref{prop: nonemptyEstimate}.
	And for $k > \lceil \log n  \rceil $, we will show the 
	summation in this part is negligible comparing to 
	the summation for $ k < \log n  $. Due to 
	the signs are alternating in the summations, 
	 we need to argue in a careful way: 
	First, show the leading term of the summation is 
	comparable to the summation for $ k < \log n  $. Second, 
	show that the summation for $k \ge \log n $ is negligible 
	comparing to the leading term. 

	By Proposition \ref{prop: nonemptyEstimate}, 
	\begin{align*}
			\mathbb{P}\{  \mathcal{J}_{A_k}(0) >0 \} 
		=
			(1+o_n( n^{-1/2} ) ) 
			\mathbb{P}\{  O_R \} 
	\end{align*}
	for $k \le \lceil \log n  \rceil$. Hence, 
	\begin{align*}
	&	
		\sum_{ k \in  [ \lceil \log n  \rceil ] } 
		(-1)^{k+1} 
		{n \choose k }
		(1-p)^{nk}
		\mathbb{P} \left\{  \mathcal{J}_{A_k}(0) >0 \right\}
	\\
	=&
		\sum_{ k \in  [ \lceil \log n  \rceil ] } 
		(-1)^{k+1} 
		{n \choose k }
		(1-p)^{nk}
		\mathbb{P} \left\{ O_R \right\}
	\\
	& +
		o( n^{-1/2} )
		\sum_{ k \in  [ \lceil \log n  \rceil ] } 
		{n \choose k }
		(1-p)^{nk}
		\mathbb{P} \left\{ O_R \right\}.
	\end{align*}
	
	For $k \ge 2$, 
	\begin{align} \label{eq: colrowEstimate}
		\frac{ { n \choose k } (1-p)^{nk} }{
		{ n \choose k-1} (1-p)^{n(k-1)} }	
	=
		\frac{n-k+1}{k} (1-p)^n
	\le
		\frac{1}{k} n\exp(-pn)
	\le 
		\frac{1}{k},
	\end{align}
	where the last inequality holds due to $pn \ge \log n $.
	As an alternating sequence whose absolute values are 
	decreasing,
	\begin{align*}
	&	\sum_{ k \in  [ \lceil \log n  \rceil ] } 
		(-1)^{k+1} 
		{n \choose k }
		(1-p)^{nk}
		\mathbb{P} \left\{ O_R \right\}
	\\
	\ge &
		{ n \choose 1} (1-p)^{n} 
		\mathbb{P} \left\{ O_R \right\}
	-
		{ n \choose 2} (1-p)^{2n} 
		\mathbb{P} \left\{ O_R \right\}
	\\
	\ge &
		\frac{1}{2} 
		{ n \choose 1} (1-p)^{n} 
		\mathbb{P} \left\{ O_R \right\}
	\end{align*}

	Notice that, \eqref{eq: colrowEstimate} also implies 
	\begin{align*}
		\frac{ { n \choose k } (1-p)^{nk} }{
		{ n \choose 1} (1-p)^{n} } \le \frac{1}{k!}
	\end{align*}
	and thus 
	\begin{align*}
		\sum_{ k \in [ \lceil \log n  \rceil ] } 
		{n \choose k }
		(1-p)^{nk}
		\mathbb{P} \left\{ O_R \right\}
	\le
		e
		{ n \choose 1} (1-p)^{n} 
		\mathbb{P} \left\{ O_R \right\}.
	\end{align*}

	Therefore, we conclude that 
	\begin{align*}	
		\sum_{ k \in  \lceil \log n  \rceil } 
		(-1)^{k+1} 
		{n \choose k }
		(1-p)^{nk}
		\mathbb{P} \left\{  \mathcal{J}_{A_k}(0) >0 \right\}
	& \ge 
		\frac{1}{3}
		{ n \choose 1} (1-p)^{n} 
		\mathbb{P} \left\{ O_R \right\}.
	\end{align*}
	
	Now we turn to the summation for $k \ge \lceil \log n  \rceil $, 
	\begin{align*}
	&	\sum_{ k \ge  \lceil \log n  \rceil  } (-1)^{k+1}
			{n \choose k }
			(1-p)^{nk}
			\mathbb{P} \left\{  \mathcal{J}_{A_k}(0) >0 \right\}
	\\
	\le &
		{n \choose 1}(1-p)^n
		\sum_{ k \in [ \lceil \log n  \rceil ] } \frac{1}{k!}
	\\ 
	\le & 
		{n \choose 1}(1-p)^n
		\exp( - \frac{1}{2} \log( \log n ) \log n  ) .
	\end{align*}

	Next, $ \mathbb{P}( {\bf C}_j(A) = \vec{0} ) = (1-p)^n$
	for $j \in [ n] $ and the independence of columns implies 
	$ \mathbb{P} \left\{ O_R \right\} = 
		1 - ( 1 - (1-p)^n )^n$. 

	For $ x \in (0, 1/2 )$, we have 
	\begin{align*}
		1-x = \exp( -x+ O(x)^2).
	\end{align*}
	Conversely, for $ 0 \le x \le 1.5$, 
	$1 - \exp( -x )  = O(x)$. 

	Applying these inequalites to the estimate of 
	$ \mathbb{P} \left\{ O_R \right\}$, we obtain
	\begin{align} \label{eq: OREstimate}
		\mathbb{P} \left\{ O_R \right\}
	& = 
		1 - \exp\left( - 
			(1 + o_n(1) ) 
			\exp(-pn) n \right)
	=  
		O( \exp(-pn) n )
	>
		\exp( - C \log n  ),
	\end{align}
	where we rely on $ \exp(-pn)n \le 1$.

	And it allows us to compare the summation for 
	$ k \ge \lceil \log n  \rceil$ to the first term of 
	of the summation:
	\begin{align*}	
		\sum_{ k \ge  \lceil \log n  \rceil  } (-1)^{k+1}
			{n \choose k }
			(1-p)^{nk}
			\mathbb{P} \left\{  \mathcal{J}_{A_k}(0) >0 \right\}
	&=
		o_n( n^{-1/2} ) 
		{ n \choose 1} (1-p)^{n} 
		\mathbb{P} \left\{  O_R \right\}.
	\end{align*}

	Therefore, we can conclude that 
	\begin{align*}
		\mathbb{P} \left\{ O_R \cap O_C \right\} 
	&=
		( 1 + o_n( n^{-1/2} ) 
		\mathbb{P} \left\{ O_R \right\}.
		\left(  
			\sum_{ k \in [ \lceil \log n  \rceil ] } (-1)^{k+1}
			{n \choose k }
			(1-p)^{nk}
		\right).
	\end{align*}

	The same approach could show that 
	\begin{align*}
		\mathbb{P} \left\{ O_C \right\} 
	&=
		( 1 + o_n( n^{-1/2} ) 
		\left(  
			\sum_{ k \in [ \lceil \log n  \rceil ] } (-1)^{k+1}
			{n \choose k }
			(1-p)^{nk}
		\right).
	\end{align*}
	Finally, we obtain
	\begin{align*}
		\mathbb{P} \left\{ O_R \cap O_C \right\} 
	& = 
		( 1 + o_n( n^{-1/2} ) ) 
		\mathbb{P}\left\{ O_R \right\}
		\mathbb{P}\left\{ O_C \right\}
	=
			( 1 + o_n( n^{-1/2} ) ) 
		(\mathbb{P}\left\{ O_R \right\} )^2  .
	\end{align*}

	At this point we are ready to estimate 
	$ \mathbb{P} \left\{ O_R \cup O_C \right\} $ :
	\begin{align*}
		\mathbb{P} \left\{ O_R \cup O_C \right\}
	&=
		\mathbb{P} \left\{ O_R \right\} 
	+	
		\mathbb{P} \left\{ O_C \right\} 
	-
		\mathbb{P} \left\{ O_R\cap O_C \right\} 
	\\
	& =
		2 \mathbb{P} \left\{ O_R \right\} 
		- ( 1 + o_n(n^{-1/2} )) \mathbb{P} \left\{ O_R \right\}^2
	\\
	& = 
		( 1 + o_n(n^{-1/2} \mathbb{P}\left\{ O_R \right\}  ))
		(2 \mathbb{P} \left\{ O_R \right\}  
		- \mathbb{P} \left\{ O_R \right\}^2).
	\end{align*}
	Substituting $ \mathbb{P} \left\{ O_R \right\} = 
		1 - ( 1 - (1-p)^n )^n	
	$, we obtain the statement of the lemma: 
	\begin{align*}
		\mathbb{P} \left\{ O_R \cup O_C \right\}
	& = 
		( 1 + o(n^{-1/2} \exp(-pn)n ) )	
		\left( 1 - \left( 1 - (1-p)^n \right)^{2n}
		\right) .
	\end{align*}

	It remains to show that $ \mathbb{P} \left\{  O_R \cup O_C \right\} 
	= O_n( n\exp(-pn))$, but it is immediate due to $ \mathbb{P}\left\{ O_R \right\} 
	= \mathbb{P} \left\{ O_C \right\} $ and $ \mathbb{P} \left\{ O_R \right\} 
	= O_n( n\exp(-pn))$ from \ref{eq: OREstimate}.
	
\end{proof}

\subsection{Tail Estimates for Binomial and Hypergeometric distribution.}
\begin{proof}[Proof Proposition \ref{prop: Binomial}]
	Let 
	\begin{align*}
		p_k := \mathbb{P}\{ Y = k \} = {n \choose k} p^k (1-p)^{n-k}.
	\end{align*}
	
	The expected value of \( Y \) is \( pn \). For \( k \) sufficiently smaller than \( pn \), we aim to demonstrate that 
	\( \mathbb{P}\{ Y \le k \} \) is dominated by \( p_k \). Conversely, for \( k \) sufficiently large, \( \mathbb{P}\{ Y \ge k \} \) is dominated by \( p_k \). Once it is established, then it is sufficient to bound \( p_k \) from above.
	
	For \( k \le \frac{1}{2}pn \),
	\begin{align*}
		\frac{q_{k-1}}{q_{k}} = \frac{k}{n-k+1} \frac{1-p}{p} \le \frac{\frac{1}{2}pn}{n-\frac{1}{2}pn+1} \frac{1-p}{p} \le \frac{1}{2},
	\end{align*}
	where the first inequality follows from the fact that \( x \mapsto \frac{x}{n+1-x} \) is monotone increasing for \( x \in [0, n+1) \). Thus, we can conclude that for \( 0 \le k \le \frac{1}{2}pn \),
	\[
		\mathbb{P}\{ Y \le k \} \le \sum_{s=0}^k q_{k-s} \le p_k \sum_{s=0}^\infty \left(\frac{1}{2}\right)^s \le 2p_k.
	\]
	
	Similarly, for \( k \ge 2pn \),
	\begin{align*}
		\frac{q_{k+1}}{q_{k}} = \frac{n-k}{k+1} \frac{p}{1-p} \le \frac{n-2pn}{2pn+1} \frac{p}{1-p} \le \frac{1}{2},
	\end{align*}
	which in turn implies that \( \mathbb{P}\{ Y \ge k \} \le 2p_k \).
	
	It remains to estimate \( p_k \). By the standard estimate \( {n \choose k} \le \left(\frac{en}{k}\right)^k \),
	\begin{align} \label{eq:BinomialExact}
		p_k \le \left( \frac{enp}{k(1-p)} \right)^k (1-p)^n,
	\end{align}
	and the statement of the Proposition follows.
\end{proof}	

\begin{proof}[Proof of Proposition \ref{prop: HyperGeom}]
	By a counting argument, 
	we have, for \( l \in [k] \),
	\begin{align*}
		\mathbb{P}\{ |U \cap [m]| = l \} = \frac{{m \choose l} {n - m \choose k - l}}{{n \choose k}}.
	\end{align*}
	
	Recall the standard estimate for combinations:
	\[
		\left( \frac{a}{b} \right)^b \le {a \choose b} \le \left( \frac{ea}{b} \right)^b
	\]
	for positive integers \( a \ge b \ge 1 \). If we simply use this bound the derive the tail bound, then the gap \( e^b \) is too large for us. We need a slightly better bound. For a positive integer \( s \ge 1 \),
	\[
		\sqrt{2\pi s} \left( \frac {t_2}{e} \right)^s 
	\le 
		s! \le e \sqrt {t_2} \left( \frac {t_2}{e} \right)^s.
	\]

% 	\[
% 		s! 
% 	= 
% 		\Big(1 + O\Big(\frac{1} {t_2} \Big) \Big) 
% 		\sqrt{2\pi s} \left( \frac {t_2}{e} \right)^s, 
% 	\quad \text{and} \quad
% 		\sqrt{2\pi s} \left( \frac {t_2}{e} \right)^s 
% 	\le 
% 		s! \le e \sqrt {t_2} \left( \frac {t_2}{e} \right)^s.
% 	\]
	Applying the bounds to estimate \( {n \choose k} \), we have
	\begin{align*}
		{n \choose k} 
	 = 
		\frac{n!}{k! (n-k)!}
	\ge  &
		\frac{\sqrt{2\pi}}{e^2} \sqrt{\frac{n}{k(n-k)}} \left( \frac{n}{k} \right)^k \left( \frac{n}{n-k} \right)^{n-k} \\
	\ge &
		\frac{\sqrt{2\pi}}{e^2\sqrt{k}}  
		\left( \frac{n}{k} \right)^k \left( \frac{n}{n-k} \right)^{n-k}.
	\end{align*}
	
	The last term \( \left( \frac{n}{n-k} \right)^{n-k} \) can be simplified into the following form:
	\begin{align*}
		\left( \frac{n}{n-k} \right)^{n-k} 
	= 
		\left(1 + \frac{k}{n-k}\right)^{n-k} 
	\ge 
		e^k \exp\left(-\frac{k^2}{n-k}\right) 
	= &
		e^k \exp\left(-\frac{2k^2}{n}\right) \\
	& \phantom{AAA AAA}\ge 
	 	e^{k-1},
	\end{align*}
	where we used \(\log(1 + x) > x - x^2\) for \(x > 0\) and the assumption \( 2k^2 \le n \). Therefore, we have
	\begin{align*}
		{n \choose k} 
	\ge 
		\frac{\sqrt{2\pi}}{e^3 \sqrt{k}} \left( \frac{en}{k} \right)^k.
	\end{align*}
	By a similar computation, for \(1 \le l < k\), we have
	\begin{align*}
		{n - m \choose k - l} \le \frac{1}{\sqrt{2\pi} \sqrt{k - l}} \left( \frac{e(n - m)}{k - l} \right)^{k - l},
	\end{align*}
	where \( \left( \frac{n - m}{n - m - k + l} \right)^{k - l} \le 	
	\left( \frac{n/2}{n/2-k+l} \right)^{k - l} \le 
	\left( \frac{n/2}{n/2-n/4} \right)^{k - l} 
	\le e^{k - l} \) is used.
	
	Combining these two estimates and \( {m \choose l} \le \left( \frac{em}{l} \right)^l \) to get
	\begin{align*}
		\mathbb{P}\{ |U \cap [m]| = l \} 
	= & 
		 \frac{{m \choose l} {n - m \choose k - l}}{{n \choose k}}\\
	\le &
		C \sqrt{\frac{k}{k - l + 1}} \left( \frac{mk}{ln} \right)^l \left( \frac{k}{k - l} \right)^{k - l} \left( \frac{n - m}{n} \right)^{k - l} \\
	\le &
		C \sqrt{\frac{k}{k - l + 1}} \left( \frac{mk}{ln} \right)^l \left( 1 + \frac{l}{k - l} \right)^{k - l} \cdot 1 \\
	\le& 
		C \sqrt{\frac{k}{k - l + 1}} \left( \frac{emk}{ln} \right)^l,
	\end{align*}
	where \(C > 0\) is some universal constant. 
	
	Due to \( l \mapsto C \sqrt{\frac{k}{k - l + 1}} \left( \frac{emk}{ln} \right)^l \) decaying geometrically, we can find a suitable \(C_{\text{hg}} > 1\) and replace \(e\) by \(3\) to obtain the statement of the Proposition.
	
\end{proof}

\section{Proof of Theorem \ref{thm:R-vector}, and Theorem \ref{thm:T_2T_3-vector} for $\cal R$, $\cal T_2$, and $\cal T_3$ vectors}
In this section, we will prove Theorem \ref{thm:R-vector} and Theorem \ref{thm:T_2T_3-vector}. 

The proof of Theorem \ref{thm:R-vector} and \ref{thm:T_2T_3-vector}
are essentially the same as the proofs appeared in Litvak-Tikhomirov 
\cite{LT20} with minor modification due to the definition of $ \mathcal{T}$ 
and $ \mathcal{R}$ vectors and the involvement of $\beta$. 

Both theorems can be proved by a net-argument approach. 
Briefly speaking, for each subcolelction $\cal R, \cal T_2$, and $\cal T_3$, we will find a net 
$ \mathscr{N}$. Then, we show that for each $x \in \mathscr{N}$, we can bound $ \left\| Ax \right\| $ quantitatively away from $0$ with high probability, say $1-q$ with $q = o(1)$.   Then, with probability at most 
$q|N|$, there exists $y$ in the corresponding set 
($\mathcal{R}, \mathcal{T}_2,$ or $\mathcal{T}_3$)
such that $ \left\| Ay \right\| $ is small.
Certainly, the net $\mathscr{N}$ needs to satsfies one crucial property: For each $x$ in the 
corresponding set
( $\mathcal{R}, \mathcal{T}_2,$ or $\mathcal{T}_3$ ),
there exists $y \in \mathscr{N}$ such that $ \left\| A(x-y) \right\| $
is very small.

To achieve that, we need the following:

Fix $\beta$ to be a positive integer, we define the event  
\begin{align} \label{def: omegaNorm}
	\Omega_{\rm norm} =& \Omega_{\rm norm}(C_{\rm norm}) \\
\nonumber
:=& \Big\{ \|A - \mathbb{E} A\| \le C_{\rm norm} \sqrt{pn} 
\quad \mbox{and} \quad 
\|A\| \le C_{\rm norm} \sqrt{pn} + pn 
\Big\} 
\end{align}
for a parameter $C_{\rm norm} \ge 1$. 
By Lemma \ref{lem: OperatorNorm}, we can choose  
$ C_{\rm {norm}}$ depending only on $\beta$, so that 
\begin{align} \label{eq: omegaNormEstimate}
	\mathbb{P}\left\{ \Omega_{\rm norm}^c \right\} 
	= o(\Prob\{ \Omega_{\rm RC}^c\}).
\end{align}
From now on, we fix such constant $C_{\rm norm}$. 

% $\Omega_{\rm norm}$ be the event that 
% $ \left\|  A  - \mathbb{E} A  \right\| \le C_{\rm{norm}}
% \sqrt{  pn  } $ and 
% $ \left\|  A  \right\| \le C_{\rm {norm}}\sqrt{ 
% pn }+ pn $ 
% where $C_{\rm {norm}} > 0$  satisfies
% \begin{align}
	% \mathbb{P}\left\{ \Omega_{\rm norm}^c \right\} 
	% \le \exp(- 3 C_{ \ref{thm: T_1}}\beta pn ) \le 
		% ( 1 - p )^{ 3\beta n },
% \end{align}
% where $C_{ \ref{thm: T_1}} \ge 1$ is the smallest constant so that 
% $( 1 - p )^{ n } \ge \exp( - C_{  \ref{thm: T_1} } pn) $ holds for 
% $ \frac{ \log( n ) }{ n  } \le p \le \frac{ 1 }{ 2 }$ 
% By Lemma \ref{lem: OperatorNorm}, we can choose  
% $ C_{\rm {norm}}$ depending only on $\beta$. 

Also, we want $ \mathscr{N} $ to be a net in the corresponding norm: 
for $x \in \mathbb{R}^{ n }$, 
\begin{align*}
	| \left\| x \right\| | : =  \left\| P_{\bf e} x \right\| + 
	\sqrt{  pn  } \left\| P_{\bf e^\perp }x \right\|,
\end{align*}
where $ { \bf e } \in S^{ n - 1} $ is the unique vector 
such that every component equals to $ \frac{1}{\sqrt{ n} } $,
$P_{\bf e}$ is the orthogonal projection to ${\bf e}^\perp$, 
and $P_{ \bf e^\perp}$ orthogonal projection to the span of ${ \bf e }$.
Instead of using the standard Euclidean norm, the new norm is defined 
due to $ \|  A  {\bf e} \|$ is much 
larger than $\|  A  v\|$ for 
$v \in  S^{ n -1 } \cap {\bf e}^\perp$.)

This section breaks into three parts. 

The first part is about the net construction and the estimate 
of its cardinality. For convenience, we will define 
\begin{align} \label{eq: defTi'}
	\mathcal{T}_i' = \left\{ x \in \mathcal{T}_i\,:\, 
		x^*_{ n_{  {t_2}  +i-2} } = 1
	\right\} 
	= \left\{  
		\frac{x}{x^*_{ n_{   {t_2}  + i-2}}} \, 
		:\, x \in \mathcal{T}_i
	\right\} 
\end{align}
for $i \in \left\{ 2, \,3 \right\} $.

The second part is about probability estimate of 
$\left\| Ax \right\| > 0$ for $x$ in $\mathcal{R}$, 
$\mathcal{T}_2'$, and $\mathcal{T}_3'$. Roughly speaking, 
for $x \notin \mathcal{T}_1$, it is not too sparse so that one could 
apply Rogozin's Theorem \ref{thm:Rogozin}.

The last part will be the proof of 
Theorem \ref{thm:R-vector} and Theorem \ref{thm:T_2T_3-vector}.

\subsection{Net Construction and Cardinality Estimate }

Recall that for 
	$n_{  {t_2}  } 
	\le k 
	\le \frac{n}{ \log^2( pn )}$
, let $B = [k, n ]$. And 
\begin{align*}
	\mathcal{R}_k^1 &:= \bigg\{ 
		x \in  ( \mathcal{Y}( \tau) \cap AC( \tau, \rho) ) 
		 \backslash \mathcal{T}	
	\,:\, 
		 x^*_{ \lfloor  \tau n  \rfloor }=1
	, \, 
		 \frac{\left\| x_{\sigma_x(B)} \right\| }{
			 \left\| x_{\sigma_x(B)} \right\|_\infty 
		 } \ge \frac{ C_{\rm {Rgz}} }{\sqrt{p}} \\
& \hspace{5.5cm}		\mbox{ and }
		 \sqrt{\frac{ n }{2}} 
		 \le \left\| x_{\sigma_x(B)} \right\| 
		 \le C_{  \mathcal{T}_2 } \sqrt{  pn ^2 }
	 \bigg\} 
\\
	\mathcal{R}_k^2 &:= \bigg\{ 
		x \in \mathcal{Y}( \tau ) \backslash \mathcal{T}
		 \,:\, x^*_{ \lfloor  \tau n  \rfloor }=1, \, 
		 \frac{\left\| x_{\sigma_x(B)} \right\| }{
			 \left\| x_{\sigma_x(B)} \right\|_\infty 
		 } \ge \frac{ C_{\rm {Rgz}} }{\sqrt{p}} \\
	& \hspace{5.5cm}		\mbox{ and }
		 \frac{2\sqrt{n}}{ \tau } 
		 \le \left\| x_{\sigma_x(B)} \right\| 
		 \le C_{  \mathcal{T}_2}^2 p 
		 n^{\frac{3}{2}}
    \bigg\}  
\\   
    \mathcal{AC}( \tau , \rho) 
&:= \bigg\{ x \in \mathbb{R}^n\,:\,
        \exists \lambda\in \mathbb{R} \mbox{ s.t } 
        |\lambda| = x^*_{\lfloor  \tau n  \rfloor}\\
	& \hspace{3.5cm}	\phantom{A}	\mbox{ and }
        \left|  \left\{ i\in [n] \,:\, |x_i - \lambda |< 
            \rho \lambda \right\}  \right|  > n - 
            \lfloor  \tau n  \rfloor    
        \bigg\}.
\end{align*}

Let $m$ be the positive integer such that 
$$
	\frac{1}{\sqrt{2}} 3^{m-1} < C_{ \mathcal{T}_2} ^2 pn
	\le \frac{1}{\sqrt{2}} 3^m.  
$$
Accordingly, we define 
\begin{align*}
	\psi_1 = \frac{1}{\sqrt{2}}, \quad
	\psi_t = 3 \psi_{t-1}
\mbox{ for } t \in [2,m-1], \quad \mbox{and} \quad
	\psi_m = C_{ \mathcal{T}_2} ^2 pn.
\end{align*}

Since $C_{\cal T_2}$ is does not depend on $p$ and $n$, 
we have 
\begin{align} \label{eq: Rvector mEstiamte}
	m \le 2 \log(  pn  ).
\end{align}

For $s\in [2]$, we define 
\begin{align*}
	\mathcal{R}_{kt}^s := \left\{ x\in \mathcal{R}_k^s\,:
		\psi_t \sqrt{n} 
		\le \left\| x_{\sigma_x(B)} \right\| \le 
		\psi_{t+1} \sqrt{n}
	\right\}.
\end{align*}

In this subsection, we will prove the following lemma:

\begin{lemma} \label{lem:R-net}
	For any constant $R\ge 40$, there exists $\tau_0>0$ so that the following holds: 	
	Let $0 < \tau \le \tau_0$, $0< \rho \le \frac{1}{2R}$.
	Consider any subcollection $\cal R_{kt}^s \subseteq \cal R_k^2$ with $ s \in \{1,2\}$, $n_{t_2} < k \le \frac{n}{\log^3(pn)}$, and $t \in [m]$. 
	For any $\varepsilon$ satisfying 	
	 $$40 \psi_t \frac{\sqrt{n}}{R}\le \varepsilon 
	\le \psi_t \sqrt{n},$$ where $\psi_t$ and 
	$m$ are defined according to relation above. Then there exists an $\varepsilon$-net
	$ \mathscr{N}_{kt}^s \subset \mathcal{R}_{kt}^s$ with respect to 
	$\left\| \left| \cdot \right|  \right\| $ of cardinality at most 
	$\left( \frac{e}{\tau } \right)^{3 \tau n}$.
\end{lemma}

\begin{lemma} \label{lem:T-net}
   There exists an $ \frac{ \sqrt{ 2 n } }{
   C_{  \mathcal{T}_2} \sqrt{  pn  }}$-net 
   $ \mathscr{N}_{\mathcal{T}_i} $  
   in $|\|\cdot\||$-norm of $\mathscr{T}_i'$ 
   for $i=2,3$ such that  
   $ \left| \mathscr{N}_{\mathcal{T}_i} \right| \le 
   \exp(2 \log(  pn  ) 
   n_{   {t_2}  +i-1})  $.
\end{lemma}

We begin with some techniacal propositions that are required 
to prove the this lemma.

\begin{proposition} \label{prop:netBasic}
	Let $\left\| \cdot \right\|_K$ be a norm 
	on $\mathbb{R}^n$. (i.e. $l_2$ and $l_\infty$ norm)
	Consider the set 
	\begin{align*}
        \left\{ x\in \mathbb{R}^n \,:\, \left\| x 
            \right\|_K < a,\, 
            |\mbox{supp}(x)|\le l
        \right\} 
    	\end{align*}
	where $l<n$ is a positive integer and $a>0$.    
	For $\varepsilon >0$, 
    	there exists an $\varepsilon$-net 
        $ \mathscr{N}(l,a, \varepsilon)_K $ 
	such that 
	\begin{align*}
		\mathscr{N}(l,a, \varepsilon)_K 
	&\le 
		\left( 
			\frac{3a}{\varepsilon}
			\frac{en}{ l} 
		\right)^l & 
		\varepsilon &< a 
	\\
		\mathscr{N}(l,a, \varepsilon)_K 
	&\le 
		1
	& 
		\varepsilon & \ge  a .
	\end{align*}	
	In particular, we always have 
	\begin{align} \label{eq: netK}
		\mathscr{N}(l,a, \varepsilon)_K 
	&\le 
		\left( 
			\max\{  \frac{ a }{ \varepsilon},\, 1 \}
			\frac{ 3en}{ l} 
		\right)^l .
	\end{align}
    \end{proposition}

\begin{proof}
	Notice that for $\varepsilon > a $, we could trivially set 
	$ \mathscr{N}(l, a, \varepsilon) = \{ \vec{0} \}$.

	Now let $ 0 < \varepsilon < a$. 
    Let $I\subset [n]$ be a subset with $ |I|=l $.
    By the standard volumetric argument,
    there exists an $\varepsilon$-net $\mathscr{N}_I$
    of 
    $$
     \left\{ x \in \mathbb{R}^n\,:\, 
        \mbox{supp}(x)\subset I,\,
        \left\| x \right\|_K \le a
    \right\}  $$ 
    with size bounded by 
    $\left( 1+\frac{2a}{\varepsilon} \right)^l \le 
    \left( \frac{3a}{\varepsilon} \right)^l 
	$.
	Let 
	\begin{align*}
			\mathscr{N}(l,\, a,\, \varepsilon)_K 
	: = 
		\cup_{I \subset [n], |I|=l} \mathscr{N}_I.
	\end{align*}
    There are ${n \choose l}$ subsets of $[n]$ with size 
    $l$, which is bounded by $ \left( \frac{en}{l} 
	\right)^l  $. Thus, the cardinality of 
	$ \mathscr{N}(l, a, \varepsilon) $ 
    is bounded by 
    $\left( 
			\frac{3a}{\varepsilon}
			\frac{en}{ l} 
		\right)^l$.
\end{proof}

Next, for every $x \in \mathcal{R}_{ki}^s$, we will 
partition the support of $x$ according to $\sigma_x$ 
for $x\in \mathcal{R}_{ki}^s$.
The net we will construct for Lemma \ref{lem:R-net}
relies on this partition.

Fix an integer $k$ such that 
$ n_{   {t_2}   }
\le 
	k 
\le 
	\frac{ n }{
		\log^2(  pn  )
	}$.
For $x\in \mathbb{R}^n$, 
let 
\begin{align*}
    B_1(x)    & =   \sigma_x([n_{
          {t_2}  }]) &  & \\
	&B_{11}(x)  =   \sigma_x([n_{1}]) &  &\\
    &B_{1j}(x)  =   \sigma_x([n_{j}]
        \backslash [n_{j-1}]) &    2 &\le j \le   {t_2}   \\	
	B_2(x)    & =   \sigma_x([k]
    \backslash [n_{  {t_2}  }]) & \\
	B_3(x)    & =   \sigma_x([n_{  {t_2}  +2}]) 
    \backslash [k]) &  \\
\end{align*}
To deal with $\mathcal{R}_{kt}^1$, let 
$ B_0'(x) := \left\{ 
	i \in [ n ] 
\,:\, 
	|\lambda_x - x_i | \le 
	 \rho
\right\} $ 
where $\lambda_x$ is the $\lambda$ appeared in 
$\mathcal{AC}( \tau ,\rho)$ and define 
$ B_0(x) = B_0'\backslash (B_1(x)\cup B_2(x) \cup B_3(x))$
and $B_4(x)  = [n]\backslash B_0(x) $.

For $\mathcal{R}_{kt}^2$, we will skip $B_0$ and simply 
set $B_4(x) = [n]\backslash (B_1(x)\cup B_2(x) \cup B_3(x))$. 

For $a>0$,
\begin{align*}
	 \mathcal{T}(a) & = \left\{ P_{B_1(x)}x\,:\,
	x \notin \mathcal{T}_1 \mbox{ and } 
	x_{n_{  {t_2}  }}^*=a\right\}.
\end{align*}

\begin{proposition} \label{prop: N1net}
	For $\varepsilon >0,\, a>0$, 
	there exists an $\varepsilon$-net 
	$ \mathscr{N}_{1}(a,\varepsilon) $ 
	in $l_\infty$ norm 
	of the set $ \mathcal{T}(a)$ with cardinality bounded by 
	\begin{align*}
		& \exp \left( 2
			\log\left( 
				\frac{ a }{\varepsilon }
				(  pn  )^3
			\right)  n_{   {t_2}   }	
		\right) 	&
		\mbox{ for }  
		\varepsilon &<  C_{\mathcal{T}_1}  pn  a, 
	\\
		& \exp \left( 2
			\log\left(  
				( pn  )^3
			\right)  n_{   {t_2}   }	
		\right) 	&
		\mbox{ for }  
			 C_{\mathcal{T}_1}  pn  a \le 
			\varepsilon &< 
			( C_{\mathcal{T}_1}  pn )^{   {t_2}   } a, 
	\\
		& 1 &
		\mbox{ for } 
			\varepsilon & > 
			( C_{\mathcal{T}_1}  pn )^{   {t_2}   } a.
	\end{align*}
	In particular, we always have 
	\begin{align*}
 		\mathscr{N}_{1}(a,\varepsilon) 
	& \le 
		\exp \left( 2
			\log \left( 
				\max\{ \frac{a}{ \varepsilon }, 1 \}
				(  pn  )^3
			\right) n_{   {t_2}   }
		\right) 
	\end{align*}

\end{proposition}
\begin{proof}
	{\bf Construction of $\mathscr{N}_1$: }
	Observe that for $x\in  \mathcal{T}(a)$
	and $ j \in [0,   {t_2}   ]$, 
	\begin{align*}
     		\left\| P_{B_{1j}(x)} x \right\|_\infty 
	& \le
		(  C_{\mathcal{T}_1}  pn  )^{  {t_2}  -j+1}a, 
		& &\mbox{ and }  &
        	|\mbox{supp}(P_{B_{1j}(x)})| 
	& \le 
		n_{j} - n_{j-1},   	
	\end{align*}
	where the first inequality is due to 	
    	$	x^*_{[n_{  {t_2}  -j}]}  \le 	
		( C_{\mathcal{T}_1}  pn )^j a$.
    	For $j =0,1,\dots,  {t_2}  $, let
	\begin{align*}
		\mathscr{N}_{ 1j } = 
		\mathscr{N}(
    		n_{j}-n_{j-1}, 
		(  C_{\mathcal{T}_1}  pn  )^{  {t_2}  -j+1}a, 
		\varepsilon
		)_\infty	
	\end{align*}
	be the net described in Proposition \ref{prop:netBasic}
	corresponding to $l_\infty$ norm. 
	( For conenvience, we set $ n_{ -1 } = 0$. )
	Let 
	$$\mathscr{N}_1 = \mathscr{N}_{11}\times  \mathscr{N}_{12} \times 
    		\dots \mathscr{N}_{1  {t_2}  }.$$
 	For $x\in  \mathcal{T}(a)$, 
    there exists $y_j \in \mathscr{N}_{1j}$
    for $j\in [0,   {t_2}  ]$ such that 
    $\mbox{supp}(y_j) \subset B_{1j}(x)$ 
    and $\left\| y_j-P_{B_{1j}(x)} \right\|_\infty 
    \le \varepsilon$.

    Due to $\left\{ B_{1j}(x) \right\}_{j\in 
    [  {t_2}  ]}$ are disjoints, we have 
    \begin{align*}
     \left\| P_{B_1(x)}-\sum_{j\in 
    [  {t_2}  ]}y_j \right\|_\infty = 
    \max_{j\in [  {t_2}  ]}
    \left\|  P_{B_{1j}(x)} - y_j
    \right\|_\infty  \le \varepsilon,
    \end{align*}
    and thus $\mathscr{N}_1$ 
    is an $\varepsilon$-net in $l_\infty$ norm for $ \mathcal{T}_a$.

{\bf Upper bound for the cardinality of $\mathscr{N}_1$ }
   By \eqref{eq: netK},  
   \begin{align*}
	   | \mathscr{N}_{ 1j } | 
   \le & 
   	\bigg(
		\max \{ \frac{ (  C_{\mathcal{T}_1} pn)^{  {t_2}  - j +1 }a  }{ \varepsilon }
		,\, 1 \} 
		\frac{ 3e n  }{ n_j - n_{ j-1 } } 	
	\bigg)^{ n_j - n_{ j-1 }}\\
   \le&
   	\bigg(
		\max \{ \frac{ (  C_{\mathcal{T}_1} pn)^{  {t_2}  - j +1 }a  }{ \varepsilon }
		,\, 1 \} 
		\frac{ 6e n  }{ n_j } 	
	\bigg)^{ n_j }.
   \end{align*}
    First of all, if 
    $
    		\varepsilon 
	\ge 
		(  C_{\mathcal{T}_1}  pn  )^{   {t_2}   + 1 } a $
	then $ | \mathscr{N}_1 | = 1$ due to $ | \mathscr{N}_{ 1j } | =1$
	for $j \in [0,  {t_2}  ]$. 

    Next, let us assume  $ \varepsilon <
	(  C_{\mathcal{T}_1}  pn  )^{   {t_2}   } a $. 
	Let $  j_0 $ be 
	the largest integer in  $ [0,   {t_2}   ]$
	such that 
	$	\varepsilon 
	\ge 
		(  C_{\mathcal{T}_1}  pn  )^{   {t_2}   - j_0 + 1 } a 
    	 $. 
	Then,
	\begin{align*}
		| \mathscr{N}_{ 1j } | 
   	\le 
		\begin{cases}
		\exp\bigg(
			\underbrace{\log\big(\frac{ (  C_{\mathcal{T}_1} pn)^{  {t_2}  - j +1 }a  }{ \varepsilon }
			\frac{ 6e n  }{ n_j } 	\big)
		n_{ j } }_{  b_j }
		\bigg)&   0 \le j \le j_0, \\
		1 & j_0 < j \le   {t_2}   ,
		\end{cases}
   	\end{align*}
	and hence, 
	\begin{align*}
		| \mathscr{N}_1 | \le \exp(  \sum_{j=0}^{j_0} b_j ).
	\end{align*}

	Next, we will show $\sum_{j=0}^{j_0} b_j \le 2 b_j$ by deriving the 
	geometric progression of $ b_j$.  Specifically, we will show 
	$\frac{ b_{ j-1 } }{ b_j } \le \frac{ 1 }{ 2 }$  for $j \in [j_0]$. 

	Fix $j\in [j_0] $ if $ j_0 >0$. 
	\begin{align*}
		b_{j-1} 
	= 
		\log \left( \frac{
			6ea (  C_{\mathcal{T}_1}  pn  )^{   {t_2}   - j +1} 
			n
		}{
			\varepsilon  n_j
		}  
			\frac{ n_j }{ n_{j-1} } 
			 C_{\mathcal{T}_1}  pn 
		\right)  n_{j-1}.
	\end{align*}

	By $    \frac{(  C_{\mathcal{T}_1}  pn  )^{   {t_2}   - j +1} 
		a}{ \varepsilon } \ge 1$
	and 
	$ n_j 
	\le n_{  {t_2}   } 
	\le \frac{1}{ \gamma   p  }$, 
		
	\begin{align*}
			\frac{
			6ea (  C_{\mathcal{T}_1}  pn  )^{   {t_2}   - j +1} 
			n
		}{
			\varepsilon  n_j
		}  
	& \ge 
		 pn .
	\end{align*}
	Next,  by $ \frac{ n_j }{ n_{j-1} } 
	\le \frac{  pn }{ \log^3(  pn  ) }$, 
	\begin{align*}
		\frac{ n_j }{ n_{j-1} } 
		 C_{\mathcal{T}_1}  pn 
	\le 
		(  pn  )^3. 
	\end{align*}

	Therefore, 
	\begin{align*}
		b_{j-1} 
	& = 
		\left( 
			\log \left( 
				\frac{6ea(  C_{\mathcal{T}_1}  pn )^{  {t_2}   - j + 1} n}{
					\varepsilon n_j
				} \right)  
		+
			\log \left( 
				\frac{ n_j}{ n_{j-1}}
				 C_{\mathcal{T}_1}  pn 
			 \right) 	
		 \right)  n_{j-1} \\
	& \le 
		4\log \left( \frac{
			6ea (  C_{\mathcal{T}_1}  pn  )^{   {t_2}   - j +1} 
			n
		}{
			\varepsilon  n_j
		}  
		\right)  n_{j-1}
	 \le \frac{1}{2} b_j.
	\end{align*}
	where we used $ \frac{ n_j }{ n_{j-1} } 
	 \le \frac{  pn }{ \log^3(  pn  ) }$ .
	Thus, the cardinality of $\mathscr{N}_1$ is bounded by 
		$ \exp( 2b_{j_0} )$ .
	
	It remains to estimate the value of $b_{ j_0 }$. 
	If $ \varepsilon <  ( C_{\mathcal{T}_1}  pn ) a$, 
	then $j_0 =   {t_2}  $ and
	\begin{align*}
		| \mathscr{N}_1 | 
	& \le 
		\exp \Big( 2 \log \Big( \frac{6ea  C_{\mathcal{T}_1} (  pn  )  n 
			}{ \varepsilon n_s } \Big) n_s  \Big) 
	 \le 
		\exp \left( 
			2 \log \left( \frac{a ( pn )^3}{\varepsilon} \right) 
			n_  {t_2}   
		\right) .
	\end{align*}

	For $ \varepsilon  >  (  C_{\mathcal{T}_1}  pn  ) a$, 
	by the the definition of $j_0$, 
	\begin{align*}
		\frac{ (  C_{\mathcal{T}_1}  pn  )^{   {t_2}   - j_0 + 1} a 	
		}{ \varepsilon } 
	& \le 
		 C_{\mathcal{T}_1}  pn .
	\end{align*}
	And hence
	\begin{align*}
		| \mathscr{N}_1 | 
	& \le 
		\exp( 2 \log ( \frac{ 3e  C_{\mathcal{T}_1} (  pn  ) n }{ 
		n_{j_0 } } ) n_{j_0} ) 
	 \le 
		\exp \left( 
			6 \log \left(  pn  \right) 
			n_{  {t_2}  }
		\right) 
	\end{align*}

	where we used $ n_{j_0} 
	\le 
		n_{   {t_2}   }
	\le 	
		\frac{1}{ \gamma   p  }$ for 
	the second inequality. 
	
\end{proof}
   
\begin{proposition}
	For $ n_{s_0} 
	\le 
		k 
	\le 
		\frac{ n }{ 
			\log^2(  pn  )
	}$,
	there exists a 
	$ \rho \sqrt{ n }$-net $ \mathscr{N}_4 $  
	in $l_2$ norm for
	$ \left\{ P_{B_0(x)}x\,:\, x\in \mathcal{R}_{k}^1 \right\}  $ 
	such that 
	$ | \mathscr{N}_4 | \le  \exp(2\log(\frac{e}{ \tau })  \tau n  ) $.
\end{proposition}

\begin{proof}
	For $\mathcal{R}_{k}^1$, $|B_0(x)'| > n - 
		 \tau n  $. 
	Consider the collection \\
	$ \left\{ I\subset [ n ] 
	\,:\, 
		|I|\ge n -  \tau n  
	\right\} $. First, we give an estimate of 
	its cardinality:
	\begin{align*}
		|\left\{ I\subset [ n ] 
			\,:\, 
		|I|\ge n -  \tau n   \right\} | 
	\le &
		\sum_{j=0}^{  \tau n  }{ n \choose n - j} 	
	=  
		\sum_{j=0}^{  \tau n  }{ n \choose   j} 	\\
	\le &
		 \tau   
		( \frac{ e n }{  \tau   } )^{  \tau   }\\
	= &
		\exp \left( \log(  \tau n  ) + 
		\log(\frac{e}{ \tau })  \tau n   \right) .
	\end{align*}

	Next, let 
	$$\mathscr{N}_4 := \left\{ \pm P_I( \sqrt{ n }e) 
	\,:\,   I\subset [n] \mbox{ and } |I|\ge n -  \tau n  
	\right\}. $$ 
	Then, $|\mathscr{N}_4| \le \exp( 2\log( \frac{e}{ \tau })
		 \tau n  )$.
	Also, for $x\in \mathcal{R}_{ki}^1$, 
	the vector 
	$$y = \lambda_x P_{B_0(x)} \sqrt{ n}e 
	\in \mathscr{N}_4$$ 
	satisfies 
	$ \left\| P_{B_0(x)}x-y \right\|_2 \le \rho \sqrt{ n}  $.
\end{proof}

\begin{proof}[Proof of Lemma \ref{lem:R-net}]
	For $x \in \mathcal{R}_{kt}^s$, 
    	$ \left\| P_{B_3(k,x)}x \right\| \le 
    	   \psi_{t+1}\sqrt{n}$, we will use different nets 
    	   to approximate $P_{B_i(x)}x$ 
    	   for different $i$. Consider the following nets:
% ------------------------------------------------------------------
% Verbal description of the four nets that were formerly in the table
% ------------------------------------------------------------------
\begin{enumerate}[leftmargin=0pt,itemindent=1.5em,labelsep=0.5em]

	\item
	\textbf{Net $\mathscr{N}_1$.}  
	It is built for the parameter scale $C_{\mathcal T_2}^{\,2}pn$ with accuracy 
	$\varepsilon_1$, i.e.  
	\[
	  \mathscr N_1 := \mathscr N_1\!\bigl(C_{\mathcal T_2}^{\,2}pn,\;\varepsilon_1\bigr).
	\] 
	Its logarithmic cardinality satisfies  
	\[
	  \log\!|\mathscr N_1|
	  \;\le\;
	  2\log\!\Bigl(
			\max\!\bigl\{\tfrac{C_{\mathcal T_2}^{\,2}pn}{\varepsilon_1},\,1\bigr\}
			(pn)^3
		  \Bigr)\,n_{t_2},
	\]
	and the $\ell_2$‐covering error is at most $\sqrt{n_{t_2}}\;\varepsilon_1$.
	
	\item 
	\textbf{Net $\mathscr{N}_2$.}  
	Formed on a set of dimension $k-n_{t_2}$ at scale $C_{\mathcal T_2}^{\,2}pn$
	with accuracy $\varepsilon_2$:  
	\[
	  \mathscr N_2 := 
	  \mathscr N\!\bigl(k-n_{t_2},\,C_{\mathcal T_2}^{\,2}pn,\,\varepsilon_2\bigr)_{\infty}.
	\]
	Its size is bounded by  
	\[
	  \log\!|\mathscr N_2|
	  \;\le\;
	  \log\!\Bigl(
		\max\!\bigl\{\tfrac{C_{\mathcal T_2}^{\,2}pn}{\varepsilon_2},\,1\bigr\}
		\tfrac{3e\,n}{k}
	  \Bigr)\,k,
	\]
	and the associated $\ell_2$‐distance is $\sqrt{k}\;\varepsilon_2$.
	
	\item
	\textbf{Net $\mathscr{N}_3$.}  
	Constructed on $\lfloor\tau n\rfloor-k$ dimensions with scale
	$\psi_{t+1}\sqrt n$ and accuracy $\varepsilon_3$:  
	\[
	  \mathscr N_3 := 
	  \mathscr N\!\bigl(\lfloor\tau n\rfloor-k,\,
						\psi_{t+1}\sqrt n,\,
						\varepsilon_3\bigr)_{2}.
	\]
	The size estimate is  
	\[
	  \log\!|\mathscr N_3|
	  \;\le\;
	  \log\!\Bigl(
		\tfrac{\psi_{t+1}\sqrt n}{\tau\varepsilon_3}
	  \Bigr)\,\tau n,
	\]
	while the $\ell_2$‐covering error is simply $\varepsilon_3$.
	
	\item%
	\textbf{Net $\mathscr{N}_0$.}  
	This is the set of signed coordinate projections of the vector
	$\sqrt n\,e$ on “large” coordinate subsets:  
	\[
	  \mathscr N_0 :=
	  \Bigl\{\,\pm P_I(\sqrt n\,e)\;:\;I\subset[n],\;
				 |I|\ge n-\lfloor\tau n\rfloor\Bigr\}.
	\]
	Its logarithmic size obeys  
	\[
	  \log\!|\mathscr N_0|
	  \;\le\;
	  2\log\!\bigl(\tfrac{e}{\tau}\bigr)\,\tau n,
	\]
	and every relevant vector lies within $\rho\sqrt n$ in $\ell_2$ of some
	net point.
	
	\end{enumerate}

    	Let $\mathscr{N} = \mathscr{N}_1\times \mathscr{N}_2 
		\times \mathscr{N}_3 \times \mathscr{N}_4$. 
    	For $x \in \mathcal{R}_{kt}^1$, there exists 
	$y_i \in \mathscr{N}_i$ for $i=0,1,2,3$ such that 

    	\begin{align} \label{eq:netDistance}
    	    \left\| x- \sum_{i=0}^3 y_i \right\|^2  \le 
    	    \sum_{i=0 }^3\left\| P_{B_i(x)}x-y_i \right\|^2 + 
    	    \left\| P_{B_4(x)}x\right\|^2 
    	    \le  \rho^2 n 
	    	+ n_{   {t_2}   }\varepsilon_1^2 
		+ k\varepsilon_2^2 
    	    	+ \varepsilon_3^2 +  \tau n .
    	\end{align}

    	Let $\varepsilon_3 =  \frac{1}{3}\varepsilon$, 
	$\varepsilon_1 = \frac{1}{ \sqrt{ n_  {t_2}   } } \varepsilon_3$,
	and 
    	$\varepsilon_2 = \frac{1}{\sqrt{k}}\varepsilon_3$. 
    	Then, 
    	$$
    	    \eqref{eq:netDistance} 
	\le 
		( \rho^2 + \tau) n 
	+ 
		\frac{1}{3} \varepsilon^2 
    	\le 
		\frac{2}{3}\varepsilon^2  
    	$$ when $ \tau > 0 $ is sufficiently small 
	due to the assumption that  $40 \psi_t \frac{\sqrt{n}}{R}$
	and $\phi_t \ge \phi_1 \ge \frac{ 1 }{ \sqrt{2} }$. Hence, 
	$\mathscr{N}$ is a $\sqrt{\frac{2}{3}}\varepsilon$-net for 
    	the corresponding $\mathcal{R}^1_{kt}$ 
    	in $l_2$-norm.

	Because of the same lower bound for $\varepsilon$, we have
	$\varepsilon_1, \varepsilon_2 \ge \frac{1}{R}$.
	Together with 
	$ n_  {t_2}   
	\le	k 
	\le \frac{ n }{ \log^2(  pn  )}$, 

	\begin{align*}
		\log(| \mathscr{N}_1 |) 
	& \le 	
		8 \log(  pn  ) n_s ,&
		\log( | \mathscr{N}_2| )
	& \le 
		3 \log(  pn  ) k , \mbox{ and  } 
	\\	
		\log( | \mathscr{N}_3| )
	& \le 
		 \log( \frac{9}{40} \frac{R}{\tau})  \tau n  
		 \le 2 \log( \frac{1}{\tau } )  \tau n  .
	\end{align*}
	Therefore, we have	
	\begin{align} \label{eq: netNsize}
		\log(|\mathscr{N}|) \le \left( \frac{1}{r} 
	\right)^{\frac{5}{2}rn} .
	\end{align} 

    	For $\mathcal{R}_{kt}^2$, we set 
	$\mathscr{N}= \mathscr{N}_1 \times \mathscr{N}_2 \times \mathscr{N}_3$
    	with the same choice of $\varepsilon_1, 
    	\varepsilon_2,$ and $\varepsilon_3$. Similarly, 
    	for $x\in \mathcal{R}_{kt}^2$ there exists 
    	$y\in \mathscr{N}$ satisfying 
    	\begin{align*}
    	    \left\| x - y \right\| \le 
    	     \frac{1}{3} \varepsilon^2  + n
    	    \le \frac{2}{3} \varepsilon^2 
    	\end{align*}
    	where the last inequality holds when 
    	$\psi_t>\frac{2}{ \tau }$ and $r$ is 
    	sufficiently small. And this is sufficient 
    	since for $\psi_t < \frac{2}{ \tau }$, 
    	$\mathcal{R}_{kt}^2$ is an empty set by definition.  

    	Therefore, we have constructed a 
    	$ \sqrt{\frac{2}{3}}\varepsilon$-net for 
    	the corresponding $\mathcal{R}_{kt}^s$ 
    	in the $l_2$-norm. 

	It remains to modity the net $ \mathscr{N}$ to be 
	the $\varepsilon$-net stated in the lemma corresponding to 
	$|\|\cdot\||$-norm. 

	By definition of $\mathcal{R}$, the following holds: $ \mathcal{R} \in \mathcal{Y}( \tau ) \backslash \mathcal{T}$ 
	Then, by Proposition \ref{prop: TnormBound}, 
	$$\forall x \in \mathcal{R}, x_1^* \le n^{1 + o_{n}(1) } pn $$ and thus, 
    	$$|\langle x, {\bf e} \rangle| 
	\le \frac{1}{ \sqrt{ n } }
		\sum_{i\in [ n ]} x_i^* \le n^{3}$$
	due to $ p < 1$. 
    	Let $ \mathscr{N}_{{\bf e}} $ be an 
    	$0.1 \varepsilon \frac{1}{  pn   }$ 
	net of the set $ \left\{ t{\bf e}\,:\, |t|\le n^{3} \right\}$. 
	With $ \varepsilon > 40 \phi_t \frac{ \sqrt{n} }{ R }  \ge  n^{1/4} $, 
	$|\mathscr{N}_{{\bf e}}| \le 
    	n^4$. 

    	Let $\mathscr{N}_{kt} = \left\{ P_{\bf e}y+v\,:\,
    	y \in \mathscr{N},\, v\in \mathscr{N}_{\bf e}
	\right\}$. Then, together with \eqref{eq: netNsize}, 
	$$\left| \mathscr{N}_{kt} \right|\le | \mathscr{N} | \cdot | \mathscr{N}_e | \le 
    	(\frac{e}{ \tau })^{3  \tau n  }. $$ 

    	For $x\in \mathcal{R}_{kt}^s$, 
    	let $y\in \mathscr{N}$ and 
    	$v \in \mathscr{N}_{\bf e}$ such that 
    	$\left\| x-y \right\| \le \sqrt{\frac{2}{3}}\varepsilon$ 
    	and $\left\| P_{\bf e}^\perp y - v \right\| 
	\le \frac{0.1 \varepsilon}{ pn}$. 

    	Then, 
    	\begin{align*}
    	    |\| x- P_{\bf e}y + v \|| 
    	    = \|P_{\bf e}x-P_{\bf e}y\| + 
    	    pn \| P_{\bf e}^\perp x - v\| 
    	    \le \|x-y\| + pn  \| P_{\bf e}^\perp x - v\|
    	    \le \varepsilon.
    	\end{align*}
\end{proof}

\begin{proof}[Proof of Lemma \ref{lem:T-net}]
    The proof is almost identical to that of Lemma 
    \ref{lem:R-net}. So we just show the case for 
    $\mathcal{T}_3'$. 

    Let $k= n_{   {t_2}   +1}$ and we keep the same 
    definition of  $B_1(x), B_2(x)$, and $B_3(x)$.

    For $x\in \mathcal{T}_3'$, we have 
    $x^*_{ n_  {t_2}   } \le C_{  \mathcal{T}_2}\sqrt{  pn }$, 
    $x^*_{ n_{   {t_2}   +1 } } \le 1$, and 
    $x^*_{  n_{   {t_2}  +2 } } 
    \le \frac{1}{ C_{  \mathcal{T}_2} \sqrt{  pn  } } $.
    Thus, we use the following nets:
	% ------------------------------------------------------------------
% Verbal description of the three nets listed in Table \ref{tab:net}

% ------------------------------------------------------------------
\begin{enumerate}[leftmargin=0pt,itemindent=1.5em,labelsep=0.5em]

	\item 
	\textbf{Net $\mathscr{N}_{1}$.}  
	\[
	  \mathscr N_1 :=
	  \mathscr N_1\!\Bigl(
		 C_{\mathcal T_2}\sqrt{pn},
		 \tfrac{1}{2C_{\mathcal T_2}\sqrt{pn}}
	  \Bigr).
	\]
	Its size satisfies the bound
	\[
	  \log|\mathscr N_1| \;\le\; 10\,
	  \log(pn)\; n_{t_2},
	\]
	and for every vector $x\in\mathcal T'_3$ there exists a point
	$y\in\mathscr N_1$ such that
	\[
	  \bigl\|p_{B_1(x)}x - y\bigr\|_2
	  \;\le\;
	  \frac{\sqrt{n_{t_2}}}{2C_{\mathcal T_2}\sqrt{pn}}.
	\]
	
	\item
	\textbf{Net $\mathscr{N}_{2}$.}  
	\[
	  \mathscr N_2 :=
	  \mathscr N\!\Bigl(
		k-n_{t_2},\,
		C_{\mathcal T_2}\sqrt{pn},\,
		\tfrac{1}{2C_{\mathcal T_2}\sqrt{pn}}
	  \Bigr)_{\infty}.
	\]
	It obeys
	\[
	  \log|\mathscr N_2|
	  \;\le\;
	  3\,\log(pn)\;k,
	\]
	and covers $\mathcal T'_3$ in the sense that for every
	$x\in\mathcal T'_3$ there is a $y\in\mathscr N_2$ with
	\[
	  \bigl\|p_{B_2(x)}x - y\bigr\|_2
	  \;\le\;
	  \frac{\sqrt{k}}{2C_{\mathcal T_2}\sqrt{pn}}.
	\]
	
	\item
	\textbf{Net $\mathscr{N}_{3}$.}  
	\[
	  \mathscr N_3 :=
	  \mathscr N\!\Bigl(
		\lfloor\tau n\rfloor - k,\,
		1,\,
		\tfrac{1}{C_{\mathcal T_2}\sqrt{pn}}
	  \Bigr)_{\infty}.
	\]
	Its logarithmic size is bounded by
	\[
	  \log|\mathscr N_3|
	  \;\le\;
	  \log\!\Bigl(\tfrac{3e\,C_{\mathcal T_2}\sqrt{pn}}{\tau}\Bigr)\,
	  \tau n,
	\]
	and for every $x\in\mathcal T'_3$ we can find
	$y\in\mathscr N_3$ with
	\[
	  \bigl\|p_{B_3(x)}x - y\bigr\|_2
	  \;\le\;
	  \frac{\sqrt{n_{s_0+2}}}{C_{\mathcal T_2}\sqrt{pn}}.
	\]
	
	\end{enumerate}
    Then, $\mathscr{N} = \mathscr{N}_1 \times \mathscr{N}_2
    	\times \mathscr{N}_3 $
    will be the net which approximates 
    $\mathscr{T}_3'$ well in the $l_2$ 
    norm. To pass it to the triple norm 
    $|\| \cdot \||$. Let 
    \begin{align*}
        \mathscr{N}_{\mathcal{T}_3'} := \left\{
            P_{\bf e}y +v \,:\, 
            y\in \mathscr{N} \mbox{ and }
            v \in \mathscr{N}_{\bf e}
         \right\} 
    \end{align*}
    where $\mathscr{N}_{\bf e}$ is the 
    net appeared in the proof of Lemma 
    \ref{lem:R-net}.
\end{proof}

\vspace{1cm}

\subsection{Tail bound for individual probability}

In this section, we will give a probability estimate of 
$\left\|  A  x \right\| > 0 $
for $x\in \mathcal{R}_{k}^s$, $x\in \mathcal{T}_2'$, and 
$x\in \mathcal{T}_3'$, respectively. 
\begin{proposition} \label{prop:idProbRvector}
	There exists a universal constant $ c_{ \ref{prop:idProbRvector} }>0$ so that 
	the following holds: For every $x\in \mathcal{R}_{kt}^s$, 
%	let $I= \sigma_{x}([k,n])$. Then, 
	\begin{align*}
	&\mathbb{P}  \left\{ \exists I \subseteq [n] \mbox{ w. } |I| = n-\beta+1 \mbox{ s.t. }
			\left\|   A_{I,[n]}  x \right\|
		\le 
			\frac{1}{4 } C_{\rm {Rgz}} 
			\sqrt{  pn  }
			\psi_t \sqrt{ n}
		\right\}\\
	\le &
		\exp \left( - c_{ \ref{prop:idProbRvector} } n  \right) .
	\end{align*}
\end{proposition}
\begin{proof}
	Let $J= \sigma_{x}([k,n])$. By definition of $\mathcal{R}_{kt}^s$, 
	$ \frac{\left\| x_{J} \right\|}{\left\| x_{J} \right\| _\infty } 
		\ge \frac{2C_{\rm {Rgz}}}{\sqrt{  p  }} $.

	Applying \eqref{eq:RogozinIprod} we get
	\begin{align*}
		\mathcal{L}\left(
			(  A  x)_i, 
			\frac{\left\| x_J \right\| \sqrt{  p  } 
			}{ 2C_{\rm {Rgz}} } 
		\right) 
	\le 
		\frac{1}{2} 
	\Rightarrow 
		\mathbb{P} \bigg\{ |(  A  x)_i| > \frac{\left\| x_J \right\| \sqrt{  p  } 
		}{ 2C_{\rm {Rgz}} } \bigg\} \ge \frac{ 1 }{ 2 }. 
	\end{align*}
	Due to the events $$ \bigg\{
		|( A  x )_i| 
	> 
		\frac{\left\| x_J \right\| \sqrt{ p }}{2C_{\rm {Rgz}}} 
	\bigg\}$$ are jointly independent 
	for $i \in [ n ]$. 
	By the standard Hoeffding's inequality, there exists an universal constant $c_{ \ref{prop:idProbRvector} }>0$ 
	such that 
	\begin{align*}
		\mathbb{P} \left(
			\left| \left\{ 
				i\in [n]
			\,:\, 
				|(  A  x)_i |
			> 
				\frac{\left\| x_J \right\| \sqrt{ p }}{2C_{\rm {Rgz}}}
			\right\}  \right|
		\le 
			\frac{ n }{ 4 }
		\right)
	\le 
		\exp(- c_{ \ref{prop:idProbRvector} } n ) .
	\end{align*}
	If the compliement of the event in the above inequality holds, then 
	for each $I \subseteq [n]$ with $|I|=n-\beta +1$, 
	\begin{align*}
		\left\|  A_{I,[n]}  x \right\| 
	\ge 
		\sum_{ 
			i \in I 
			\mbox{ s.t. }
				|(  A  x)_i |
			> 
				\frac{\left\| x_J \right\| \sqrt{p}}{2C_{\rm {Rgz}}}
		} 
		(  A  x )_i^2 
	\ge &
		\Big(\frac{ n }{ 4 } - \beta \Big)
		\frac{  p }{ 2 C_{\rm {Rgz}}} \left\| x_J \right\|^2 \\
	& \phantom{AAA A}\ge 
		\frac{np}{16 C_{\rm Rgz}}\left\| x_J \right\|^2.
	\end{align*}
	Therefore, 
	\begin{align*}
		\mathbb{P}  \left\{
			\left\|   A  x \right\|
		\le 
			\frac{1}{4 } C_{\rm {Rgz}} 
			\sqrt{  pn  }
			\left\| x_J \right\| 
		\right\}
	\le
		\exp \left( - c_{ \ref{prop:idProbRvector} }n \right). 
	\end{align*}
	Finally, the statement of the Proposition follows due to 
	$ \left\| x_J \right\| \ge \psi_t \sqrt{ n }$ 
	by the definition of $\mathcal{R}_{kt}^s$. 
\end{proof}

\begin{lemma} \label{lem: indivProCore}
	Let $A$ be a $n\times n$ random Bernoulli matrices 
	with parameter $p$. Assuming that $n$ is sufficiently 
	large. 
	For a vector $x\in \mathbb{R}^n$ satisfying 
	\begin{align*}
		x^*_{m_1}=3a \, \mbox{ and } x^*_{m_1} > 3x^*_{n-m_1}. 
	\end{align*}
	where $m_0 \le m_1 \le \frac{n}{2}$ and $a>0$. 
	We have 
	\begin{align*}
		\mathbb{P} \left\{ 
		 \left| \left\{ i\in [n] \,:\, |(Ax)_i|>a \right\}  \right|  
		< \frac{1}{50}qn 
		\right\} \le \exp(-\frac{1}{40}qn)
	\end{align*}
	where 
	$$q = 2m_0p(1-p)^{2m_0-1}.$$
	Further, if $\lfloor \frac{m_1 }{m_0 } \rfloor q>C$ for some universal constant $C>0$, 
	then 
	\begin{align*}
		\mathbb{P} \Big\{
		|\left\{ i\in [n]\,:\, |(Ax)_i|>a \right\}| < n/4
	 \Big\} 
		\le 2 \exp \Big(-\frac{1}{12}
		\log \Big( \lfloor \frac{m_1}{m_0} \rfloor q \Big)n \Big).
	\end{align*}
\end{lemma}

Remark: We will apply this lemma with 
$m_0 = n_{   {t_2}   }
	= \frac{ 1 }{ \gamma   p }$ 
with $m_1= m_0$ or $m_1 = n_{   {t_2}   + 1}$. 
In either cases,
$ q \ge  m_1  p $.  
As a Corollary, we have 
\begin{corollary} \label{cor: indivProT2T3}

	There exists an universal constant $ c_{ \ref{cor: indivProT2T3}} >0$ 
	so that the following 
	estimates holds: 
	For any fixed constant $\beta$, if $n$ is sufficiently large, then for every $x \in {\mathcal{T}_2'}$,  
	\begin{align*}
		\mathbb{P}\left\{  \exists I \subseteq [n] \mbox{ with } |I|= n-\beta+1 \mbox{ s.t. } \|  A_{I,[n]} x\| 
		< 
	\frac{  \sqrt{ n } }{ \log( n )  }\right\} 	
		&\le 
		\exp(- c_{ \ref{cor: indivProT2T3} } 
			n ) , \\
	\end{align*}
	and for every $x \in {\mathcal{T}_3'}$,
	\begin{align*}
		&\mathbb{P}\left\{ 
		\exists I \subseteq [n] \mbox{ with } |I|= n-\beta+1 \mbox{ s.t. } \|  A_{I,[n]} x\| 
		< 
		\frac{  \sqrt{ n } }{ \log( n )  }	
		\right\} \\
		 \le& 
		\exp \left(
			- c_{ \ref{cor: indivProT2T3} }
			\log(  pn  )
			n
		\right).
	\end{align*}
\end{corollary}
\begin{proof}
	
	For $ x \in \mathcal{T}_2'$, from the definition of $\mathcal{T}_2'$ we know that 
	\begin{align*}
		1 
	= 
		x^*_{  n_{  {t_2}   } }	
	\ge 
		\frac{ 1 }{ C_{ \mathcal{T}_2 } \sqrt{pn}  } 
		x^*_{  n_{   {t_2}  +1 } }
	\ge
		3 x^*_{ n - n_{  {t_2}   } }.
	\end{align*}

	Applying Lemma \ref{lem: indivProCore} with the matrix $  A $, 
	$ m_0 = m_1 = n_  {t_2}  $,
	and $a = 1/3$ we obtain
	\begin{align*}
		\mathbb{P} \left\{ 
		 \left| \left\{ i\in [n]\,:\, |(Ax)_i|>1/3 \right\}  \right|  
		< \frac{1}{50}qn 
		\right\} \le \exp(-\frac{1}{40}qn)
	\end{align*}
	where $ q = 2 n_s p ( 1- p  )^{ 2 n_s -1}$. 
	Since $ n_  {t_2}   
	=  \lceil \frac{1}{ \gamma   p } \rceil$, with the estimate 
	$ ( 1-x ) = \exp( -x + O(x^2))$ we have 
	\begin{align} \label{eq: T23qestimate}
		q \ge \frac{ 2 }{ \gamma } \exp \Big( - \frac{ 3 }{ \gamma }  \Big)
	\end{align}
	due to $ \gamma \ge 1$. Hence, there exists 
	$ c_{ \ref{cor: indivProT2T3} } > 0$ such that 
	\begin{align*}
		\mathbb{P} \big\{ 
		\left| \left\{ i\in I(J)\,:\, |(Ax)_i|>1/3 \right\}  \right|
		<
		c_{ \ref{cor: indivProT2T3} }n 
		\big\} 	
	\le
		\exp (
			-  c_{ \ref{cor: indivProT2T3} }
			n
		).
	\end{align*}
	Hence, with probability at least $1 - \exp( - c_{ \ref{cor: indivProT2T3} }n)$, for each $I \subseteq [n]$ with $|I|=n-\beta +1$, 
	we have 
	\begin{align*}	
		&\left| \left\{ i\in I(J) \cap I \,:\, |(Ax)_i|>1/3 \right\}  \right|
		\ge  c_{ \ref{cor: indivProT2T3} } n - \beta  = o(n/\log^2(n)) \\
	\Rightarrow&  
		\|  A_{I,[n]} x\| 
		\ge  
			\frac{1}{3} \sqrt{c_{ \ref{cor: indivProT2T3} } n - \beta}
		\ge 
		\frac{  \sqrt{ n } }{ \log( n )  }	
	\end{align*}

	For $ x \in \mathcal{T}_3'$, we have 
	\begin{align*}
		1 
	= 
		x^*_{  n_{  {t_2}  +1 } }
	\ge 
		C_{ \mathcal{T}_2 } \sqrt{ pn } x^*_{  n_{  {t_2}  +2 }  } 
	\ge 
		C_{ \mathcal{T}_2 } \sqrt{ pn } x^*_{  n - n_{  {t_2}  +1  } }.
	\end{align*}

	Here we apply Lemma \ref{lem: indivProCore} with the same parameters except 
	$m_1 = n_{   {t_2}  +1 }$ this time. 
	In this case, 
	\begin{align*}
		\frac{ m_1 }{ m_0 } q
	=	
		\frac{ n_{  {t_2}  +1 }  }{ n_{  {t_2}   } } q
	\ge
	 	\gamma \sqrt{ pn }  q
	\ge
	(   pn )^{1/4}. 
	\end{align*}
	Thus, by Lemma \ref{lem: indivProCore} 
	\begin{align*}
		\mathbb{P} \Big\{
		|\left\{ i\in [n]\,:\, |(Ax)_i|> 1/3 \right\}| < n/4
	 \Big\}  
	\le &
		2 \exp \left(
			- \frac{1}{12} \cdot \frac{ 1 }{ 4 }
			\log \left(  
				pn 
			\right) n
		\right) \\
	\le &
		\exp \left(
			- c_{ \ref{cor: indivProT2T3} }
			\log(  pn  )
			n
		\right)\,,
	\end{align*}
	where the last inequality follows provided $ c_{ \ref{cor: indivProT2T3} } >0$ is sufficiently small. Given the complement of the above event, for each $I \subseteq [n]$ with $|I|=n-\beta +1$, 
	we have 
	\begin{align*}	
		&\left| \left\{ i\in I(J) \cap I \,:\, |(Ax)_i|>1/3 \right\}  \right|
		\ge  n/4 - \beta  = o(n/\log^2(n)) \\
	\Rightarrow&  
		\|  A_{I,[n]} x\| 
		\ge  
			\frac{1}{3} \sqrt{n/4 - \beta}
		\ge 
		\frac{  \sqrt{ n } }{ \log( n )  }.
	\end{align*}
	Therefore, the corollary follows. 

\end{proof}

Here we will set up some notations for the proof of Lemma \ref{lem: indivProCore}. 
Let $ h = \lfloor \frac{m_1}{m_0} \rfloor $. 
We define the corresponding sets:

\begin{align*}
J^l_t &=& \sigma_x \big(
	[(t-1)m_0+1, tm_0] \big) & & t\in [ h ] \\
J^r_t &=& \sigma_x \big(
	[n-tm_0+1, n-(t-1)m_0] \big) & & t\in [ h ] \\ 
J_t & =& J^l_t \bigcup J^r_t & &  \\
J_0 & =&  [n] \backslash \left( \cup_{t\in [h]}J_t \right)  \\
\mathcal{J} & = & (J_1,\dots, J_h). & &
\end{align*}

For $J \subset [ n ]$, we define 
\begin{align*}
	I(J) =  \big\{ i\in [n]\,:\, \exists j_0 \mbox{ s.t. }
		 a_{ij_0}=1 \mbox{ and } a_{ij}=0 \mbox{ for }
		 j\in J \backslash \left\{ j_0 \right\} \big\}.
\end{align*}

Next, let $ \mathcal{I}$ be a collection of subsets of  $ [ n ]$ : 
$\mathcal{I}=(I_1,\dots, I_s)$ where $I_i\subset [n]$. 
Let $M$ be a $[n] \times J_0$ matrix with $0,1$ entries.
Specifically, $M$ is a $n \times |J_0|$ matrix whose 
columns are indexed by $J_0$ and rows are indexed by $[ n ]$.

Let $ \Omega_{ \mathcal{J},\mathcal{I}, M } $ be the event 
of $A$ that 
1. $I(J_t) = I_t$ for $t\in [h]$, and 
2. $A_{J_0}=M$, where $A_{J_0}$ is the submatrix of $A$ 
with columns $J_0$. 

For $i \in [ n ]$ and $J \subset [ n ]$, let 
$A_{i J}$ denotes the $1 \times |J|$ submatrix 
of $A$ with row $i$ and columns $J$.
If we condition on $\Omega_{ \mathcal{J},\mathcal{I}, M } $, 
then 
$ \left\{  A_{iJ_t} \right\}_{i\in [n], t\in [h]}  $
are jointly independent but not necessary i.i.d.

Fix $i \in [n]$, we have 
\begin{align*}
	(Ax)_i = \sum_{t =0}^h \sum_{j\in J_t} a_{ij}x_j
	:= \sum_{t =0}^h \xi_t.
\end{align*}

\begin{proposition}
	Condition on $ \Omega_{\mathcal{J},\mathcal{I}, M} $ and fix 
	$i\in [n]$ and $t\in [h]$ such that $i\in I(J_t)$. 
	Let $\xi_t := \sum_{j\in J_t} a_{ij}x_j$. Then, 
	$ \mathcal{L}(\xi_t , 2a) \le \frac{1}{2}$. 
\end{proposition}

\begin{proof}
	Fix $i\in I(J_t)$. $A_{iJ_t}$ contains exactly one 
	none-zero entry: There exists $j_0$ which is uniformly 
	chosen among $J_t$ such that $a_{ij_0}=1$ and $a_{ij}=0$
	for $j \in J\backslash \left\{ j_0 \right\} $. 
	Notice that $\xi_t = x_{j_0}$. 
	
	First, $\mathbb{P} \left\{ j_0 \in J^l_t \right\}  = 
	\mathbb{P}\left\{ j_0 \in J^r_l \right\} = \frac{1}{2}$. 
	Secondly, for $j_l \in J^l_t, j_r \in J^r_t$, we have 
	$|x_{j_l}| >3 |x_{j_r}| \ge a$ by the definition of $J_t$.  
	Together we conclude that $ \mathcal{L}(x_{j_0},\, a) \le 
	\frac{1}{2} $ .
\end{proof}

\begin{proof}[ Proof of Lemma \ref{lem: indivProCore} ]
	Consider the set 
	\begin{align*}
		S_i:=\{t\in [h]\,:\, i\in I(J_t)\}.
	\end{align*}
	If we condition on $\Omega_{\mathcal{J},\mathcal{I}, M }$,
	the set $S_i$ is fixed and determined by $ \mathcal{I}$. 
	In the case that $S_i$ is non-empty,  let $t\in S_i$ and we have
	$ 
		\mathcal{L}((Ax)_i,\, a) \le 
		\mathcal{L}(\xi_t,\, a) \le \frac{1}{2}.
	$
	
	\medskip 	

	We begin with proving the first statement. 
	Let $J:=J_1$ and $I:=I_1$. The expected size of $I(J)$ is $nq$. 
	Let $O_1$ be the event that $ |I(J)| \ge\frac{1}{5}qn $, by 
	\eqref{eq:BnmLow} we have 
	\begin{align} \label{eq: IndiProO_1}
	 \mathbb{P}\left\{ O_1^c \right\} \le \exp \Big(-\frac{1}{3}qn \Big) .  
	\end{align}

	Notice we could partition the event $O_1$ into 
	subevents of the form $ \Omega _ {\mathcal{J},\mathcal{I}, M}$.
	Now we condition on an subevent 
	$ \Omega _ {\mathcal{J},\mathcal{I}, M} \subset O_1$. In this 
	case, $I(J) = I $ is a fixed set. 
	
	Since $\left\{ (Ax)_i \right\}_{i\in I(J)}$ are jointly independent, we have 
	\begin{align*}	
		\mathbb{P} \left\{  \left| \left\{ i\in I(J)\,:\, |(Ax)_i|>a \right\}  
		\right|  \le \frac{1}{5}\cdot \frac{1}{2}|I(J)|\, \Biggm\vert 
		\, \Omega _ {\mathcal{J},\mathcal{I}, M}
		\right\}   
	\le &
		\exp( - \frac{1}{3}\cdot \frac{1}{2} |I(J)|) \\
	\le &
		\exp( - \frac{1}{30}qn)\,,
	\end{align*}	
	by \eqref{eq:BnmLow}. Hence, 
	
	\begin{align} \label{eq: IndiProO_1.2}
		\mathbb{P} \left\{  
			\left| \left\{ i \in I(J) \,:\, 
				| (Ax)_i | > a
			\right\}  \right|  
		\le 
			\frac{1}{50} qn
		\, \Biggm\vert \, 
			O_1
		\right\} 
	\le
		\exp( - \frac{1}{30} qn ).
	\end{align}

	Let $O_2$ be the event that 
	$ \left| \left\{ i\in I(J)\,:\, |(Ax)_i|>a \right\}  \right|  
		\ge \frac{1}{50}qn $. 
	By \eqref{eq: IndiProO_1} and \eqref{eq: IndiProO_1.2}, 
	\begin{align*}
		\mathbb{P}\left\{ O_2^c \right\} \le \mathbb{P} \left\{ O_1^c \right\}
		+ \mathbb{P} \left\{ O_2^c \, |\, O_1 \right\} \le
		\exp(- \frac{1}{40}qn).
	\end{align*}
	
	% Furthemore, within the event $O_2$ , 
	% we have $ \left\| Ax \right\|^2 \ge \frac{1}{50}qna^2 $.

	\medskip

	As for the second statement, let us assume the following:
	$$hq > 6 \log(2e),$$
	or $6 \log(2e)$ is the constant $C$ stated in the Lemma. 

	The expected size of $|S_i|$ is $ hq $. Recall that $A_{i,J_t}$ are jointly independent for 
	$t\in [h]$. Thus, applying \eqref{eq:BnmLow}
	we get 
	\begin{align*}
		\mathbb{P} \left\{ 
			|S_i| < \frac{1}{5} hq
		 \right\}  \le 
		 \exp(-\frac{1}{3} hq ) :=q_2.
	\end{align*} 
	
	Let $I:= \left\{ i\in [n]\,:\, |S_i|> \frac{1}{5} hq \right\} $
	and $O_1$ be the events that $|I| \ge \frac{n}{2}$.
	By \eqref{eq:BnmUp}, 
	\begin{align*}
		\mathbb{P} \left\{ O_1^c \right\}  
	= 
		\mathbb{P} \left\{ 
			|I^c| \ge \frac{n}{2}
		\right\} 
	\le 
		\left( \frac{ en q_2}{n/2} \right)^{ n/2 }
	\le 
		\exp( -\frac{1}{ 12 } \log(hq) n ).
	\end{align*} 
	provided that $hq > 6 \log(2e)$.
	
	We could also partition $O_1$ into subevents of the form 
	$ \Omega_ { \mathcal{J}, \mathcal{I}, M} $. Now we fix 
	a subevent $ \Omega_ { \mathcal{J}, \mathcal{I}, M} $
	of $O_1$.
	
	By Rogozin's Theorem (Theorem \ref{thm:Rogozin}), we have
	\begin{align*}
	\mathcal{L} ( (Ax)_i ,\, a) \le 
		\mathcal{L}( \sum_{t\in S_i}\xi_t, \, a) \le 
		\frac{C}{\sqrt{\frac{1}{2}|S_i| }} \le 
		\frac{C\sqrt{10}}{\sqrt{hq}}
	\end{align*}
	for $i\in I$. Let 
	\begin{align*}
		I' := \left\{ i\in I_1\,:\, |(Ax)_i|>a \right\}.
	\end{align*}
	Due to independence of the rows (after conditioning on
	 $ \Omega_{\mathcal{J}, \mathcal{I}, M}$), 
	 by \eqref{eq:BnmUp} we have
	\begin{align*}
		\mathbb{P} \left\{ |I'|^c > \frac{|I|}{2}  \, \Biggm\vert\, 
		\Omega_ {\mathcal{J}, \mathcal{I}, M} \right\} 
		\le \exp \left( 
			\log( \frac{ 2eC\sqrt{10}}{ \sqrt{hq}}) \frac{|I|}{2}
		 \right) 
		\le \exp(- \log(hq)\frac{|I|}{5}),
	\end{align*}
	and thus 
	\begin{align*}
		\mathbb{P} \left\{  
		 | \left\{ i\in I_1\,:\, |(Ax)_i|>a \right\} |
	\le 
		\frac{n}{4}
	 \, \Biggm\vert\, 
	 	O_1
		\right\} 
	\le 
		\exp( - \log( hq ) \frac{n}{10} )
	\end{align*}
	
	Thus, let $O_2$ be the event that $|I'|>\frac{n}{4}$, we have 
	\begin{align*}
	\mathbb{P} \left\{ O_2^c \right\} \le \mathbb{P}\left\{ O_1^c \right\} 
	+ \mathbb{P}\left\{ O_2^c\,|\, O_1 \right\} 
	 \le 2 \exp(-\frac{1}{12}\log(hq)n)
	\end{align*}
	
\end{proof}

\subsection{ Estimate for $ \mathcal{T}_2', \mathcal{T}_3'$ 
and $\mathcal{R}_{kt}^s$}

\begin{proof}[ Proof of Theorem \ref{thm:R-vector} ]	
	Recall that 
	\begin{align*}
		\mathcal{R} 
	:= 
		\cup_{ s \in [2] } 
		\cup_{ n_  {t_2}   
			\le k
			\le \frac{ n }{ \log^2(  pn  )} }
		\cup_{ t \in [ \mathfrak{m} ] }
		\mathcal{R}^s_{kt}.
	\end{align*}
	
	Now we focus on $ \mathcal{R}^s_{kt} $
	for a triple $(s, k, t)$. Let $\mathscr{N}:=\mathscr{N}_{kt}^s$ be the net 
	descibed in Lemma \ref{lem:R-net}. The cardinality 
	of the net is bounded by 
	$ \left( \frac{e}{ \tau } \right)^{3  \tau n  }$. 
	When $ \tau$ is sufficiently small, we may assume 
	$ \mathscr{N}  \le \exp(  
		\frac{ c_{ \ref{prop:idProbRvector} } }{ 2 }n)$
	where $c_{ \ref{prop:idProbRvector} } >0$ is a universal constant 
	appeared in Proposition \ref{prop:idProbRvector}. 

	By Proposition \ref{prop:idProbRvector} and the union
	bound argument, we have 
	\begin{align*}
	&\mathbb{P}\left\{ \exists x\in \mathscr{N}, \, 
			\exists I \subseteq [n] \mbox{ w. } |I| = n-\beta+1 \mbox{ s.t. }
			\left\|   A_{I,[n]}  x \right\|
		\le 
			\frac{ C_{\rm {Rgz}} }{ 4 }  
			\sqrt{  pn  }
			\psi_t \sqrt{ n}
		\right\}  \\
	\le &
		\exp(- \frac{ c_{ \ref{prop:idProbRvector} } }{ 2 } n ).  
	\end{align*}
	Now we condition on the complement of the above event, namely 
	\begin{align*}
		& \Omega(s,k,t) \\
	:= & \left\{ \forall x\in \mathscr{N}, \, 
			\forall I \subseteq [n] \mbox{ w. } |I| = n-\beta+1, \,  
			\left\|   A_{I,[n]}  x \right\|
		> 
			\frac{ C_{\rm {Rgz}} }{ 4 }  
			\sqrt{  pn  }
			\psi_t \sqrt{ n} \right\}.
	\end{align*}
	Suppose we condition on $\Omega_{\rm norm}$ (see definition from \eqref{def: omegaNorm}) and $\Omega(s,k,t)$. Consider $ x\in \cal R^s_{kt}$ and $I \subseteq [n]$ with $|I|=n-\beta+1$. First, from the definition of $\mathscr N$, there exists $y \in \mathscr{N}$ such that 
	$$
		\|x - y \| \le \frac{40}{R} \psi_i \sqrt{n}. 
	$$  
	Now, 
	\begin{align*}
		\|A_{I,[n]}x\| 	
	\ge &
		\|A_{I,[n]}y\| - \|A_{I,[n]}(x-y)\| \\
	\ge &
		\|A_{I,[n]}y\| - \big\|\big(A_{I,[n]} - \mathbb{E}A_{I,[n]}\big)(x-y)\big\|
		- \|\mathbb{E}A_{I,[n]}(x-y)\| \\
	\ge &
		\|A_{I,[n]}y\| - \big\|\big(A - \mathbb{E}A\big)(x-y)\big\|
		- \|\mathbb{E}A(x-y)\| \\ 
	\ge & 
		\frac{C_{\rm {Rgz}}}{4}\sqrt{  pn  } 
		\psi_t\sqrt{ n} 
		- 2\frac{40}{R}\psi_i\sqrt{ n} C_{\rm {norm}}
		\sqrt{  pn  } \\
	\ge& 
		\frac{C_{\rm {Rgz}}}{ 8 }
		\sqrt{  pn  } \psi_t\sqrt{ n },
	\end{align*}

	%that 
	%$ \left\| Ax \right\| \ge  
		%\frac{C_{\rm {Rgz}}}{ 2 \sqrt{10} }
		%\sqrt{  pn  }
			%\psi_i\sqrt{ n } $ 
	%for all $x\in \mathcal{R}_{kt}^s$ and 
	%$\Omega_{\rm norm}$.
	%For $x\in \mathcal{R}_{kt}^s$, 
	%$\exists y \in \mathscr{N}$ such that 
	%$|\|x-y\||<   \frac{40}{R}\psi_i \sqrt{ n }$.
	%Also, 
	%\begin{align*}
		%\left\| Ax \right\| &\ge 
		%\left\| Ay \right\| -\left\| A(x-y) \right\| \ge
		%\left\| Ay \right\| -\left\| (A-\mathbb{E}A)(x-y)
		%\right\| - \left\| \mathbb{E}A(x-y) \right\|\\ 
		%&\ge 
		%\frac{C_{\rm {Rgz}}}{2\sqrt{10}}\sqrt{  pn  } 
		%\psi_t\sqrt{ n} 
		%- 2\frac{40}{R}\psi_i\sqrt{ n} C_{\mbox{norm}}
		%\sqrt{  pn  } \ge \frac{C_{\rm {Rgz}}}{ 20 }
		%\sqrt{  pn  } \psi_t\sqrt{ n } . 
	%\end{align*}
	where the last inequality holds when $R$ is sufficiently 
	large. 

	Furthermore, since for any $ x \in \mathcal{R}$, 
	$ x \notin \mathcal{T}$. 
	By Proposition \ref{prop: TnormBound},
	$ \left\| x \right\|  \le n^{1+ o_{n }(1) }
	pn $.
	And by the definition of $\mathcal{R}_{kt}^s$, 
	$ \psi_t \sqrt{ n }
	\ge \sqrt{ \frac{ n }{ 2 }}$, we obtain
	\begin{align*}
		\frac{C_{\rm {Rgz}}}{ 2 \sqrt{10} }\sqrt{pn} 
			\psi_t \sqrt{ n }
	\ge
		n^{ - \frac{1}{2} -  o _ n (1) }
		( pn  )^{-1/2} \left\| x \right\|
	\end{align*} 
	and hence, 	
	\begin{align*}
	&	\mathbb{P} \Big\{  
			\exists x \in \mathcal{R}_{kt}^s \mbox{ and }
		\exists I \subseteq [n] \mbox{ w. } |I| = n-\beta+1 \mbox{ s.t. }\\
 	&  \phantom{AAA AAA AAA AAA AAA AAA AA}
			\left\|   A_{I,[n]}  x \right\|
			 \le  
		n^{ - \frac{1}{2} -  o _ n (1) }
		( pn  )^{-1/2} \left\| x \right\|
		\Big\} \\
	\le& 
		\Prob\{ \Omega^c(s,k,t) \cup \Omega_{\rm norm}^c\}.
	\end{align*}

	Applying the union bound argument over all possible triple $(s,k,t)$, ($ s \in [2]$, 
	$ n_s \le k \le \frac{ n  }{ \log^2( pn ) }$,
	$ 1\le t \le  m  \le  2\log( pn )$)
	to get 
	\begin{align*}
	& \mathbb{P} \Big\{  
			\exists x \in \mathcal{R} 
			\mbox{ and }
		\exists I \subseteq [n] \mbox{ w. } |I| = n-\beta+1 \mbox{ s.t. } \\
 	&  \phantom{AAA AAA AAA AAA AAA AAA AA}
			\left\|   A_{I,[n]}  x \right\|
		\le 
		n^{ - \frac{1}{2} -  o _ n (1) }
		( pn  )^{-1/2} \left\| x \right\|
		\Big\}  \\
	\le& 
		\Prob\Big\{ (\bigcup_{s,k,t}\Omega^c(s,k,t)) \cup \Omega_{\rm norm}^c \Big\}
	\le 
		n \exp\Big( - \frac{ c_{ \ref{prop:idProbRvector} } }{ 2 }n \Big)
		+ \Prob\{ \Omega_{\rm norm}^c \}. 
	\end{align*}
	Finally, together with 
	$ 
		\mathbb{P} \{ \Omega_{\rm norm}^c \} 
	= o(\mathbb{P} \{ \Omega_{\rm RC}^c \})$,
	the theorem follows. 
\end{proof}

\begin{proof} [ Proof of Theorem \ref{thm:T_2T_3-vector} ]
	The proof for Theorem \ref{thm:T_2T_3-vector} is essentially 
	the same as that of Theorem \ref{thm:R-vector}.
	Instead of using Lemma \ref{lem:R-net} and 
	Proposition \ref{prop:idProbRvector} for 
	the net and individual probability estimate, we 
	can replace it with Lemma \ref{lem:T-net} and 
	Corollary \ref{cor: indivProT2T3}. Here we will skip the proof. 

\end{proof}

\bibliographystyle{plain}

\bibliography{Ref}

\end{document}